\documentclass[a4paper,10pt]{article}%
\usepackage{amsmath}%
\usepackage{amsfonts}%
\usepackage{amssymb}%
\usepackage{amsthm}
\usepackage{graphicx}
\usepackage[latin1]{inputenc}
\usepackage[T1]{fontenc}
\usepackage[english,french]{babel}

\newtheorem*{theo}{Theorem}
\newtheorem{thm}{Theorem}
\newtheorem{theorem}{Theorem}[section]

\newtheorem{corollary}[theorem]{Corollary}

\newtheorem{definition}{Definition}[section]

\newtheorem{lemma}[theorem]{Lemma}

\newtheorem{proposition}[theorem]{Proposition}

\begin{document}

\sloppy

\title{Distortion elements for surface homeomorphisms}
\author{E. Militon}
\date{\today}
\maketitle

\setlength{\parskip}{10pt}

\selectlanguage{english}
\begin{abstract}
Let $S$ be a compact orientable surface and $f$ be an element of the group $\mathrm{Homeo}_{0}(S)$ of homeomorphisms of $S$ isotopic to the identity. Denote by $\tilde{f}$ a lift of $f$ to the universal cover $\tilde{S}$ of $S$. In this article, the following result is proved: if there exists a fundamental domain $D$ of the covering $\tilde{S} \rightarrow S$ such that
$$ \lim_{n \rightarrow + \infty} \frac{1}{n} d_{n} \mathrm{log}(d_{n})=0,$$
where $d_{n}$ is the diameter of $\tilde{f}^{n}(D)$, then the homeomorphism $f$ is a distortion element of the group $\mathrm{Homeo}_{0}(S)$.
\end{abstract}

\section{Introduction}

Given a compact manifold $M$, we denote by $\mathrm{Diff}_{0}^{r}(M)$ the identity component of the group of $C^{r}$-diffeomorphisms of $M$. A way to understand better this group is to try to describe the subgroups of this group. In other words, given a group $G$, does there exist an injective group morphism from the group $G$ to the group $\mathrm{Diff}_{0}^{r}(M)$? If, for this group $G$, we can answer affirmatively to this first question, one can try to describe the group morphisms from the group $G$ to the group $\mathrm{Diff}_{0}^{r}(M)$ as best as possible (ideally up to conjugacy but this is often an unattainable goal). The concept of distortion element, which we will define, allows to obtain rigidity results on group morphisms from $G$ to $\mathrm{Diff}_{0}^{r}(M)$ and will give us very partial answers to these questions.

Let us give now the definition of distortion elements. Remember that a group $G$ is \emph{finitely generated} if there exists a finite generating set $\mathcal{G}$: any element $g$ in this group is a product of elements of $\mathcal{G}$ and their inverses, $g=s_{1}^{\epsilon_{1}}s_{2}^{\epsilon_{2}} \ldots s_{n}^{\epsilon}$ where the $s_{i}$'s are elements of $\mathcal{G}$ and the $\epsilon_{i}$ are equal to $+1$ or $-1$. The minimal integer $n$ in such a decomposition is denoted by $l_{\mathcal{G}}(g)$. The map $l_{\mathcal{G}}$ is invariant under inverses and satisfies the triangle inequality $l_{\mathcal{G}}(gh) \leq l_{\mathcal{G}}(g)+l_{\mathcal{G}}(h)$. Therefore, for any element $g$ in the group $G$, the sequence $(l_{\mathcal{G}}(g^{n}))_{n \geq 0}$ is sub-additive, so the sequence $(\frac{l_{\mathcal{G}}(g^{n})}{n})_{n}$ converges. When the limit of this sequence is zero, the element $g$ is said to be \emph{distorted} or a \emph{distortion element} in the group $G$. Notice that this notion does not depend on the generating set $\mathcal{G}$. In other words, this concept is intrinsic to the group $G$. The notion extends to the case where the group $G$ is not finitely generated by saying that an element $g$ of the group $G$ is distorted if it belongs to a finitely generated subgroup of $G$ in which it is distorted. The main interest of the notion of distortion is the following rigidity property for groups morphisms: for a group morphism $\varphi : G \rightarrow H$, if an element $g$ is distorted in the group $G$, then its image under $\varphi$ is also distorted. In the case where the group $H$ does not contain distortion element other than the identity element in $H$ and where the group $G$ contains a distortion element different from the identity, such a group morphism cannot be an embedding: the group $G$ is not a subgroup of $H$.

Let us give now some simple examples of distortion elements. In any group $G$, the torsion elements, \emph{i.e.} those of finite order, are distorted. In free groups and free abelian groups, the only distorted element is the identity element. The simplest examples of groups which contain a distortion element which is not a torsion element are the Baumslag-Solitar groups which have the following presentation~: $BS(1,p)=<a,b \ | \ bab^{-1}=a^{p}>$. Then, for any integer $n$ : $b^{n}ab^{-n}=a^{p^{n}}$. Taking $\mathcal{G}=\left\{a,b\right\}$ as a generating set of this group, we have $l_{\mathcal{G}}(a^{p^{n}}) \leq 2n+1$ and the element $a$ is distorted in the group $BS(1,p)$ if $\left| p \right| \geq 2$. A group $G$ is said to be \emph{nilpotent} if the sequence of subgroups $(G_{n})_{n \geq 0}$ of $G$ defined by $G_{0}=G$ and $G_{n+1}=[G_{n},G]$ (this is the subgroup generated by elements of the form $[g_{n},g]=g_{n}gg_{n}^{-1}g^{-1}$, where $g_{n} \in G_{n}$ and $g \in G$) stabilizes and is equal to $\left\{1_{G} \right\}$ for a sufficiently large $n$. A typical example of a nilpotent group is the Heisenberg group which is the group of upper triangular matrices whose diagonal entries are $1$ and other entries are integers. In a nilpotent non-abelian group $N$, one can always find three distinct elements $a$, $b$ and $c$ different from the identity such that $[a,b]=c$ and the element $c$ commutes with $a$ and $b$. In this case, we have $c^{n^{2}}=[a^{n},b^{n}]$ so that, in the subgroup generated by $a$ and $b$ (and also in $N$), the element $c$ is distorted: $l_{\left\{a,b\right\}}(c^{n^{2}}) \leq 4n$. A general theorem by Lubotzky, Mozes and Raghunathan implies that there exist distortion elements (and even elements with a logarithmic growth) in some lattices of higher rank Lie groups, for instance in the group $\mathrm{SL}_{n}(\mathbb{Z})$ for $n \geq 3$. In the case of the group $\mathrm{SL}_{n}(\mathbb{Z})$, one can even find a generating set consisting of distortion elements. Notice that, in mapping class groups (see \cite{FLM}) and in the group of interval exchange transformations (see \cite{Nov}), any distorted element is a torsion element.

Let us consider now the case of diffeomorphisms groups. The following theorem has led to progress on Zimmer's conjecture. Let us denote by $S$ a compact surface without boundary endowed with a probability measure $area$ with full support. Finally, let us denote by $\mathrm{Diff}^{1}(S,area)$ the group of $C^{1}$-diffeomorphisms of the surface $S$ which preserve the measure $area$. Then we have the following statement:

\begin{theo}{(Polterovich \cite{Pol}, Franks-Handel \cite{FH})} \label{PFH}
If the genus of the surface $S$ is greater than one, any distortion element in the group $\mathrm{Diff}^{1}(S,area)$ is a torsion element.
\end{theo}

As nilpotent groups and $\mathrm{SL}_{n}(\mathbb{Z})$ have some non-torsion distortion elements, they are not subgroups of the group $\mathrm{Diff}^{1}(S,area)$. In the latter case, using a property of  almost simplicity of the group $\mathrm{SL}_{n}(\mathbb{Z})$, one can see even that a group morphism from the group $\mathrm{SL}_{n}(\mathbb{Z})$ to the group $\mathrm{Diff}^{1}(S,area)$ is "almost" trivial (its image is a finite group). Franks and Handel proved actually a more general result on distorsion elements in the case where the measure $area$ is any borelian probability measure which allows them to prove that this last statement is true for any measure $area$ with infinite support. They also obtain similar results in the cases of the torus and of the sphere. A natural question now is to wonder whether these theorems can be generalized in the case of more general diffeomorphisms or homeomorphisms groups (with no area-preservation hypothesis).

Unfortunately, one may find lots of distorted elements in those cases. The most striking example of this phenomenon is the following theorem by Calegari and Freedman:

\begin{theo}{(Calegari-Freedman \cite{CF})}
For an integer $d \geq 1$, every homeomorphism in the group $\mathrm{Homeo}_{0}(\mathbb{S}^{d})$ is distorted.
\end{theo}

In the case of a higher regularity, Avila proved in \cite{Avi} that any diffeomorphism in $\mathrm{Diff}_{0}^{\infty}(\mathrm{S}^{1})$ for which arbitrarily large iterates are arbitrarily close to the identity in the $C^{\infty}$-topology (such an element will be said to be \emph{recurrent}) is distorted in the group $\mathrm{Diff}_{0}^{\infty}(\mathrm{S}^{1})$: for instance, the irrational rotations are distorted. Using Avila's techniques and a local perfection result (due to Haller, Rybicki and Teichmann \cite{HT}), I obtained the following result (see \cite{Mil}):

\begin{thm}
For any compact manifold $M$ without boundary, any recurrent element in $\mathrm{Diff}_{0}^{\infty}(M)$ is distorted in this group.
\end{thm}

For instance, irrational rotations of the $2$ dimensional sphere or translations of the $d$-dimensional torus are distorted. More generally, there exist non-trivial distortion elements in the group of $C^{\infty}$-diffeomorphism of any manifold which admits a non-trivial $C^{\infty}$ circle action. Notice that, thanks to the Anosov-Katok method (see \cite{Her} and \cite{FaH}), we can build recurrent elements in the case of the sphere or of the $2$-dimensional torus which are not conjugate to a rotation. Anyway, we could not hope for a result analogous to the theorem by Polterovich and Franks and Handel, at least in the $C^{1}$ category, as we will see that the Baumslag-Solitar group $BS(1,2)$ embeds in the group $\mathrm{Diff}_{0}^{1}(M)$ for any manifold $M$ (this was indicated to me by Isabelle Liousse).

Identify the circle $\mathbb{S}^{1}$ with $\mathbb{R} \cup \left\{ \infty \right\}$. Then consider the (analytical) circle diffeomorphisms $a: x \mapsto x+1$ and $b:x \mapsto 2x$. The relation $bab^{-1}=a^{2}$ is satisfied and, therefore, the two elements $a$ and $b$ define an action of the group $BS(1,2)$ on the circle. By thickening the point at infinity (\emph{i.e.} by replacing the point at infinity with an interval), we have a compactly-supported action of our group on $\mathbb{R}$. This last action can be made $C^{1}$. Finally, by a radial action, we have a compactly-supported $C^{1}$ action of this Baumslag-Solitar group on $\mathbb{R}^{d}$ and, by identifying an open disc of a manifold with $\mathbb{R}^{d}$, we get an action of the Baumslag-Solitar group on any manifold. This gives some non-recurrent distortion elements in the group $\mathrm{Diff}_{0}^{1}(M)$ for any manifold $M$. In the case of  diffeomorphisms, it is difficult to approach a characterization of distortion element as there are many obstructions to being a distortion element (for instance, the differential cannot grow too fast along an orbit, the topological entropy of the diffeomorphism must vanish). On the contrary, in the groups of surface homeomorphisms, there is only one known obstruction to being a distortion element. We will describe it in the next section.

In this article, we will try to characterize geometrically the set of distortion elements in the group of homeomorphisms isotopic to the identity of a compact orientable surface. The theorem we obtain is a consequence of a result valid on any manifold and proved in the fourth section. This last result has a major drawback: it uses the fragmentation length which is not well understood except in the case of spheres. Thus, we will try to connect this fragmentation length to a more geometric quantity: the diameter of the image of a fundamental domain under a lift of the given homeomorphism. It is not difficult to prove that the fragmentation length dominates this last quantity: this will be treated in the third section of this article. However, conversely, it is more difficult to show that this last quantity dominates the fragmentation length. In order to prove this, we will make a distinction between the case of surfaces with boundary (Section 5), which is the easiest, the case of the torus (Section 6) and the case of higher genus closed manifolds (Section 7). The last section of this article gives examples of distortion elements in the group of homeomorphisms of the annulus for which the growth of the diameter of a fundamental domain is "fast".

\section{Notations and results}

Let $M$ be a manifold, possibly with boundary. We denote by $\mathrm{Homeo}_{0}(M)$ (respectively $\mathrm{Homeo}_{0}(M, \partial M)$) the identity component of the group of compactly-supported homeomorphisms of $M$ (respectively of the group of homeomorphisms of $M$ which pointwise fix a neighbourhood of the boundary $\partial M$ of $M$). Given two homeomorphisms $f$ and $g$ of $M$ and a subset $A$ of $M$, an \emph{isotopy} between $f$ and $g$ relative to $A$ is a continuous path of homeomorphisms $(f_{t})_{t \in [0,1]}$ which pointwise fix $A$ such that $f_{0}=f$ and $f_{1}=g$. If $A$ is the empty set, it is called an isotopy between $f$ and $g$.

In what follows, $S$ is a compact orientable surface, possibly with boundary, different from the disc and from the sphere. We denote by $\Pi : \tilde{S} \rightarrow S$ the universal cover of $S$. The surface $\tilde{S}$ is seen as a subset of the euclidean plane $\mathbb{R}^{2}$ or of the hyperbolic plane $\mathbb{H}^{2}$ so that deck transformations are isometries for the euclidean metric or the hyperbolic metric. We endow the surface $\tilde{S}$ with this metric. In what follows, we identify the fundamental group $\Pi_{1}(S)$ of the surface $S$ with the group of deck transformations of the covering $\Pi : \tilde{S} \rightarrow S$. If $A$ is a subset of the hyperbolic plane $\mathbb{H}^{2}$ (respectively of the euclidean plane $\mathbb{R}^{2}$), we denote by $\delta(A)$ the diameter of $A$ for the hyperbolic distance (respectively the euclidean distance).

For a homeomorphism $f$ of $S$, a \emph{lift} of $f$ is a homeomorphism $F$ of $\tilde{S}$ which satisfies $\Pi \circ F=f \circ \Pi$. For an isotopy $(f_{t})_{t \in [0,1]}$, a lift of $(f_{t})_{t \in [0,1]}$ is a continuous path $(F_{t})_{t \in [0,1]}$ of homeomorphisms of $\tilde{S}$ such that, for any $t$, the homeomorphism $F_{t}$ is a lift of the homeomorphism $f_{t}$. For a homeomorphism $f$ in $\mathrm{Homeo}_{0}(S)$, we denote by $\tilde{f}$ a lift of $f$ obtained as the time $1$ of a lift of an isotopy between the identity and $f$ which is equal to the identity for $t=0$. If moreover the boundary of $S$ is non-empty and the homeomorphism $f$ is in $\mathrm{Homeo}_{0}(S, \partial S)$, the homeomorphism $\tilde{f}$ is obtained by lifting an isotopy relative to the boundary $\partial S$. If there exists a disc $\mathbb{D}^{2}$ embedded in the surface $S$ which contains the support of the homeomorphism $f$, we require moreover that the support of $\tilde{f}$ is contained in $\Pi^{-1}(\mathbb{D}^{2})$.
Notice that the homeomorphism $\tilde{f}$ is unique except in the cases of the groups $\mathrm{Homeo}_{0}(\mathbb{T}^{2})$ and $\mathrm{Homeo}_{0}([0,1]\times \mathbb{S}^{1})$. This last claim is a consequence of a theorem by Hamstrom (see \cite{Ham}): if $S$ is a surface without boundary of genus greater than $1$, then the topological space $\mathrm{Homeo}_{0}(S)$ is simply connected. Moreover, if $S$ is a surface with a nonempty boundary, the topological space $\mathrm{Homeo}_{0}(S, \partial S)$ is simply connected. Finally, let us prove the claim in the case of an element $f$ in $\mathrm{Homeo}_{0}(S)$. Take two lifts $F_{1}$ and $F_{2}$ of $f$ to $\tilde{S}$ as above. Then the double of $F_{1}$, which is the homeomorphism on the double of $\tilde{S}$ canonically defined by $F_{1}$, is equal to the double of $F_{2}$ by Hamstrom's theorem. This proves the claim.

Notice that, for any deck transformation $\gamma \in \Pi_{1}(S)$, and any homeomorphism $f$ in $\mathrm{Homeo}_{0}(S)$, $\gamma \circ \tilde{f}= \tilde{f} \circ \gamma$. This is true when the surface $S$ is the torus or the closed annulus and, in the other cases, the map $\gamma \circ \tilde{f}$ as well as the map $\tilde{f} \circ \gamma$ is the time $1$ of a lift of an isotopy between the identity and $f$ which is equal to $\gamma$ for $t=0$.

\begin{definition}
We call \emph{fundamental domain} of $S$ any compact connected subset $D$ of $\tilde{S}$ which satisfies the following properties:
\begin{enumerate}
\item $\Pi(D)=S$.
\item For any deck transformation $\gamma$ in $\Pi_{1}(S)$ different from the identity, the interior of $D$ is disjoint from the interior of $\gamma(D)$. 
\end{enumerate}
\end{definition}

The main theorem of this article is a partial converse to the following property (observed by Franks and Handel in \cite{FH}, Lemma 6.1):

\begin{proposition} \label{distorsionimpliquecroissance}
Denote by $D$ a fundamental domain of $\tilde{S}$ for the action of $\Pi_{1}(S)$.\\
If a homeomorphism $f$ in $\mathrm{Homeo}_{0}(S)$ (respectively in $\mathrm{Homeo}_{0}(S, \partial S)$) is a distortion element of $\mathrm{Homeo}_{0}(S)$ (respectively of $\mathrm{Homeo}_{0}(S, \partial S)$), then:
$$ \lim_{n \rightarrow + \infty} \frac{\delta(\tilde{f}^{n}(D))}{n}=0.$$
\end{proposition}

\noindent \textbf{Remark} In the case where the surface considered is the torus $\mathbb{T}^{2}$ or the annulus $[0,1]\times \mathbb{S}^{1}$, the conclusion of this proposition is equivalent to saying that the rotation set of $f$ is reduced to a single point (see \cite{MZ} for a definition of the rotation set of a homeomorphism of the torus isotopic to the identity; the definition is analogous in the case of the annulus).

\begin{proof}
Let $f$ be a distortion element in $\mathrm{Homeo}_{0}(S)$ (respectively in $\mathrm{Homeo}_{0}(S, \partial S)$).
Denote by $\mathcal{G}=\left\{g_{1}, g_{2}, \ldots , g_{p} \right\}$ a finite subset of $\mathrm{Homeo}_{0}(S)$ (respectively of $\mathrm{Homeo}_{0}(S, \partial S)$) such that:
\begin{enumerate}
\item The homeomorphism $f$ belongs to the group generated by $\mathcal{G}$.
\item The sequence $(\frac{l_{\mathcal{G}}(f^{n})}{n})_{n \geq 1}$ converges to $0$.
\end{enumerate}
Then we have a decomposition of the form:
$$f^{n}=g_{i_{1}} \circ g_{i_{2}} \circ \ldots \circ g_{i_{l_{n}}}$$
where $l_{n}=l_{\mathcal{G}}(f^{n})$. This implies the following equality:
$$I \circ \tilde{f}^{n}=\tilde{g}_{i_{1}} \circ \tilde{g}_{i_{2}} \circ \ldots \circ \tilde{g}_{i_{l_{n}}}$$
where $I$ is an isometry of $\tilde{S}$. Let us take $\mu= \max_{1 \leq i \leq p, \ x \in \tilde{S}} d(x,\tilde{g}_{i}(x))$. As, for any index $i$ and any deck transformation $\gamma$ in $\Pi_{1}(S)$, $\gamma \circ \tilde{g}_{i}= \tilde{g}_{i} \circ \gamma$ and as the distance $d$ is invariant under deck transformations, $\mu$ is finite. Then, for any two points $x$ and $y$ of the fundamental domain $D$, we have:
$$\begin{array}{rcl}
d(\tilde{f}^{n}(x),\tilde{f}^{n}(y)) & = & d(I \circ \tilde{f}^{n}(x),I \circ \tilde{f}^{n}(y)) \\
 & \leq & d(I \circ \tilde{f}^{n}(x),x)+d(x,y)+d( I \circ \tilde{f}^{n}(y),y) \\
 & \leq & l_{n}\mu+ \delta(D)+l_{n}\mu
\end{array}
$$
which implies the proposition, by sublinearity of the sequence $(l_{n})_{n \geq 0}$.
\end{proof}

The main theorem of this article is the following:

\begin{theorem} \label{croissanceimpliquedistorsion}
Let $f$ be a homeomorphism in $\mathrm{Homeo}_{0}(S)$ (respectively in $\mathrm{Homeo}_{0}(S, \partial S)$). If:
$$\lim_{n \rightarrow + \infty} \frac{\delta(\tilde{f}^{n}(D))log(\delta(\tilde{f}^{n}(D)))}{n}=0,$$
then $f$ is a distortion element in $\mathrm{Homeo}_{0}(S)$ (respectively in $\mathrm{Homeo}_{0}(S, \partial S)$).
\end{theorem}

\noindent \textbf{Remark} The property $\lim_{n \rightarrow + \infty} \frac{\delta(\tilde{f}^{n}(D))log(\delta(\tilde{f}^{n}(D)))}{n}=0$ is independent of the chosen fundamental domain $D$, as we will see in the next section. Thus, it is invariant under conjugation.

The proof of this theorem occupies the next five sections. Let us now introduce a new notion in order to complete this proof.

Let $M$ be a compact $d$-dimensional manifold. We will call closed ball of $M$ the image of the closed unit ball under an embedding from $\mathbb{R}^{d}$ to the manifold $M$. Let: 
$$H^{d}=\left\{ (x_{1},x_{2}, \ldots, x_{d}) \in \mathbb{R}^{N}, x_{1} \geq 0 \right\}.$$
We will call closed half-ball of $M$ the image of $B(0,1)\cap H^{d}$ under an embedding $p:H^{d} \rightarrow M$ such that:
$$p(\partial H^{d})=p(H^{d}) \cap \partial M.$$
 Let us fix a finite family $\mathcal{U}$ of closed balls or closed half-balls whose interiors cover $M$. Then, by the fragmentation lemma (see \cite{Fis} or \cite{Bou}), there exists a finite family $(f_{i})_{1 \leq i \leq n}$ of homeomorphisms in $\mathrm{Homeo}_{0}(M)$, each supported in one of the sets of $\mathcal{U}$, such that:
$$f=f_{1} \circ f_{2} \circ \ldots \circ f_{n}.$$
We denote by $\mathrm{Frag}_{\mathcal{U}}(f)$ the minimal integer $n$ in such a decomposition: it is the minimal number of factors necessary to write $f$ as a product (\emph{i.e.} composition) of homeomorphisms supported each in one of the balls of $\mathcal{U}$.

Let us come back to the case of a compact surface $S$ and denote by $\mathcal{U}$ a finite family of closed discs or of closed half-discs whose interiors cover $S$. Denote by $D$ a fundamental domain of $\tilde{S}$ for the action of $\Pi_{1}(S)$. We now describe the two steps of the proof of Theorem \ref{croissanceimpliquedistorsion}. The first step of the proof consists of checking that the quantity $\mathrm{Frag}_{\mathcal{U}}(f)$ is almost equal to $\delta(\tilde{f}(D))$:

\begin{theorem} \label{quasiisom}
There exist two real constants $C>0$ and $C'$ such that, for any homeomorphism $g$ in $\mathrm{Homeo}_{0}(S)$:
$$ \frac{1}{C}\delta(\tilde{g}(D))-C' \leq \mathrm{Frag}_{\mathcal{U}}(g) \leq C\delta(\tilde{g}(D))+C'.$$
\end{theorem}

In the case when the boundary of the surface $S$ is nonempty, let us denote by $S'$ a submanifold of $S$ homeomorphic to $S$, contained in the interior of $S$ and which is a deformation retract of $S$. We denote by $\mathcal{U}$ a family of closed balls of $S$ whose union of interiors cover $S'$.

\begin{theorem} \label{quasiisom2}
There exist two real constants $C>0$ and $C'$ such that, for any homeomorphism $g$ in $\mathrm{Homeo}_{0}(S, \partial S)$ supported in $S'$:
$$ \frac{1}{C}\delta(\tilde{g}(D))-C' \leq \mathrm{Frag}_{\mathcal{U}}(g) \leq C\delta(\tilde{g}(D))+C'.$$
\end{theorem}

The lower bound of the fragmentation length is not difficult: it is treated in the next section in which we will also see that the quantity $\mathrm{Frag}_{\mathcal{U}}$ is essentially independent from the cover $\mathcal{U}$ chosen. On the other hand, it is much more technical to establish the upper bound. In the proof of this bound, we distinguish three cases: the case of surfaces with boundary (Section 5), the case of the torus (Section 6) and the case of higher genus compact surfaces without boundary (Section 7). The proof seems to depend strongly on the fundamental group of the surface considered. In particular, it is easier in the case of surfaces with boundary whose fundamental groups are free. In the case of the torus, the proof is a little tricky and, in the case of higher genus closed surfaces, the proof is more complex and uses Dehn algorithm for small-cancellation groups (surface groups in this case).

Let us explain now the second step of the proof. Denote by $M$ a compact manifold and $\mathcal{U}$ a finite family of closed balls or half-balls whose interiors cover $M$. In Section 4, we will prove the following theorem which asserts that,  for a homeomorphism $f$ in $\mathrm{Homeo}_{0}(M)$, if the sequence $\mathrm{Frag}_{\mathcal{U}}(f^{n})$ does not grow too fast with $n$, then the homeomorphism $f$ is a distortion element:

\begin{theorem} \label{fragdistth}
If $$\lim_{n \rightarrow + \infty} \frac{\mathrm{Frag}_{\mathcal{U}}(f^{n}).log(\mathrm{Frag}_{\mathcal{U}}(f^{n}))}{n}=0,$$
then the homeomorphism $f$ is a distortion element in $\mathrm{Homeo}_{0}(M)$.
\end{theorem}

Moreover, in the case of a manifold $M$ with boundary, if $\mathcal{U}$ is a finite family of closed balls contained in the interior of $M$ whose interiors cover the support of a homeomorphism $f$ in $\mathrm{Homeo}_{0}(M, \partial M)$, this last theorem remains true in the group $\mathrm{Homeo}_{0}(M, \partial M)$. This section uses a technique due to Avila (see \cite{Avi}).

Theorem \ref{croissanceimpliquedistorsion} is clearly a consequence of these two theorems.

The last section is dedicated to the proof of the following theorem which proves that Proposition \ref{distorsionimpliquecroissance} is optimal:

\begin{theorem} \label{exemple}
Let $(v_{n})_{n \geq 1}$ be a sequence of positive real numbers such that:
$$\lim_{n \rightarrow + \infty} \frac{v_{n}}{n}=0.$$

Then there exists a homeomorphism $f$ in $\mathrm{Homeo}_{0}(\mathbb{R}/\mathbb{Z} \times [0,1] , \mathbb{R}/\mathbb{Z} \times \left\{ 0,1 \right\} )$ such that:
\begin{enumerate}
\item $ \forall n \geq 1, \ \delta(\tilde{f}^{n}([0,1] \times [0,1])) \geq v_{n}$.
\item The homeomorphism $f$ is a distortion element in $\mathrm{Homeo}_{0}(\mathbb{R}/\mathbb{Z} \times [0,1] , \mathbb{R}/\mathbb{Z} \times \left\{ 0,1 \right\} )$.
\end{enumerate}
\end{theorem}

This theorem means that being a distortion element gives no clues on the growth of the diameter of a fundamental domain other than the sublinearity of this growth. This theorem remains true for any surface $S$: it suffices to embed the annulus $ \mathbb{R}/\mathbb{Z} \times [0,1]$ in any surface $S$ and to use this last theorem to see it.

\section{Quasi-isometries}

In this section, we prove the lower bound in Theorems \ref{quasiisom} and \ref{quasiisom2}. More precisely, we prove these theorems using the following propositions whose proof will be discussed in Sections 5, 6 and 7.

\begin{proposition} \label{diamfrag}
There exists a finite cover $\mathcal{U}$ of $S$ by closed discs and half-discs as well as real constants $C \geq 1$ and $\ C' \geq 0$ such that, for any homeomorphism $g$ in $\mathrm{Homeo}_{0}(S)$:
$$\mathrm{Frag}_{\mathcal{U}}(g) \leq  C \mathrm{diam}_{\mathcal{D}}(\tilde{g}(D_{0}))+C'.$$
\end{proposition}

Here is a version of the previous proposition in the case of the group $\mathrm{Homeo}_{0}(S, \partial S)$.

\begin{proposition} \label{diamfrag2}
Fix a subsurface with boundary $S'$ of $S$ which is contained in the interior of $S$, is a deformation retract of $S$ and is homeomorphic to $S$. There exists a finite cover $\mathcal{U}$ of $S'$ by closed discs contained in the interior of $S$ as well as real constants $C \geq 1$ and $\ C' \geq 0$ such that, for any homeomorphism $g$ in $\mathrm{Homeo}_{0}(S)$ supported in $S'$:
$$\mathrm{Frag}_{\mathcal{U}}(g) \leq  C \mathrm{diam}_{\mathcal{D}}(\tilde{g}(D_{0}))+C'.$$
\end{proposition}

In order to prove these theorems, we need some notation. As in the last section, let us denote by $S$ a compact surface.
Two maps $a,b:\mathrm{Homeo}_{0}(S) \rightarrow \mathbb{R}$ are \emph{quasi-isometric} if and only if there exist real constants $C \geq 1$ and $C' \geq 0$ such that:
$$ \forall f \in \mathrm{Homeo}_{0}(S), \ \frac{1}{C}a(f)-C' \leq b(f) \leq Ca(f)+C'.$$
More generally, an arbitrary number of maps $\mathrm{Homeo}_{0}(S) \rightarrow \mathbb{R}$ are said to be quasi-isometric if they are pairwise quasi-isometric.

Let us consider a fundamental domain $D_{0}$ of $\tilde{S}$ for the action of the group $\Pi_{1}(S)$ which satisfies the following properties (see figure \ref{ledomaineD0}) :
\begin{enumerate}
\item If the surface $S$ is closed of genus $g$, the fundamental domain $D_{0}$ is a closed disc bounded by a $4g$-gone with geodesic edges.
\item If the surface $S$ has a nonempty boundary, the fundamental domain $D_{0}$ is a closed disc bounded by a polygon with geodesic edges such that any edge of this polygon which is not contained in $\partial \tilde{S}$ connects two edges contained in $\partial \tilde{S}$.
\end{enumerate}

\begin{figure}[ht]
\begin{center}
\includegraphics{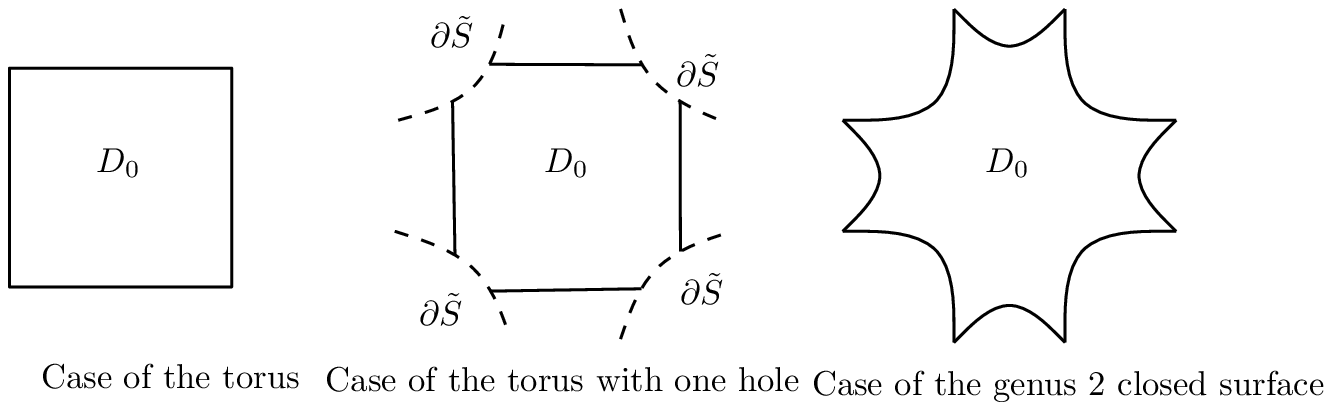}
\end{center}
\caption{The fundamental domain $D_{0}$}
\label{ledomaineD0}
\end{figure}

Let us take:
$$ \mathcal{D}= \left\{ \gamma(D_{0}), \ \gamma \in \Pi_{1}(S) \right\}.$$
For fundamental domains $D$ and $D'$ in $\mathcal{D}$, we denote by $d_{\mathcal{D}}(D,D')+1$ the minimal number of fundamental domains met by a path which connects the interior of $D$ to the interior of $D'$. The map $d_{\mathcal{D}}$ is a distance on $\mathcal{D}$. We now give an algebraic definition of this quantity. Denote by $\mathcal{G}$ the finite set of deck transformations in $\Pi_{1}(S)$ which send $D_{0}$ to a polygon in $\mathcal{D}$ adjacent to $D_{0}$, \emph{i.e.} which shares an edge in common with $D_{0}$. Then the subset $\mathcal{G}$ is symmetric and is a generating set of $\Pi_{1}(S)$. Notice that the map
$$\begin{array}{rrcl}
d_{\mathcal{G}}: & \Pi_{1}(S) \times \Pi_{1}(S) & \rightarrow & \mathbb{R} \\
& (\varphi, \psi) & \mapsto & l_{\mathcal{G}}(\varphi^{-1} \psi)
\end{array}
$$
is a distance on the group $\Pi_{1}(S)$. Then, for any pair $(\varphi, \psi)$ of deck transformations in the group $\Pi_{1}(S)$, we have:
$$ l_{\mathcal{G}}(\varphi^{-1} \psi)=d_{\mathcal{D}}(\varphi(D_{0}), \psi(D_{0})).$$
One can see it by noticing that $d_{\mathcal{D}}$ is invariant under the action of the group $\Pi_{1}(S)$ and by proving by induction on $l_{\mathcal{G}}(\psi)$ that:
$$l_{\mathcal{G}}(\psi)=d_{\mathcal{D}}(D_{0},\psi(D_{0})).$$

Given a compact subset $A$ of $\tilde{S}$, we call \emph{discrete diameter} of $A$ the following quantity:
$$ \mathrm{diam}_{\mathcal{D}}(A)=\max \left\{ d_{\mathcal{D}}(D,D'), \left\{
\begin{array}{l}
D \in \mathcal{D}, \ D' \in \mathcal{D} \\
D \cap A \neq \emptyset, \ D' \cap A \neq \emptyset
\end{array}
\right.
\right\}
.$$
For a fundamental domain $D_{1}$ in $\mathcal{D}$, we call \emph{éloignement} of $A$ with respect to $D_{1}$ the following quantity:
$$ \mathrm{el}_{D_{1}}(A)= \max \left\{ d_{\mathcal{D}}(D_{1},D), \left\{
\begin{array}{l}
D \in \mathcal{D} \\
D \cap A \neq \emptyset
\end{array}
\right.
\right\}
.$$

Notice that, in the case where $D_{1} \cap A \neq \emptyset$, we have:
$$\mathrm{el}_{D_{1}}(A) \leq \mathrm{diam}_{\mathcal{D}}(A) \leq 2 \mathrm{el}_{D_{1}}(A).$$

In this section, we prove the following statement, using Proposition \ref{diamfrag}:

\begin{proposition} \label{equiv}
Given two finite sets $\mathcal{U}$ and $\mathcal{U}'$ of closed balls or half-balls whose interiors cover the surface $S$, the maps $\mathrm{Frag}_{\mathcal{U}}$ and $\mathrm{Frag}_{\mathcal{U}'}$ are quasi-isometric.\\
Given two fundamental domains $D$ and $D'$ of $\tilde{S}$ for the action of the fundamental group of $S$, the maps $f \mapsto \delta(\tilde{f}(D))$ and $g \mapsto \delta(\tilde{g}(D'))$ are quasi-isometric.\\
Fix now a finite cover $\mathcal{U}$ of $S$ as above and a fundamental domain $D$. Then the following maps $\mathrm{Homeo}_{0}(S) \rightarrow \mathbb{R}$ are quasi-isometric:
\begin{enumerate}
\item The map $\mathrm{Frag}_{\mathcal{U}}$.
\item The map $g \mapsto \delta(\tilde{g}(D))$.
\item The map $g \mapsto \mathrm{diam}_{\mathcal{D}}(\tilde{g}(D_{0}))$.
\end{enumerate}
\end{proposition}

When the boundary of the surface $S$ is nonempty, we have an analogous proposition in the case of the group $\mathrm{Homeo}_{0}(S, \partial S)$. As in the last section, let us denote by $S'$ a submanifold with boundary of $S$ homeomorphic to $S$, contained in the interior of $S$, and which is a deformation retract of $S$, and by $\mathcal{U}$ a finite family of closed balls contained in the interior of $S$ whose union of the interiors contains $S'$. Finally, let us denote by $G_{S'}$ the group of homeomorphisms in $\mathrm{Homeo}_{0}(S, \partial S)$ which are supported in $S'$.

\begin{proposition} \label{equiv2}
Let us fix a fundamental domain $D$ of $\tilde{S}$ for the action of the fundamental group of $S$. The following maps $G_{S'} \rightarrow \mathbb{R}$ are quasi-isometric:
\begin{enumerate}
\item The map $\mathrm{Frag}_{\mathcal{U}}$.
\item The map $g \mapsto \delta(\tilde{g}(D))$.
\item The map $g \mapsto \mathrm{diam}_{\mathcal{D}}(\tilde{g}(D_{0}))$.
\end{enumerate}
\end{proposition}

The proof of this proposition is quite the same as the proof of the previous one: that is why we will not provide it.

These two propositions directly imply Theorems \ref{quasiisom} and \ref{quasiisom2}.

\begin{proof}
Let us prove first that, for any two fundamental domains $D$ and $D'$, the maps $g \mapsto \delta(\tilde{g}(D))$ and $g \mapsto \delta(\tilde{g}(D'))$ are quasi-isometric.
Let us take:
$$\left\{\gamma_{1}, \gamma_{2}, \ldots, \gamma_{p} \right\}= \left\{ \gamma \in \Pi_{1}(S), \ D' \cap \gamma(D) \neq \emptyset \right\}.$$
Notice that:
$$D' \subset \bigcup_{i=1}^{p} \gamma_{i}(D)$$
and the right-hand side is path-connected. Then:
$$\tilde{g}(D') \subset \bigcup_{i=1}^{p} \tilde{g}(\gamma_{i}(D)).$$
Then the lemma below implies that:
$$\delta(\tilde{g}(D')) \leq p\delta(\tilde{g}(D)).$$
As the fundamental domains $D$ and $D'$ play symmetric roles, this implies that the maps $g \mapsto \delta(\tilde{g}(D))$ and $g \mapsto \delta(\tilde{g}(D'))$ are quasi-isometric.

\begin{lemma}
Let $X$ be a path-connected metric space. Let $(A_{i})_{1 \leq i \leq p}$ be a family of closed subsets of $X$ such that:
$$X= \bigcup_{i=1}^{p} A_{i}.$$
Then:
$$\delta(X)= \sup_{x \in X, y \in X}d(x,y) \leq p \max_{1 \leq i \leq p} \delta(A_{i}).$$
\end{lemma}

\begin{proof}
Let $x$ and $y$ be two points in $X$. By path-connectedness of $X$, there exists an integer $k$ between $1$ and $p$, an injection $\sigma : [1,k] \cap \mathbb{Z} \rightarrow [1,p] \cap \mathbb{Z}$ and a sequence $(x_{i})_{1 \leq i \leq k+1}$ of points in $X$ which satisfy the following properties:
\begin{enumerate}
\item $x_{1}=x$ and $x_{k+1}=y$.
\item For any index $i$ between $1$ and $k$, the points $x_{i}$ and $x_{i+1}$ both belong to $A_{\sigma(i)}$.
\end{enumerate}
Then:
$$ \begin{array}{rcl}
d(x,y) & \leq & \sum \limits _{i=1}^{k}d(x_{i},x_{i+1}) \\
 & \leq & \sum \limits _{i=1}^{k} \delta(A_{\sigma(i)}) \\
 & \leq & p \max \limits _{1 \leq i \leq p} \delta(A_{i}).
\end{array}
$$
This last inequality implies the lemma.
\end{proof}

Let us show now that, for two finite families $\mathcal{U}$ and $\mathcal{U}'$  as in the statement of Proposition \ref{equiv}, the maps $\mathrm{Frag}_{\mathcal{U}}$ and $\mathrm{Frag}_{\mathcal{U}'}$ are quasi-isometric. The proof of this fact requires the following lemmas.

\begin{lemma}\label{décomp}
Let $\epsilon>0$. Let us denote by $B$ the unit closed ball of $\mathbb{R}^{d}$. There exists an integer $N \geq 0$ such that any homeomorphism in $\mathrm{Homeo}_{0}(B, \partial B)$ can be written as a composition of at most $N$ homeomorphisms in $\mathrm{Homeo}_{0}(B, \partial B)$ $\epsilon$-close to the identity (for a distance which defines the $C^{0}$-topology on this group).
\end{lemma}

\begin{lemma}\label{fragmentation}
Let $M$ be a compact manifold and $\left\{U_{1}, U_{2}, \ldots , U_{p} \right\}$ be an open cover of $M$. There exist $\epsilon>0$ and an integer $N'>0$ such that, for any homeomorphism $g$ in $\mathrm{Homeo}_{0}(M)$ (respectively in $\mathrm{Homeo}_{0}(M, \partial M)$) $\epsilon$-close to the identity, there exist homeomorphisms $g_{1}, \ldots, g_{N'}$ in $\mathrm{Homeo}_{0}(M)$ (respectively in $\mathrm{Homeo}_{0}(M, \partial M)$) such that:
\begin{enumerate}
\item Each homeomorphism $g_{i}$ is supported in one of the $U_{j}$'s.
\item $g=g_{1} \circ g_{2} \circ \ldots \circ g_{N'}$.
\end{enumerate}
\end{lemma}

Lemma \ref{décomp} is a consequence of Lemma 5.2 in \cite{BCLP} (notice that the proof works in dimensions higher than $2$). Lemma \ref{fragmentation} is classical. It is a consequence of the proof of Theorem 1.2.3 in \cite{Bou}. These two lemmas imply that, for an open cover of the disc $\mathbb{D}^{2}$, there exists an integer $N$ such that any homeomorphism in $\mathrm{Homeo}_{0}(\mathbb{D}^{2}, \partial \mathbb{D}^{2})$ can be written as a composition of at most $N$ homeomorphisms supported each in one of the open sets of the covering. Now, for an element $U$ in $\mathcal{U}$, we denote by $U \cap \mathcal{U}'$ the cover of $U$ given by the intersections of the elements of $\mathcal{U}'$ with $U$. The application of this last result to the ball $U$ with the cover $U \cap \mathcal{U}'$ gives us a constant $N_{U}$. Let us denote by $\mathbf{N}$ the maximum of the $N_{U}$, where $U$ varies over $\mathcal{U}$. We directly obtain  that, for any homeomorphism $g$:
$$\mathrm{Frag}_{\mathcal{U}'}(g) \leq \mathbf{N} \mathrm{Frag}_{\mathcal{U}}(g).$$
As the two covers $\mathcal{U}$ and $\mathcal{U}'$ play symmetric roles, the fact is proved. Notice that this fact is true in any dimension.

Using a quasi-isometry between the metric spaces $(\Pi_{1}(S),d_{S})$ and $\tilde{S}$, we will prove the following lemma which implies that the last two maps in the proposition are quasi-isometric:

\begin{lemma}
There exist constants $C \geq 1$ and $C' \geq 0$ such that, for any compact subset $A$ of $\tilde{S}$:
$$\frac{1}{C} \delta(A)-C' \leq \mathrm{diam}_{\mathcal{D}}(A) \leq C \delta (A) +C'.$$
\end{lemma}

\begin{proof} Let us fix a point $x_{0}$ in the interior of $D_{0}$. The map:
$$\begin{array}{rrcl}
q: & \Pi_{1}(S) & \rightarrow & \tilde{S} \\
 & \gamma & \mapsto & \gamma(x_{0})
\end{array}
$$
is a quasi-isometry for the distance $d_{\mathcal{G}}$ and the distance on $\tilde{S}$ (this is the \v{S}varc-Milnor lemma, see \cite{Har} p.87). We notice that, for a compact subset $A$ of $\tilde{S}$, the number $\mathrm{diam}_{\mathcal{D}}(A)$ is equal to the diameter of $q^{-1}(B)$ for the distance $d_{\mathcal{G}}$, where 
$$B=\bigcup \left\{D, \ D \in \mathcal{D} \ D \cap A \neq \emptyset \right\}.$$
We deduce that there exist constants $C_{1} \geq 1$ and $C_{1}'\geq 0$ independent from $A$ such that:
$$ \frac{1}{C_{1}} \delta(B)-C_{1}' \leq \mathrm{diam}_{\mathcal{D}}(A) \leq C_{1} \delta(B)+C_{1}'.$$
The inequalities
$$ \delta(B) - 2 \delta(D_{0}) \leq \delta(A) \leq \delta(B),$$
complete the proof of the lemma.
\end{proof}

We now prove that, for any cover $\mathcal{U}$ as in the statement of Proposition \ref{equiv}, there exist constants $C \geq 1$ and $C' \geq 0$ such that, for any homeomorphism $g$ in $\mathrm{Homeo}_{0}(S)$:
$$ \frac{1}{C}\mathrm{diam}_{\mathcal{D}}(\tilde{g}(D_{0}))-C' \leq \mathrm{Frag}_{\mathcal{U}}(g).$$
Let us fix such a family $\mathcal{U}$. We will need the following lemma that we will prove later:

\begin{lemma} \label{min}
There exists a constant $C >0$ such that, for any compact subset $A$ of $\tilde{S}$ and any homeomorphism $g$ supported in one of the sets in $\mathcal{U}$:
$$ \mathrm{diam}_{\mathcal{D}}(\tilde{g}(A)) \geq \mathrm{diam}_{\mathcal{D}}(A)-C.$$
\end{lemma}

Take $k=\mathrm{Frag}_{\mathcal{U}}(g)$ and:
$$ g= g_{1} \circ g_{2} \circ \ldots \circ g_{k},$$
where each homeomorphism $g_{i}$ is supported in one of the elements of $\mathcal{U}$. Then:
$$I \circ \tilde{g}= \tilde{g}_{1} \circ \tilde{g}_{2} \circ \ldots \circ \tilde{g}_{k},$$
where $I$ is a deck transformation (and an isometry).
Lemma \ref{min} combined with an induction implies that:
$$ \forall j \in [1, k] \cap \mathbb{Z}, \  \mathrm{diam}_{\mathcal{D}}(\tilde{g}_{j}^{-1} \circ \ldots \circ \tilde{g}_{1}^{-1} \circ \tilde{g}(D_{0})) \geq \mathrm{diam}_{\mathcal{D}}(\tilde{g}(D_{0})) - jC,$$
as the homeomorphisms $\tilde{g}_{i}$ commute with $I$.
Hence:
$$ 2=\mathrm{diam}_{\mathcal{D}}(\tilde{g}_{k}^{-1} \circ \ldots \circ \tilde{g}_{1}^{-1} \circ \tilde{g}(D_{0})) \geq \mathrm{diam}_{\mathcal{D}}(\tilde{g}(D_{0})) - kC.$$
Therefore:
$$\mathrm{Frag}_{\mathcal{U}}(g) \geq \frac{1}{C} \mathrm{diam}_{\mathcal{D}}(\tilde{g}(D_{0}))- \frac{2}{C}.$$
We obtain the lower bound wanted.

\begin{proof}[Proof of Lemma \ref{min}]
For an element $U$ in $\mathcal{U}$, we denote by $\tilde{U}$ a lift of $U$, \emph{i.e.} a connected component of $\Pi^{-1}(U)$. Let us take:
$$ \mu(U)= \mathrm{diam}_{\mathcal{D}}(\tilde{U}).$$
This quantity does not depend on the lift $\tilde{U}$ chosen. We denote by $\mu$ the maximum of the $\mu(U)$, for $U$ in $\mathcal{U}$.

We denote by $U_{g}$ an element in $\mathcal{U}$ which contains the support of $g$. Let us consider two fundamental domains $D$ and $D'$ which meet $A$ and which satisfy the following relation:
$$ d_{\mathcal{D}}(D,D')=\mathrm{diam}_{\mathcal{D}}(A).$$
Let us take a point $x$ in $D \cap A$ and a point $x'$ in $D' \cap A$. If the point $x$ belongs to $\Pi^{-1}(U_{g})$, we denote by $\tilde{U}_{g}$ the lift of $U_{g}$ which contains $x$. Then the point $\tilde{g}(x)$ belongs to $\tilde{U}_{g}$ and a fundamental domain $\hat{D}$ which contains the point $\tilde{g}(x)$ is at distance at most $\mu$ from $D$ (for $d_{\mathcal{D}}$). Hence, in any case, there exists a fundamental domain $\hat{D}$ which contains the point $\tilde{g}(x)$ and is at distance at most $\mu$ from $D$. Similarly, there exists a fundamental domain $\hat{D}'$ which contains the point $\tilde{g}(x')$ and is at distance at most $\mu$ from $D'$. Therefore:
$$ d_{\mathcal{D}}(\hat{D},\hat{D'}) \geq d_{\mathcal{D}}(D,D') -2\mu.$$
We deduce that:
$$\mathrm{diam}_{\mathcal{D}}(\tilde{g}(A)) \geq \mathrm{diam}_{\mathcal{D}}(A)-2\mu,$$
what we wanted to prove.
\end{proof}

Thus, to complete the proof of Proposition \ref{equiv}, it suffices to prove Proposition \ref{diamfrag}.
\end{proof}

It suffices now to find a finite family $\mathcal{U}$ for which Proposition \ref{diamfrag} or \ref{diamfrag2} holds. We will distinguish the following cases. A section is devoted to each of them:
\begin{enumerate}
\item The surface $S$ has a nonempty boundary (Section 5).
\item The surface $S$ is the torus (Section 6).
\item The surface $S$ is closed of genus greater than one (Section 7).
\end{enumerate}
The proof of Propositions \ref{diamfrag} and \ref{diamfrag2}, in each of these cases, consists in putting back the boundary of $\tilde{g}(D_{0})$ close to the boundary of $\partial D_{0}$ by composing by homeomorphisms supported each in the interior of one of the balls of a well-chosen cover $\mathcal{U}$. Most of the time, after composing by a homeomorphism supported in the interior of one of the balls of $\mathcal{U}$, the image of the fundamental domain $D_{0}$ will not meet faces which were not met before the composition. However, it will not be always possible, which explains the technicality of parts of the proof. Then, we will have to assure that, after composing by a uniformly bounded number of homeomorphisms supported in interiors of balls of $\mathcal{U}$, the image of the boundary of $D_{0}$ will be strictly closer to $D_{0}$ than before. 

\section{Distortion and fragmentation on manifolds}

In this section, $M$ is a compact $d$-dimensional manifold, possibly with boundary.
Let us fix a finite family $\mathcal{U}$ of closed balls or half-balls of $M$ whose interiors cover $M$. For a homeomorphism $g$ in $\mathrm{Homeo}_{0}(M)$, we denote by $a_{\mathcal{U}}(g)$ the minimum of the quantities $l.log(k)$, where there exists a finite set $\left\{ f_{i}, \ 1 \leq i \leq k \right\}$ of $k$ homeomorphisms in $\mathrm{Homeo}_{0}(M)$, each supported in one of the elements of $\mathcal{U}$, and a map $ \nu : [1,l] \cap \mathbb{Z}_{>0} \rightarrow [1,k] \cap \mathbb{Z}_{>0}$ with:
$$g=f_{\nu(1)} \circ f_{\nu(2)} \circ \ldots \circ f_{\nu(l)}.$$

The aim of this section is to prove the following proposition:
\begin{proposition} \label{fragdist}
Let $f$ be a homeomorphism in $\mathrm{Homeo}_{0}(M)$.
Then:
$$ \lim_{n \rightarrow + \infty} \frac{a_{\mathcal{U}}(f^{n})}{n}=0$$
if and only if the homeomorphism $f$ is a distortion element in $\mathrm{Homeo}_{0}(M)$.
\end{proposition}

Let us give now an analogous statement in the case of the group $\mathrm{Homeo}_{0}(M, \partial M)$. Denote by $M'$ a submanifold with boundary of $M$ homeomorphic to $M$, contained in the interior of $M$ and which is a deformation retract of $M$. We denote by $\mathcal{U}$ a family of closed balls of $M$ whose interiors cover $M'$. For a homeomorphism $g$ in $\mathrm{Homeo}_{0}(M, \partial M)$ supported in $M'$, we define $a_{\mathcal{U}}(g)$ the same way as before.

\begin{proposition} \label{fragdist2}
Let $f$ be a homeomorphism in $\mathrm{Homeo}_{0}(M, \partial M)$ supported in $M'$.
Then:
$$ \lim_{n \rightarrow + \infty} \frac{a_{\mathcal{U}}(f^{n})}{n}=0$$
if and only if the homeomorphism $f$ is a distortion element in $\mathrm{Homeo}_{0}(M, \partial M)$.
\end{proposition}

As $a_{\mathcal{U}}(f) \leq \mathrm{Frag}_{\mathcal{U}}(f).log(\mathrm{Frag}_{\mathcal{U}}(f))$, these last propositions clearly imply Theorem \ref{fragdistth}.

\begin{proof}[Proof of the "if" statement in Propositions \ref{fragdist} and \ref{fragdist2}] 
If the homeomorphism $f$ is a distortion element, we denote by $\mathcal{G}$ the finite set which appears in the definition of a distortion element. Then we write each of the homeomorphisms in $\mathcal{G}$ as a product of homeomorphisms supported in one of the sets of $\mathcal{U}$. We denote by $\mathcal{G}'$ the (finite) set of homeomorphisms which appear in such a decomposition. Then the homeomorphism $f^{n}$ is a composition of $l_{n}$ elements of $\mathcal{G}'$, where $l_{n}$ is less than a constant independent from $n$ times $l_{\mathcal{G}}(f^{n})$. As the element $f$ is distorted, $\lim_{n \rightarrow + \infty} \frac{l_{n}}{n}=0$ and: 
$$a_{\mathcal{U}}(f^{n}) \leq log(card(\mathcal{G}')) l_{n}. $$
Therefore:
$$ \lim_{n \rightarrow + \infty} \frac{a_{\mathcal{U}}(f^{n})}{n}=0.$$
In the case of Proposition \ref{fragdist2}, there is only one new difficulty: the elements of $\mathcal{G}$ are not necessarily supported in the union of the balls of $\mathcal{U}$. Let us take a homeomorphism $h$ in $\mathrm{Homeo}_{0}(M, \partial M)$ with the following properties: the homeomorphism $h$ is equal to the identity on $M'$ and sends the union of the supports of elements of $\mathcal{G}$ in the union of the interiors of the balls of $\mathcal{U}$. Then it suffices to consider the finite set $h\mathcal{G}h^{-1}$ instead of $\mathcal{G}$ in order to complete the proof.
\end{proof}

The full power of Propositions \ref{fragdist} and \ref{fragdist2} will be used only for the proof of Theorem \ref{exemple} (construction of the example). In order to prove Theorem \ref{croissanceimpliquedistorsion}, we just used Theorem \ref{fragdistth} which is weaker.

\noindent \textbf{Remark} Notice that, if $\mathcal{U}$ is the cover of the sphere by two neighbourhoods of the hemispheres, the map $\mathrm{Frag}_{\mathcal{U}}$ is bounded by $3$ on the group $\mathrm{Homeo}_{0}(\mathbb{S}^{n})$  of homeomorphisms of the n-dimensional sphere isotopic to the identity (see \cite{CF}). This is a consequence of the annulus theorem by Kirby (see \cite{Kir}) and Quinn (see \cite{Q}). Thus, the following theorem by Calegari and Freedman (see \cite{CF}) is a consequence of Theorem \ref{fragdistth}:

\begin{theorem}[Calegari-Freedman \cite{CF}]
Any homeomorphism in $\mathrm{Homeo}_{0}(\mathbb{S}^{n})$ is a distortion element.
\end{theorem}

The proof of Proposition \ref{fragdist} is based on the following lemma, whose proof uses a technique due to Avila (see \cite{Avi}):

\begin{lemma} \label{Avila}
Let $(f_{n})_{n \geq 1}$ be a sequence of homeomorphisms of $\mathbb{R}^{d}$ (respectively of $H^{d}$) supported in $B(0,1)$ (respectively in $B(0,1) \cap H^{d}$). There exists a finite set $\mathcal{G}$ of compactly-supported homeomorphisms of $\mathbb{R}^{d}$ (respectively of $H^{d}$) such that:
\begin{enumerate}
\item For any natural number $n$, the homeomorphism $f_{n}$ belongs to the group generated by $\mathcal{G}$.
\item $l_{\mathcal{G}}(f_{n}) \leq 14.log(n)+14.$
\end{enumerate}
\end{lemma}

This lemma is not true anymore in case of the $C^{r}$ regularity, for $r \geq 1$. It crucially uses the following fact: given a sequence of homeomorphisms $(h_{n})$ supported in the unit ball $B(0,1)$, one can store all the information of this sequence in one homeomorphism. Let us explain now how to build such a homeomorphism. For any integer $n$, denote by $g_{n}$ a homeomorphism which sends the unit ball on a ball $B_{n}$ such that the balls $B_{n}$ are pairwise disjoint and have a diameter which converges to $0$. Then it suffices to consider the homeomorphism
$$ \prod _{n=1}^{\infty} g_{n}h_{n} g_{n}^{-1}.$$
Such a construction is not possible in the case of a higher regularity.

\noindent \textbf{Remark} There are two main differences between this lemma and the one stated by Avila:
\begin{enumerate}
\item Avila's lemma deals with a sequence of diffeomorphisms which converges sufficiently fast (in the $C^{\infty}$-topology) to the identity whereas any sequence of homeomorphisms is considered here.
\item The upper bound is logarithmic and not linear.
\end{enumerate}

\noindent \textbf{Remark} This lemma is optimal in the sense that, if the homeomorphisms $f_{n}$ are pairwise distinct, the growth of $l_{\mathcal{G}}(f_{n})$ is at least logarithmic. Indeed, if the generating set $\mathcal{G}$ contains $k$ elements, there are at most $\frac{k^{l+1}-1}{k-1}$ homeomorphisms whose length with respect to $\mathcal{G}$ is less than or equal to $l$.

Before proving Lemma \ref{Avila}, let us see why this lemma implies Propositions \ref{fragdist} and \ref{fragdist2}.

\begin{proof}[End of the proof of Propositions \ref{fragdist} and \ref{fragdist2}]
Suppose that:
$$ \lim_{n \rightarrow + \infty} \frac{a_{\mathcal{U}}(f^{n})}{n}=0.$$
Let
$$ \mathcal{U}= \left\{U_{1}, U_{2}, \ldots, U_{p} \right\}.$$
For any integer $i$ between $1$ and $p$, denote by $\varphi_{i}$ an embedding of $\mathbb{R}^{d}$ into $M$ which sends the closed ball $B(0,1)$ onto $U_{i}$ if $U_{i}$ is a closed ball or an embedding of $H^{d}$ into $M$ which sends the closed half-ball $B(0,1) \cap H^{d}$ onto $U_{i}$ if $U_{i}$ is a closed half-ball.
For any natural number $n$, let $l_{n}$ and $k_{n}$ be two positive integers such that:
\begin{enumerate}
\item $a_{\mathcal{U}}(f^{n})= l_{n} log(k_{n})$.
\item There exists a sequence $(f_{1,n},f_{2,n}, \ldots, f_{k_{n},n})$ of homeomorphisms in $\mathrm{Homeo}_{0}(M)$, each supported in one of the elements of $\mathcal{U}$, such that $f^{n}$ is the composition of $l_{n}$ homeomorphisms of this family.
\end{enumerate}

Let us build an increasing one-to-one function $\sigma: \mathbb{Z}_{>0} \rightarrow \mathbb{Z}_{>0}$ which satisfies:
$$ \forall n \in \mathbb{Z}_{>0}, \frac{l_{\sigma(n)}(14. log( \sum_{i=1}^{n} k_{\sigma(i)})+14)}{\sigma(n)} \leq \frac{1}{n}.$$
Suppose that, for some $m \geq 0$, $\sigma(1), \sigma(2), \ldots, \sigma(m)$ have been built. Then, as: 
$$ \lim_{n \rightarrow + \infty} \frac{l_{n}log(k_{n})}{n}=0,$$
we have
$$ \lim_{n \rightarrow + \infty} \frac{l_{n}(14. log( \sum_{i=1}^{m} k_{\sigma(i)}+k_{n})+14)}{n}=0.$$
Hence, we can find an integer $\sigma(m+1)> \sigma(m)$ such that:
$$\frac{l_{\sigma(m+1)}(14. log( \sum_{i=1}^{m+1} k_{\sigma(i)})+14)}{\sigma(m+1)} \leq \frac{1}{m}.$$
This completes the construction of the map $\sigma$.
Take a bijective map:
$$ \psi : \mathbb{Z}_{>0} \rightarrow \left\{ (i, \sigma(j)) \in \mathbb{Z}_{>0} \times \mathbb{Z}_{>0}, \left\{ \begin{array}{l} i \leq k_{\sigma(j)} \\ j \in \mathbb{Z}_{>0} \end{array} \right. \right\}$$
such that, if $\psi(n_{1})=(i_{1},\sigma(j_{1}))$, $\psi(n_{2})=(i_{2},\sigma(j_{2}))$ and $\sigma(j_{1})<\sigma(j_{2})$, then $n_{1} < n_{2}$. For instance, take the inverse of the bijective map
$$ \begin{array}{rcl}
\left\{ (i, \sigma(j)) \in \mathbb{Z}_{>0} \times \mathbb{Z}_{>0}, \left\{ \begin{array}{l} i \leq k_{\sigma(j)} \\ j \in \mathbb{Z}_{>0} \end{array} \right. \right\} & \rightarrow & \mathbb{Z}_{>0} \\
 (i, \sigma(j)) & \mapsto & i+\sum \limits_{j'<j} k_{\sigma(j')}
\end{array}
.
$$
Then:
$$\psi^{-1}(i, \sigma(j)) \leq \sum_{l=1}^{j}k_{\sigma(l)}.$$
Denote by $\tau_{i,j}$ an integer between $1$ and $p$ such that:
$$\mathrm{supp}(f_{i,j}) \subset U_{\tau_{i,j}}.$$
Then apply Lemma \ref{Avila} to the sequence of homeomorphisms 
$$\varphi_{\tau_{\psi(n)}}^{-1} \circ f_{\psi(n)} \circ \varphi_{\tau_{\psi(n)}},$$
where the $\varphi_{i}$'s were defined at the beginning of the proof. 
Let us denote by $\mathcal{G}$ the finite set given by Lemma \ref{Avila}. Let $\mathcal{G}_{i}$ be the finite set of homeomorphisms supported in $U_{i}$ of the form $\varphi_{i} \circ s \circ \varphi_{i}^{-1}$, where $s$ is a homeomorphism in $\mathcal{G}$. Let
$$\mathcal{G}'= \bigcup_{i=1}^{p}\mathcal{G}_{i}.$$
By Lemma \ref{Avila}:
$$ \forall n \in \mathbb{Z}_{>0}, \ l_{\mathcal{G}'}(f_{\psi(n)}) \leq Clog(n)+C'.$$
Now the homeomorphism $f^{\sigma(n)}$ can be decomposed as follows:
$$f^{\sigma(n)}=g_{1} \circ g_{2} \circ \ldots \circ g_{l_{\sigma(n)}},$$
where each of the homeomorphisms $g_{i}$ belongs to the set:
$$ \left\{ f_{1, \sigma(n)},f_{2, \sigma(n)}, \ldots, f_{k_{\sigma(n)}, \sigma(n)} \right\}.$$
Thus:
$$l_{\mathcal{G}'}(f^{\sigma(n)}) \leq l_{\sigma(n)}(Clog( \max \limits _{1 \leq i \leq k_{\sigma(n)}}\psi^{-1}(i, \sigma(n)))+C').$$
Therefore:
$$\frac{l_{\mathcal{G}'}(f^{\sigma(n)})}{\sigma(n)} \leq \frac{l_{\sigma(n)}(C. log( \sum_{i=1}^{n} k_{\sigma(i)})+C')}{\sigma(n)} \leq \frac{1}{n}$$
and the homeomorphism $f$ is a distortion element of $\mathrm{Homeo}_{0}(M)$ (respectively of $\mathrm{Homeo}_{0}(M, \partial M)$).
\end{proof}

Let us now prove Lemma \ref{Avila}. This will require two lemmas.

Let $a$ and $b$ be the generators of the free semigroup $L_{2}$ on two generators. For two compactly supported homeomorphisms $f$ and $g$ of $\mathbb{R}^{d}$, let $\eta_{f,g}$ be the semigroup morphism from $L_{2}$ to the group of homeomorphism of $\mathbb{R}^{d}$ defined by $\eta_{f,g}(a)=f$ and $\eta_{f,g}(b)=g$.

\begin{lemma} \label{groupelibre}
There exist compactly supported homeomorphisms $s_{1}$ and $s_{2}$ of $\mathbb{R}^{d}$ such that:
$$ \forall m \in L_{2}, \ m' \in L_{2}, \ m \neq m' \Rightarrow \eta_{s_{1},s_{2}}(m)(B(0,2)) \cap \eta_{s_{1},s_{2}}(m')(B(0,2))= \emptyset$$
and the diameter of $\eta_{s_{1},s_{2}}(m)(B(0,2))$ converges to $0$ when the length of $m$ tends to infinity.
\end{lemma}

\begin{figure}[ht]
\begin{center}
\includegraphics[scale=0.75]{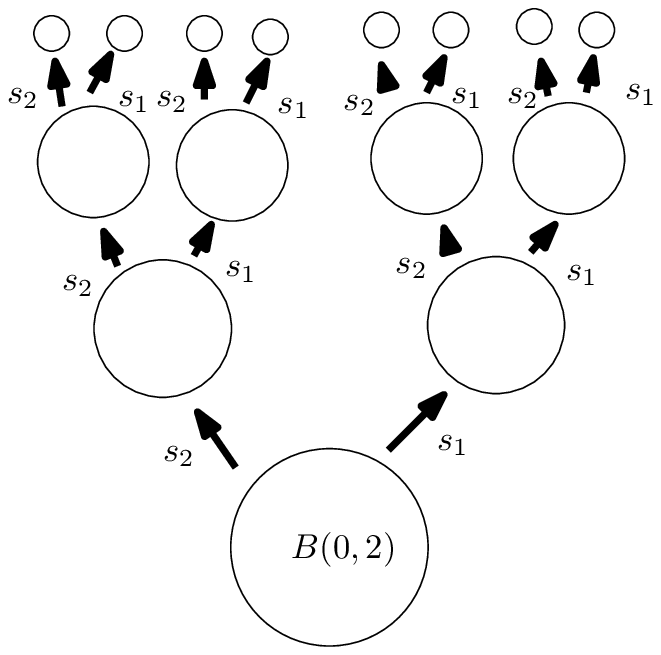}
\end{center}
\caption{Lemma \ref{groupelibre}}
\end{figure}

\begin{lemma}\label{commutateur}
Let $f$ be a homeomorphism in $\mathrm{Homeo}_{0}(\mathbb{R}^{d})$. There exist two homeomorphisms $g$ and $h$ in $\mathrm{Homeo}_{0}(\mathbb{R}^{d})$ such that:
$$f= [g,h],$$
where $[g,h] = g \circ h \circ g^{-1} \circ h^{-1}$.
\end{lemma}

This lemma is classical and seems to appear for the first time in \cite{And}. Let us prove it now.

\noindent \begin{proof}  Denote by $\varphi$ a homeomorphism in $\mathrm{Homeo}_{0}(\mathbb{R}^{d})$ whose restriction to $B(0,2)$ is defined by:
$$\begin{array}{rcl}
B(0,2) & \rightarrow & \mathbb{R}^{d} \\
x & \mapsto & \frac{x}{2}
\end{array}
.$$
For any natural number $n$, let
$$ A_{n}=\left\{x \in \mathbb{R}^{d}, \frac{1}{2^{n+1}}\leq \left\|x\right\| \leq \frac{1}{2^{n}} \right\}.$$
Let $f$ be an element in $\mathrm{Homeo}_{0}(\mathbb{R}^{N})$. As any element in $\mathrm{Homeo}_{0}(\mathbb{R}^{N})$ is conjugate to an element supported in the interior of $A_{0}$, we may suppose that the homeomorphism $f$ is supported in the interior of $A_{0}$. Then we define $g \in \mathrm{Homeo}_{0}(\mathbb{R}^{d})$ by:
\begin{enumerate}
\item $g=Id$ outside $B(0,1)$.
\item For any natural number $i$, $g_{|A_{i}}=\varphi^{i}f \varphi^{-i}$.
\item $g(0)=0$.
\end{enumerate}
Then:
$$ f= [g,\varphi].$$
\begin{figure}[ht]
\begin{center}
\includegraphics[scale=0.75]{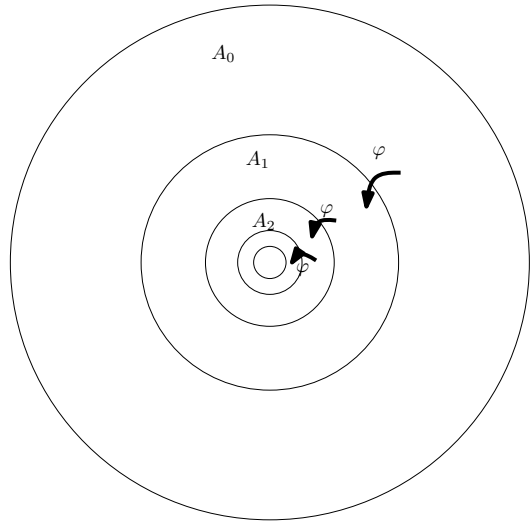}
\end{center}
\caption{Proof of Lemma \ref{commutateur}~: description of the homeomorphism $\varphi$}
\end{figure}
\end{proof}

These two lemmas remain true when we replace $\mathbb{R}^{d}$ with $H^{d}$ and $B(0,2)$ with $B(0,2) \cap H^{d}$.

Before proving Lemma \ref{groupelibre}, let us prove Lemma \ref{Avila} with the help of these two lemmas.

\begin{proof} [Proof of Lemma \ref{Avila}]
We prove the lemma in the case of homeomorphisms of $\mathbb{R}^{d}$. In the case of the half-space, the proof can be performed likewise.
For an element $m$ in $L_{2}$, let $l(m)$ be the length of $m$ as a word in $a$ and $b$. Let
$$ \begin{array}{rcl}
\mathbb{Z}_{>0} & \rightarrow & L_{2} \\
n & \mapsto & m_{n}
\end{array}$$
be a bijective map which satisfies:
$$l(m_{n})<l(m_{n'}) \Rightarrow n<n'.$$
This last condition implies that:
$$ l(m_{n})=l \Leftrightarrow 2^{l} \leq n < 2^{l+1}.$$
In particular, for any natural number $n$:
$$l(m_{n}) \leq log_{2}(n).$$
Let $s_{1}$ and $s_{2}$ be the homeomorphisms in $\mathrm{Homeo}_{0}(\mathbb{R}^{d})$ given by Lemma \ref{groupelibre}. Let $s_{3}$ be a homeomorphism in $\mathrm{Homeo}_{0}(\mathbb{R}^{d})$ supported in the ball $B(0,2)$ which satisfies:
$$s_{3}(B(0,1)) \cap B(0,1)= \emptyset.$$
We denote by $B_{n}$ the closed ball $\eta_{s_{1},s_{2}}(m_{n})(B(0,1))$. By Lemma \ref{commutateur}, there exist homeomorphisms $g_{n}$ and $h_{n}$ supported in $B(0,1)$ such that $f_{n}=[g_{n},h_{n}]$.

\begin{figure}[ht]
\begin{center}
\includegraphics[scale=0.75]{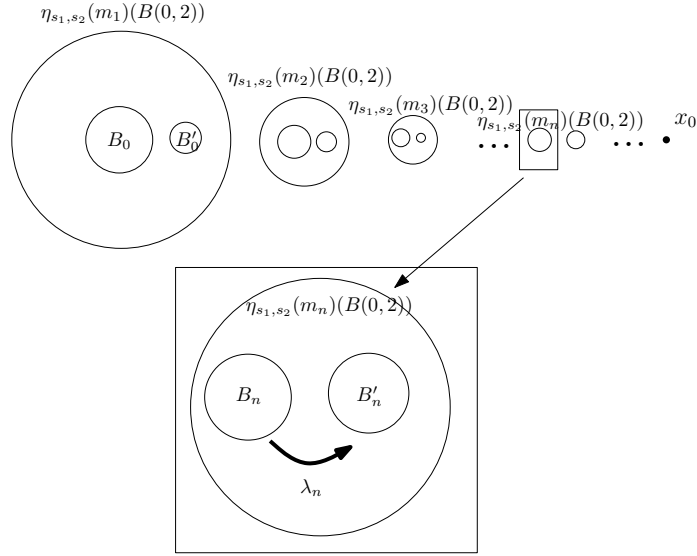}
\end{center}
\caption{Notations in the proof of Lemma \ref{Avila}}
\end{figure}

Define the homeomorphism $s_{4}$ by:
$$\left\{
\begin{array}{l}
\forall n \in \mathbb{Z}_{>0}, \ s_{4 |B_{n}}=\eta_{s_{1},s_{2}}(m_{n}) \circ g_{n} \circ \eta_{s_{1},s_{2}}(m_{n})^{-1} \\
s_{4} =Id \mbox{  on } \mathbb{R}^{d}- \bigcup \limits_{n \in \mathbb{Z}_{>0}} B_{n}
\end{array}
\right.
$$
and the homeomorphism $s_{5}$ by:
$$\left\{
\begin{array}{l}
\forall n \in \mathbb{Z}_{>0}, \ s_{5 |B_{n}}=\eta_{s_{1},s_{2}}(m_{n}) \circ h_{n} \circ \eta_{s_{1},s_{2}}(m_{n})^{-1} \\
s_{5} =Id \mbox{  on } \mathbb{R}^{d}- \bigcup \limits_{n \in \mathbb{Z}_{>0}} B_{n}
\end{array}
\right.
.$$
Let $\mathcal{G}= \left\{ s_{i}^{\epsilon}, \ i \in \left\{1, \ldots, 5 \right\} \mbox{ et } \epsilon \in \left\{-1,1 \right\}  \right\}$.
Let
$$\left\{
\begin{array}{l}
\lambda_{n}=\eta_{s_{1},s_{2}}(m_{n}) \circ s_{3} \circ \eta_{s_{1},s_{2}}(m_{n})^{-1} \\
B'_{n}=\lambda_{n}(B_{n})
\end{array}
\right.
$$
Notice that the balls $B_{n}$ and $B'_{n}$ are disjoint and contained in $\eta_{s_{1},s_{2}}(m_{n})(B(0,2))$.
Notice also that the homeomorphism $s_{4} \circ \lambda_{n} \circ s_{4}^{-1} \circ \lambda_{n}^{-1}$ (respectively $s_{5} \circ \lambda_{n} \circ s_{5}^{-1} \circ \lambda_{n}^{-1}$, $s_{4}^{-1} \circ s_{5}^{-1} \circ \lambda_{n} \circ s_{5} \circ s_{4} \circ \lambda_{n}^{-1}$) fixes the points outside $B_{n} \cup B'_{n}$, is equal to $\eta_{s_{1},s_{2}}(m_{n}) \circ g_{n} \circ \eta_{s_{1},s_{2}}(m_{n})^{-1}$ (respectively to $\eta_{s_{1},s_{2}}(m_{n}) \circ h_{n} \circ \eta_{s_{1},s_{2}}(m_{n})^{-1}$, $\eta_{s_{1},s_{2}}(m_{n}) \circ g_{n}^{-1} \circ h_{n}^{-1} \circ \eta_{s_{1},s_{2}}(m_{n})^{-1}$) on $B_{n}$ and to $\lambda_{n} \circ \eta_{s_{1},s_{2}}(m_{n})\circ g_{n}^{-1} \circ \eta_{s_{1},s_{2}}(m_{n})^{-1} \circ \lambda_{n}^{-1}$ (respectively to $\lambda_{n} \circ \eta_{s_{1},s_{2}}(m_{n})\circ h_{n}^{-1} \circ \eta_{s_{1},s_{2}}(m_{n})^{-1} \circ \lambda_{n}^{-1}$, $\lambda_{n} \circ \eta_{s_{1},s_{2}}(m_{n})\circ h_{n} \circ g_{n} \circ \eta_{s_{1},s_{2}}(m_{n})^{-1} \circ \lambda_{n}^{-1}$) on $B'_{n}$.

Therefore, the homeomorphism
$$[s_{4}, \lambda_{n}] [s_{5}, \lambda_{n}] [s_{4}^{-1}s_{5}^{-1},\lambda_{n}]$$
is equal to $\eta_{s_{1},s_{2}}(m_{n}) \circ f_{n} \circ \eta_{s_{1},s_{2}}(m_{n})^{-1}$ on $B_{n}$ and fixes the points outside $B_{n}$. Thus:
$$ f_{n}=\eta_{s_{1},s_{2}}(m_{n})^{-1} [s_{4}, \lambda_{n}] [s_{5}, \lambda_{n}] [s_{4}^{-1}s_{5}^{-1},\lambda_{n}] \eta_{s_{1},s_{2}}(m_{n}).$$
The homeomorphism $f_{n}$ hence belongs to the group generated by $\mathcal{G}$ and:
$$\begin{array}{rcl}
l_{\mathcal{G}}(f_{n}) & \leq & 2 l_{\mathcal{G}}(\eta_{s_{1},s_{2}}(m_{n}))+6l_{\mathcal{G}}(\lambda_{n})+8 \\
 & \leq & 2l_{\mathcal{G}}(\eta_{s_{1},s_{2}}(m_{n})) + 12l_{\mathcal{G}}(\eta_{s_{1},s_{2}}(m_{n})+14) \\
 & \leq & 14 log_{2}(n)+14.
 \end{array}
$$
\end{proof}

\begin{figure}
\begin{center}
\includegraphics[scale=0.75]{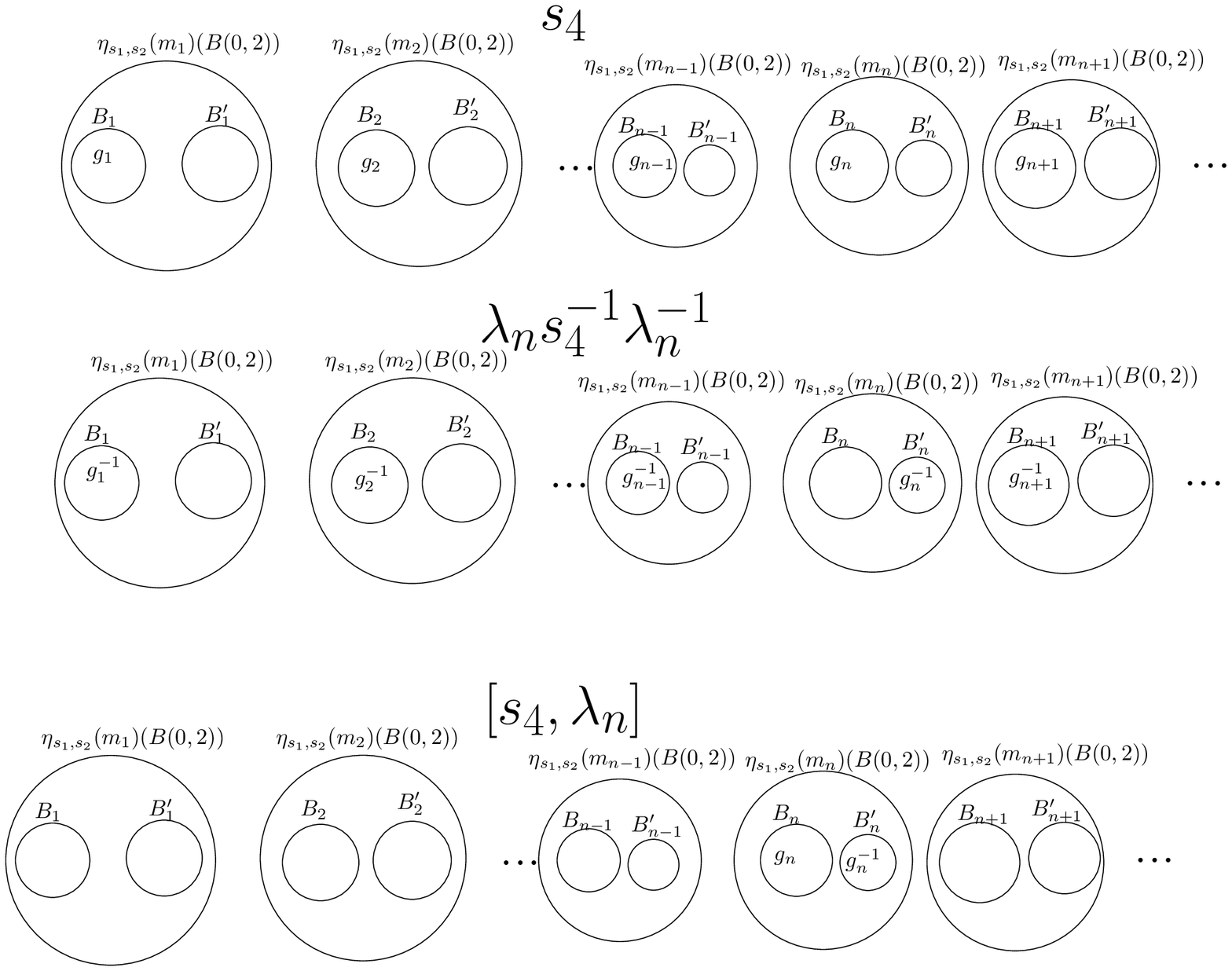}
\includegraphics[scale=0.75]{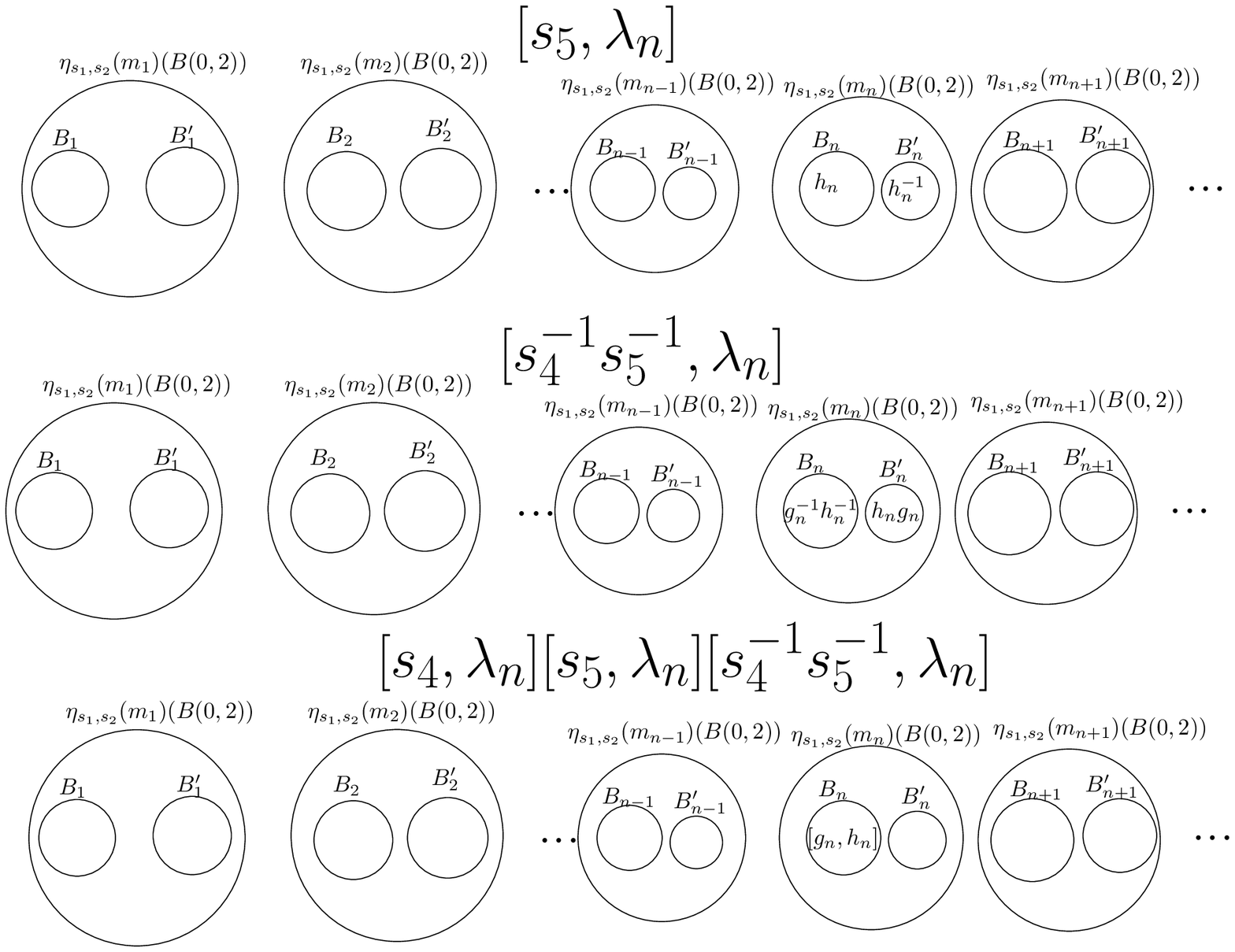}
\end{center}
\caption{The different homeomorphisms appearing in the proof of Lemma \ref{Avila}}
\end{figure}

\begin{proof}[Proof of Lemma \ref{groupelibre}]
First, let us prove the lemma in the case of homeomorphisms of $\mathbb{R}$. By perturbing two given homeomorphisms (as in \cite{Ghy}), one can find two compactly-supported homeomorphisms $\hat{s}_{1}$ and $\hat{s}_{2}$ of $\mathbb{R}$ which satisfy the following property:
$$ \forall m \in L_{2}, \ m' \in L_{2}, \ m \neq m' \Rightarrow \eta_{\hat{s}_{1},\hat{s}_{2}}(m)(0) \neq \eta_{\hat{s}_{1},\hat{s}_{2}}(m')(0).$$
Then, in the same way as in Denjoy's construction (see \cite{HK} p.403), replace each point of the orbit of $0$ under $L_{2}$ with an interval with positive length to obtain the wanted property. Thus, the proof is completed in the one-dimensional case. In the case of a higher dimension, denote by $f$ and $g$ the two homeomorphisms of $\mathbb{R}$ that we obtained in the one-dimensional case. Let $[-M,M]$ be an interval which contains the support of each of these homeomorphisms.\\
Let us look now at the case of $\mathbb{R}^{d}$. The homeomorphism:
$$\begin{array}{rcl}
\mathbb{R}^{d} & \rightarrow & \mathbb{R}^{d} \\
(x_{1},x_{2}, \ldots, x_{d}) & \mapsto & (f(x_{1}), f(x_{2}), \ldots, f(x_{d}))
\end{array}
$$
preserves the cube $[-M,M]^{d}$. Let $s_{1}$ be a homeomorphism of $\mathbb{R}^{d}$ supported in $[-M-1,M+1]^{d}$ which is equal to the above homeomorphism on $[-M,M]^{d}$. Apply the same construction to the homeomorphism $g$ to obtain a homeomorphism $s_{2}$. The ball centered on $0$ of radius $2$ of $\mathbb{R}^{d}$ is contained in the cube $[-2,2]^{d}$ and the diameters of the sets
$$ \eta_{s_{1},s_{2}}(m)([-2,2]^{d})=(\eta_{f,g}(m)([-2,2]))^{d}$$
converge to $0$ when the length of the word $m$ tends to infinity. Therefore, we have the wanted property.
The case of the half-spaces $H^{d}$ is similar as long as compactly-supported homeomorphisms which are equal to homeomorphisms of the form
$$\begin{array}{rcl}
 \mathbb{R}_{+} \times \mathbb{R}^{d-1} & \rightarrow & \mathbb{R}_{+} \times \mathbb{R}^{d-1} \\
(t,x_{1},x_{2}, \ldots, x_{d-1}) & \mapsto & (\frac{t}{2}, f(x_{1}), f(x_{2}), \ldots, f(x_{d-1}))
\end{array}
$$
in a neighbourhood of $0$ are used.
\end{proof}

\section{Case of surfaces with boundary} \label{bord}

Suppose that the boundary of the surface $S$ is nonempty. Let us prove now Proposition \ref{diamfrag2}. By considering a cover by half-discs, one can prove, with the same techniques as below, Proposition \ref{diamfrag} in the case that $S$ has a nonempty boundary: this case is left to the reader.

Recall that, in Section 3, we have chosen a "nice" fundamental domain $D_{0}$. Let $\tilde{A}$ be the set of edges of the boundary $\partial D_{0}$ which are not contained in the boundary of $\tilde{S}$ and let:
$$A =\left\{ \Pi (\beta), \ \beta \in \tilde{A} \right\}.$$
For any edge $\alpha$ in $A$, let us consider a closed disc $V_{\alpha}$ which does not meet the boundary of the surface $S$, whose interior contains $\alpha \cap S'$ and such that there exists a homeomorphism $\varphi_{\alpha} : V_{\alpha} \rightarrow \mathbb{D}^{2}$ which sends the set $\alpha \cap V_{\alpha}$ to the horizontal diameter of the unit disc $\mathbb{D}^{2}$. Choose sufficiently thin discs $V_{\alpha}$ so that they are pairwise disjoint. Let $U_{1}$ be a closed disc which contains the union of the discs $V_{\alpha}$. Let $U_{2}$ be a closed disc of $S$ which does not meet any edge in $A$, \emph{i.e.} contained in the interior of the fundamental domain $D_{0}$, and which satisfies the two following properties:
\begin{enumerate}
\item The surface $S'$ is contained in the interior of  $\bigcup \limits_{\alpha \in A} V_{\alpha} \cup U_{2}$.
\item For any edge $\alpha$ in $A$, the set $U_{2} \cap V_{\alpha}$ is homeomorphic to the disjoint union of two closed discs.
\end{enumerate}
Let $\mathcal{U}=\left\{ U_{1}, U_{2} \right\}$.

\begin{figure}[ht]
\begin{center}
\includegraphics{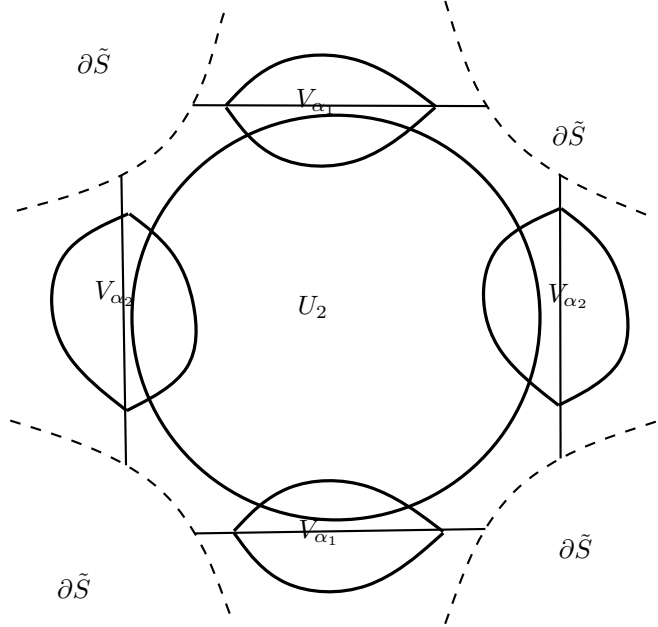}
\end{center}
\caption{Notations in the case of surfaces with boundary}
\end{figure}

The proof of the inequality in the case of the group $\mathrm{Homeo}_{0}(S, \partial S)$ requires the following lemmas:

\begin{lemma} \label{herbord}
Let $g$ be a homeomorphism in $\mathrm{Homeo}_{0}(S, \partial S)$ supported in the interior of $\bigcup  V_{\alpha} \cup U_{2}$. Suppose that $\mathrm{el}_{D_{0}}(\tilde{g}(D_{0})) \geq 2$. Then there exist homeomorphisms $g_{1}$, $g_{2}$ and $g_{3}$ in $\mathrm{Homeo}_{0}(S, \partial S)$ supported respectively in the interior of $\bigcup  V_{\alpha}$, $U_{2}$ and $\bigcup  V_{\alpha}$ such that the following property is satisfied:
$$ \mathrm{el}_{D_{0}}(\tilde{g}_{3} \circ \tilde{g}_{2} \circ \tilde{g}_{1} \circ \tilde{g}(D_{0})) \leq \mathrm{el}_{D_{0}}(\tilde{g}(D_{0}))-1.$$
\end{lemma}

\begin{lemma} \label{inbord}
 Let $g$ be a homeomorphism in $\mathrm{Homeo}_{0}(S, \partial S)$ supported in the interior of $\bigcup  V_{\alpha} \cup U_{2}$. If $\mathrm{el}_{D_{0}}(\tilde{g}(D_{0}))=1$, then:
 $$\mathrm{Frag}_{\mathcal{U}}(g) \leq 6.$$
\end{lemma}

\begin{proof}[End of the proof of Proposition \ref{diamfrag2}]
Let $k=\mathrm{el}_{D_{0}}(\tilde{g}(D_{0}))$. By Lemma \ref{herbord}, after composing $\tilde{g}$ with $3(k-1)$ homeomorphisms, each supported in one of the discs of $\mathcal{U}$, we obtain a homeomorphism $f_{1}$ supported in $\bigcup \limits_{\alpha \in A} V_{\alpha} \cup U_{2}$ with:
$$ \mathrm{el}_{D_{0}}(\tilde{f}_{1}(D_{0})) =1.$$
Then, apply Lemma \ref{inbord} to the homeomorphism $f_{1}$:
$$\mathrm{Frag}_{\mathcal{U}}(f_{1}) \leq 6.$$
Therefore:
$$\mathrm{Frag}_{\mathcal{U}}(g) \leq 3 (\mathrm{el}_{D_{0}}(\tilde{g}(D_{0}))-1)+6.$$
However, as $D_{0} \cap \tilde{g}(D_{0}) \neq \emptyset$ (the homeomorphism $g$ pointwise fixes a neighbourhood of the boundary of $S$):
$$\mathrm{el}_{D_{0}}(\tilde{g}(D_{0})) \leq \mathrm{diam}_{\mathcal{D}}(\tilde{g}(D_{0})). $$
Hence:
$$\mathrm{Frag}_{\mathcal{U}}(g
) \leq 3 \mathrm{diam}_{\mathcal{D}}(\tilde{g}(D_{0}))+3.$$
\end{proof}

Notice that we indeed proved the following more precise proposition:

\begin{proposition} \label{fragbord}
Let $g$ be a homeomorphism in $\mathrm{Homeo}_{0}(S, \partial S)$ supported in the interior of $\bigcup \limits_{\alpha \in A} V_{\alpha} \cup U_{2}$. Then:
$$\mathrm{Frag}_{\mathcal{U}}(g) \leq 3 \mathrm{diam}_{\mathcal{D}}(\tilde{g}(D_{0}))+3.$$
\end{proposition}

\begin{proof}[Proof of Lemma \ref{herbord}]
Let us first give the properties of the homeomorphisms $g_{1}$, $g_{2}$ and $g_{3}$ which will satisfy the conclusion of the lemma. Let us give an idea of the action of these homeomorphisms "with the hands". If we look at the pieces of the disc $\tilde{g}(D_{0})$ furthest from $D_{0}$, the homeomorphism $g_{1}$ repulses them back to the open set $U_{2}$, the homeomorphism $g_{2}$ repulses them outside the open set $U_{2}$ and the homeomorphism $g_{3}$ makes them exit from the fundamental domain of $\mathcal{D}$ in which these pieces were contained (see figure \ref{distorsionhomeo8}). Let us give the precise construction of these homeomorphisms.

\begin{figure}[ht]
\begin{center}
\includegraphics{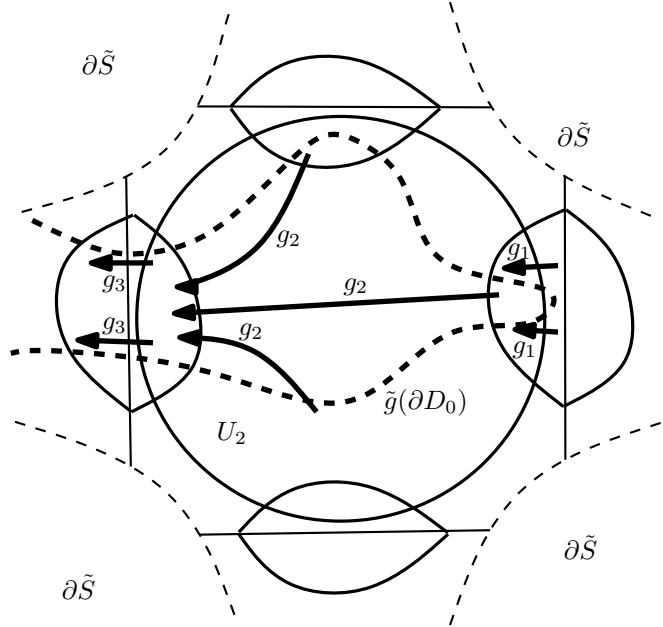}
\end{center}
\caption{Illustration of the proof of Lemma \ref{herbord}}
\label{distorsionhomeo8}
\end{figure}

Let $g_{1}$ be a homeomorphism supported in $\bigcup \limits_{\alpha \in A} V_{\alpha}$ such that:
\begin{enumerate}
\item The homeomorphism $g_{1}$ pointwise fixes $\Pi(\partial D_{0})$.
\item For any edge $\alpha$ and any connected component $C$ of $V_{\alpha} \cap g(\Pi(\partial D_{0}))$ which does not meet $\Pi(\partial D_{0})$:
$$ g_{1}(C) \subset U_{2}.$$
\end{enumerate}
One can build such a homeomorphism $g_{1}$ by taking the time $1$ of the flow of a well-chosen vector field which vanishes on $\Pi( \partial D_{0})$.

Let $g_{2}$ be a homeomorphism supported in $U_{2}$ which satisfies the following property: for any edge $\alpha$ in $A$ and for any connected component $C$ of $\mathring{U}_{2} \cap g_{1} \circ g(\Pi(\partial D_{0}))$ whose both ends (\emph{i.e.} the points of the closure of $C$ which do not belong to $C$) belong to $V_{\alpha}$, the set $g_{2}(C)$ is contained in $\mathring{V}_{\alpha}$. Let us explain how such a homeomorphism $g_{2}$ can be built. We will need the following elementary lemma which is a consequence of the Schönflies theorem:

\begin{lemma} \label{schonflies}
Let $c_{1} : [0,1] \rightarrow \mathbb{D}^{2}$ and $c_{2} : [0,1] \rightarrow \mathbb{D}^{2}$ be two injective curves which are equal in a neighbourhood of $0$ and in a neighbourhood of $1$ and such that:
\begin{enumerate}
\item $c_{1}(0)=c_{2}(0) \in \partial \mathbb{D}^{2}$ and $c_{1}(1)=c_{2}(1) \in \partial \mathbb{D}^{2}$.
\item $c_{1}((0,1)) \subset \mathbb{D}^{2}-\partial \mathbb{D}^{2}$ and $c_{2}((0,1)) \subset \mathbb{D}^{2}-\partial \mathbb{D}^{2}$.
\end{enumerate}
Then, there exists a homeomorphism $h$ in $\mathrm{Homeo}_{0}(\mathbb{D}^{2}, \partial \mathbb{D}^{2})$ such that:
$$ \forall t \in [0,1], h(c_{1}(t))=c_{2}(t).$$
\end{lemma}

\begin{corollary}
Let $(c_{i})_{1 \leq i \leq l}$ and $(c'_{i})_{1 \leq i \leq l}$ be finite sequences of injective curves $[0,1] \rightarrow \mathbb{D}^{2}$ of the closed disc $\mathbb{D}^{2}$ such that:
\begin{enumerate}
\item For any index $1 \leq i \leq l$, the maps $c_{i}$ and $c'_{i}$ are equal in a neighbourhood of $0$ and of $1$.
\item The curves $c_{i}$ are pairwise disjoint, as the curves $c'_{i}$.
\item For any index $i$, the points $c_{i}(0)$ and $c_{i}(1)$ belong to the boundary of the disc.
\item For any index $i$, the sets $c_{i}((0,1))$ and $c'_{i}((0,1))$ are contained in $\mathbb{D}^{2}- \partial \mathbb{D}^{2}$.
\end{enumerate}
Then there exists a homeomorphism $h$ in $\mathrm{Homeo}_{0}(\mathbb{D}^{2}, \partial \mathbb{D}^{2})$ such that, for any index $1 \leq i \leq l$:
$$\forall t \in [0,1], \ h(c_{i}(t))=c'_{i}(t).$$
\end{corollary}

\begin{proof}[Proof of the corollary]
It suffices to use Lemma \ref{schonflies} and an induction.
\end{proof}

Let us notice first that only a finite number of connected components of $\mathring{U}_{2}\cap g_{1} \circ g(\Pi(\partial D_{0}))$ is not contained in one of the open disc $\mathring{V}_{\alpha}$. We denote by $\mathcal{C}$ the set of such connected components with both ends in a same disc of the form $V_{\alpha}$, for an edge $\alpha$ in $A$. Let us fix now an edge $\alpha$ in $A$. Let $C$ be a connected component in $\mathcal{C}$ whose both ends belong to $V_{\alpha}$. We denote by $\delta_{C} : [0,1] \rightarrow D$ an injective path contained in $\mathring{V}_{\alpha} \cap U_{2}$ which is equal to the path $\overline{C}$ in a neighbourhood of $\delta(0)$ and of $\delta(1)$. The construction is made in such a way that the paths $\delta_{C}$ are pairwise disjoint. Then we apply the last corollary in the disc $U_{2}$ to the families of paths $(C)_{C \in \mathcal{C}}$ and $(\delta_{C})_{C \in \mathcal{C}}$ to build the homeomorphism $g_{2}$ that we wanted.

Finally, let $g_{3}$ be  a homeomorphism supported in $\bigcup \limits_{\alpha \in A} V_{\alpha}$ which satisfies, for any edge $\alpha$ in $A$, the following properties:
\begin{enumerate}
\item For any connected component $C$ of $\mathring{V}_{\alpha} \cap g_{2} \circ g_{1} \circ g ( \Pi( \partial D_{0}))$ whose both ends belong to the same connected component of $V_{\alpha}-\alpha$, $g_{3}(C) \cap \alpha = \emptyset$.
\item The homeomorphism $g_{3}$ pointwise fixes any other connected component of $\mathring{V}_{\alpha} \cap g_{2} \circ g_{1} ( \Pi( \partial D_{0}))$.
\end{enumerate}
The construction of the homeomorphism $g_{3}$ is analogous to the construction of the homeomorphism $g_{2}$. In what follows, we will not give details anymore on this kind of construction.

We claim that homeomorphisms $g_{1}$, $g_{2}$ and $g_{3}$ which satisfy the above properties satisfy also the conclusion of Lemma \ref{herbord}. This is a consequence of the two following claims.

\textbf{Claim 1.} The set of fundamental domains in $\mathcal{D}$ which meet $\tilde{g}_{3}\circ \tilde{g}_{2} \circ \tilde{g}_{1} \circ \tilde{g}(D_{0})$ is contained in the set of fundamental domains of $\mathcal{D}$ which meet $\tilde{g}(D_{0})$.

If $h$ is a homeomorphism in $\mathrm{Homeo}_{0}(S, \partial S)$, we will say that a fundamental domain $D$ in $\mathcal{D}$ is extremal for $\tilde{h}$ if it meets $\tilde{h}(D_{0})$ and satisfies:
$$d_{\mathcal{D}}(D,D_{0})=\mathrm{el}_{D_{0}}(\tilde{h}(D_{0})).$$

\textbf{Claim 2.} The fundamental domains $D$ in $\mathcal{D}$ which are extremal for $\tilde{g}$ do not meet $\tilde{g}_{3} \circ \tilde{g}_{2} \circ \tilde{g}_{1} \circ \tilde{g}(D_{0})$.

Let us assume for the moment that these two claims are true and let us prove Lemma \ref{herbord}.

Claim 1 implies that:
$$ \mathrm{el}_{D_{0}}(\tilde{g}_{3} \circ \tilde{g}_{2} \circ \tilde{g}_{1} \circ \tilde{g}(D_{0})) \leq \mathrm{el}_{D_{0}}(\tilde{g}(D_{0})).$$
Suppose that we have an equality in the above inequality. Then there exists a fundamental domain $D$ in $\mathcal{D}$ which is extremal for $\tilde{g}$ and which meets $\tilde{g}_{3} \circ \tilde{g}_{2} \circ \tilde{g}_{1} \circ \tilde{g}(D_{0})$, a contradiction with Claim 2. This proves the lemma.

Now, let us prove claim 1.\\
First, notice that, for a homeomorphism $h$ in $\mathrm{Homeo}_{0}(S, \partial S)$, the set of fundamental domains of $\mathcal{D}$ met by $\tilde{h}(D_{0})$ is equal to the set of fundamental domains of $\mathcal{D}$ met by $\tilde{h}(\partial D_{0})$ as the interior of a fundamental domain cannot contain a fundamental domain.\\
As the homeomorphisms $\tilde{g}_{1}$ and $\tilde{g}_{2}$ both pointwise fix $\bigcup \limits _{ D \in \mathcal {D}} \partial D $, the set of elements of $\mathcal{D}$ met by $\tilde{g}_{2} \circ \tilde{g}_{1} \circ \tilde{g}(\partial D_{0})$ is equal to the set of elements of $\mathcal{D}$ met by $\tilde{g}( \partial D_{0})$. Therefore, it suffices to prove the following inclusion:
$$ \left\{ D \in \mathcal{D}, \ \tilde{g}_{3} \circ \tilde{g}_{2} \circ \tilde{g}_{1} \circ \tilde{g}( \partial D_{0})\cap D \neq \emptyset \right\} \subset \left\{ D \in \mathcal{D}, \ \tilde{g}_{2} \circ \tilde{g}_{1} \circ \tilde{g}(\partial D_{0})\cap D \neq \emptyset \right\}.$$

Let $D$ be a fundamental domain which belongs to the left-hand set in the above inclusion. Let $\tilde{x}$ be a point in $\tilde{g}_{2} \circ \tilde{g}_{1} \circ \tilde{g}(\partial D_{0})$ which satisfies: $\tilde{g}_{3}(\tilde{x}) \in D$.

If the point $\tilde{x}$ belongs to the fundamental domain $D$, then the fundamental domain $D$ belongs to
$$\left\{ D' \in \mathcal{D}, \ \tilde{g}_{2} \circ \tilde{g}_{1} \circ \tilde{g}(\partial D_{0})\cap D' \neq \emptyset \right\}.$$
Hence, let us suppose that the point $\tilde{x}$ does not belong to the fundamental domain $D$. As the homeomorphism $g_{3}$ is supported in $\bigcup \limits_{\beta \in A} V_{\beta}$, there exists an edge $\alpha$ in $A$ such that the point $\Pi(\tilde{x})$ belongs to the disc $V_{\alpha}$. Let $\tilde{V}_{\alpha}$ be the lift of the disc $V_{\alpha}$ which contains $\tilde{x}$. By construction of the homeomorphism $\tilde{g}_{3}$, the point $\tilde{x}$ belongs to a connected component $\tilde{C}$ of $\tilde{g}_{2} \circ \tilde{g}_{1}( \partial D_{0}) \cap \mathring{\tilde{V}}_{\alpha}$ whose both ends belong to the interior of a same fundamental domain $D'$ in $\mathcal{D}$. Let us recall that the connected components which are not of this kind are fixed by the homeomorphism $g_{3}$. By definition of $\tilde{g}_{3}$, we have:
$$\tilde{g}_{3}(\tilde{x}) \in \tilde{g}_{3}(\tilde{C}) \subset \mathring{D}'$$
and, by hypothesis:
$$\tilde{g}_{3}(\tilde{x}) \in D.$$
Thus, $D'=D$ and, as the fundamental domain $D'$ meets $\tilde{C} \subset \tilde{g}_{2} \circ \tilde{g}_{1} \circ \tilde{g}(\partial D_{0})$, the fundamental domain $D$ belongs to the set
$$\left\{ D \in \mathcal{D}, \ \tilde{g}_{2} \circ \tilde{g}_{1} \circ \tilde{g}(\partial D_{0})\cap D \neq \emptyset \right\}.$$

We now come to the proof of claim 2. As in Section 3, let
$$\mathcal{G}= \left\{a_{i}, i \in \left\{1, \ldots, P \right\} \right\} \cup\left\{a_{i}^{-1}, i \in \left\{1, \ldots, P \right\} \right\}$$
be the generating set of the group $\Pi_{1}(S)$ which consists of the deck transformations which send the fundamental domain $D_{0}$ on a fundamental domain in $\mathcal{D}$ adjacent to $D_{0}$. As, in the case under discussion, the surface $S$ has a nonempty boundary, the group $\Pi_{1}(S)$ is the free group generated by $ \left\{a_{1}, a_{2}, \ldots, a_{p} \right\}$. Let $D_{ex}$ be a fundamental domain in $\mathcal{D}$ which is extremal for $\tilde{g}$. By definition:
$$ d_{\mathcal{D}}(D_{ex},D_{0})=\mathrm{el}_{D_{0}}(\tilde{g}(D_{0})).$$
Let us denote by $\gamma$ the deck transformation which sends $D_{0}$ to $D_{ex}$. The element $\gamma$ can be uniquely written as a reduced word on elements of $\mathcal{G}$:
$$ \gamma = s_{1} s_{2} \ldots s_{n}$$
where the $s_{i}$'s belong to the generating set $\mathcal{G}$ and $n=d_{\mathcal{D}}(D_{ex},D_{0})$. Every fundamental domain in $\mathcal{D}$ adjacent to $D_{ex}$ is a domain of the form $\gamma(s(D_{0}))$, where $s$ is an element in $\mathcal{G}$. If the element $s$ is different from $s_{n}^{-1}$, then:
$$d_{\mathcal{D}}(\gamma(s(D_{0})),D_{0})=l_{\mathcal{G}}(\gamma s)= n+1 >n= \mathrm{el}_{D_{0}}(\tilde{g}(\partial D_{0})).$$
Thus, the only face adjacent to $D_{ex}$ which meets $\tilde{g}(\partial D_{0})$ is $\gamma \circ s_{n}^{-1}(D_{0})$. We denote by $\tilde{\alpha}$ the edge which belongs to the fundamental domains $\gamma \circ s_{n}^{-1}(D_{0})$ and $D_{ex}$. The ends of any connected component of $\tilde{g}(\partial D_{0}) \cap D_{ex}$ belong to $\tilde{\alpha}$. These connected component do not meet the other edges of $\partial D_{ex}$.
Let $\tilde{V}_{\tilde{\alpha}}$ be the lift of $V_{\Pi(\tilde{\alpha})}$ which contains $\tilde{\alpha}$. We claim that:
$$\tilde{g}_{1} \circ \tilde{g}(\partial D_{0}) \cap D_{ex} \subset \tilde{V}_{\tilde{\alpha}} \cup \tilde{U}_{2},$$
where $\tilde{U}_{2}$ is the lift of $U_{2}$ which is contained in $D_{ex}$.

Let us prove this last claim. For a point $\tilde{x}$ in $D_{ex} \cap \tilde{g}(\partial D_{0}) \cap \Pi^{-1}(V_{\beta})-\tilde{V}_{\tilde{\alpha}}$, where $\beta$ is an edge in $A$, the connected component of $\tilde{g}(\partial D_{0}) \cap \Pi^{-1}(\mathring{V_{\beta}})$ which contains $\tilde{x}$ does not meet the set $\Pi^{-1}(\beta)$. Hence the point $\tilde{g}_{1}( \tilde{x})$ belongs to $U_{2}$, by construction of $g_{1}$. As, moreover, the homeomorphism $\tilde{g}_{1}$ preserves the following sets:
$$ \tilde{U}_{2} - (\bigcup_{\beta \in A} \Pi^{-1}(V_{\beta})) \mbox{ et } \tilde{V}_{\tilde{\alpha}},$$
the claim is proved.

Notice also that:
$$ \tilde{g}_{2} \circ \tilde{g}_{1} \circ \tilde{g}(\partial D_{0}) \cap D_{ex} \subset \mathring{\tilde{V}}_{\tilde{\alpha}}.$$
Indeed, the ends of any connected component of $\tilde{g}_{1} \circ \tilde{g}(\partial D_{0}) \cap \mathring{\tilde{U}}_{2}$ belong to $\tilde{V}_{\tilde{\alpha}}$. 

Let us prove that:
$$\tilde{g}_{3} \circ \tilde{g}_{2} \circ \tilde{g}_{1} \circ \tilde{g}( \partial D_{0}) \cap D_{ex}= \emptyset.$$
Let $C$ be a connected component of $\tilde{g}_{2} \circ \tilde{g}_{1} \circ \tilde{g}( \partial D_{0}) \cap \mathring{\tilde{V}}_{\tilde{\alpha}}$. As
$$ \tilde{g}_{2} \circ \tilde{g}_{1} \circ \tilde{g}( \partial D_{0}) \cap D_{ex} \subset \mathring{\tilde{V}}_{\tilde{\alpha}},$$
the ends of $C$ do not belong to $\mathring{D}_{ex} \cap \tilde{V}_{\tilde{\alpha}}$ but to $\gamma \circ s_{n}^{-1}(D_{0}) \cap \tilde{V}_{\tilde{\alpha}}$ which is the other connected component of $\tilde{V}_{\tilde{\alpha}}-\tilde{\alpha}$ (the ends of $C$ do not belong to $\alpha$ because $\mathrm{el}_{D_{0}}(\tilde{g}(D_{0}))=d_{\mathcal{D}}(D_{ex},D_{0}) \geq 2$). By construction of the homeomorphism $g_{3}$:
$$ \tilde{g}_{3}(C) \subset \gamma \circ s_{n}^{-1}( \mathring{D}_{0}).$$
Thus, the set $\tilde{g}_{3}(C)$ is disjoint from $D_{ex}$, which completes the proof of the second claim.
\end{proof}

\begin{proof}[Proof of Lemma \ref{inbord}]
For any edge $\tilde{\alpha}$ in $\tilde{A}$, we denote by $D_{\tilde{\alpha}}$ the fundamental domain in $\mathcal{D}$ which satisfies:
$$ D_{0} \cap D_{\tilde{\alpha}}= \tilde{\alpha}.$$
Let us fix an edge $\tilde{\alpha}$ of $\tilde{A}$. As $\mathrm{el}_{D_{0}}(\tilde{g}(D_{0}))=1$, the curve $\tilde{g}(\tilde{\alpha})$ does not meet fundamental domains in $\mathcal{D}$ adjacent to $D_{\tilde{\alpha}}$ and different from $D_{0}$: these fundamental domains are at distance $2$ from $D_{0}$. Let us prove that, if $\tilde{\beta}$ is an edge of $\tilde{A}$ different from $\tilde{\alpha}$, then 
$$ \tilde{g}(\tilde{\alpha}) \cap D_{\tilde{\beta}}= \emptyset .$$
Otherwise, we would have
$$ \tilde{g}(D_{\tilde{\alpha}}) \cap D_{\tilde{\beta}} \neq \emptyset,$$
for an edge $\tilde{\beta}$ different from $\tilde{\alpha}$. Let us denote by $s$ the deck transformation which sends $D_{0}$ to $D_{\tilde{\alpha}}$. Then:
$$ 2= d_{\mathcal{D}}(D_{\tilde{\alpha}}, D_{\tilde{\beta}})=d_{\mathcal{D}}(D_{0},s^{-1}(D_{\tilde{\beta}})).$$
Moreover:
$$\tilde{g}(s(D_{0})) \cap D_{\tilde{\beta}} \neq \emptyset.$$
Hence:
$$ \tilde{g}(D_{0}) \cap s^{-1}(D_{\tilde{\beta}}) \neq \emptyset.$$
It contradicts the hypothesis
$$\mathrm{el}_{D_{0}}(\tilde{g}(D_{0}))=1.$$
Thus, for any edge $\tilde{\alpha}$ in $\tilde{A}$:
$$\tilde{g}(\tilde{\alpha}) \subset \mathring{D}_{\tilde{\alpha}} \cup \mathring{D}_{0} \cup \tilde{\alpha}.$$
For an edge $\tilde{\alpha}$ in $\tilde{A}$, we denote by $\tilde{V}_{\tilde{\alpha}}$ the lift of $V_{\Pi(\tilde{\alpha})}$ which contains the edge $\tilde{\alpha}$.

We now build homeomorphisms $g_{1}$ and $g_{2}$ supported respectively in $\bigcup \limits_{\alpha \in A} V_{\alpha}$ and in $U_{2}$ such that:
$$ \forall \tilde{\alpha} \in \tilde{A}, \ \tilde{g}_{2} \circ \tilde{g}_{1} \circ \tilde{g}(\tilde{\alpha}) \subset \mathring{\tilde{V}}_{\tilde{\alpha}} \cup \tilde{\alpha}.$$

\begin{figure}[ht]
\begin{center}
\includegraphics{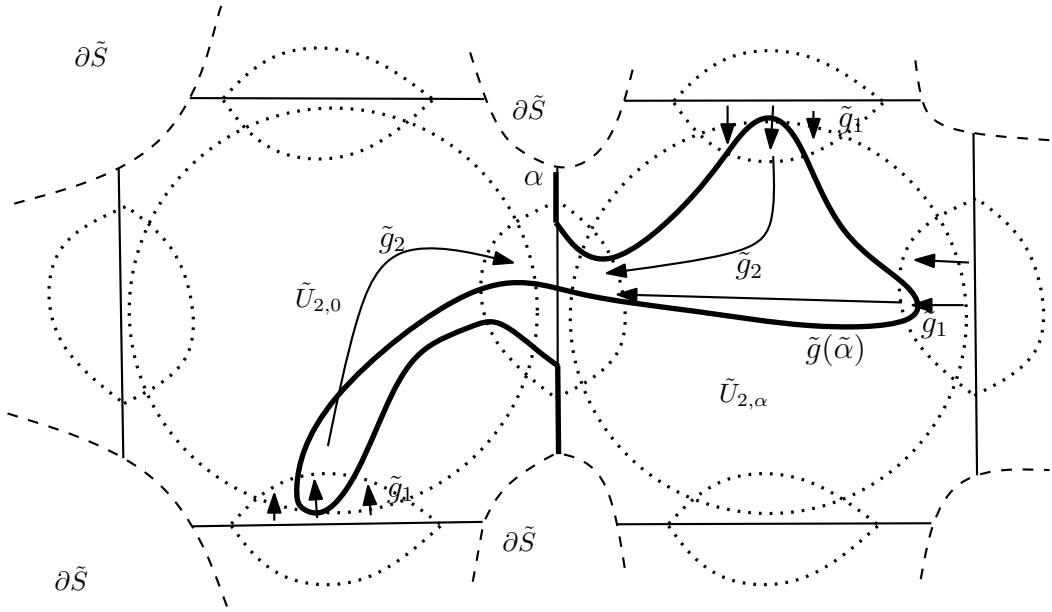}
\end{center}
\caption{Proof of Lemma \ref{inbord}: the homeomorphisms $g_{1}$ and $g_{2}$}
\end{figure}

As in the proof of Lemma \ref{herbord}, we build homeomorphisms $g_{1}$ and $g_{2}$ which satisfy the following properties:
\begin{enumerate}
\item The homeomorphism $g_{1}$ is supported in $\bigcup \limits_{\alpha \in A} V_{\alpha}$ and pointwise fixes $\partial D_{0}$.
\item For any edge $\alpha$ in $A$ and any connected component $C$ of $g ( \Pi (\partial D_{0})) \cap \mathring{V}_{\alpha}$ which does not meet $\alpha$, we have $g_{1}(C) \subset U_{2}$.
\item The homeomorphism $g_{2}$ is supported in $U_{2}$.
\item For any connected component $C$ of $g_{1} \circ g (\Pi (\partial D_{0})) \cap \mathring{U}_{2}$ whose ends belong to a same connected component of $V_{\alpha}-\alpha$ and for an edge $\alpha$ in $A$, $g_{2}(C) \subset \mathring{V}_{\alpha}$.
\end{enumerate}
Let us denote by $\tilde{U}_{2,0}$ the lift of the disc $U_{2}$ contained in $D_{0}$ and, for any edge $\tilde{\alpha}$ in $\tilde{A}$, $\tilde{U}_{2,\tilde{\alpha}}$ the lift of the disc $U_{2}$ contained in $D_{\tilde{\alpha}}$. By the same techniques as in the proof of Lemma \ref{herbord}, for any edge $\tilde{\alpha}$ in $\tilde{A}$:
$$\tilde{g}_{1} \circ \tilde{g} (\tilde{\alpha}) \subset \mathring{\tilde{U}}_{2,0} \cup \tilde{V}_{\tilde{\alpha}} \cup \mathring{\tilde{U}}_{2, \tilde{\alpha}}$$
and
$$ \tilde{g}_{2} \circ \tilde{g}_{1} \circ \tilde{g} (\tilde{\alpha}) \subset \mathring{\tilde{V}}_{\tilde{\alpha}}.$$

We will now build homeomorphisms $g_{3}$ and $g_{4}$ of $S$ supported respectively in $\bigcup \limits_{\alpha \in A} V_{\alpha}$ and $U_{2}$ such that, for any edge $\tilde{\alpha}$ in $\tilde{A}$, the homeomorphism $\tilde{g}_{4} \circ \tilde{g}_{3} \circ \tilde{g}_{2} \circ \tilde{g}_{1} \circ \tilde{g}$ pointwise fixes $\partial \tilde{V}_{\tilde{\alpha}}$.

Let $g_{3}$ be a homeomorphism supported in $\bigcup \limits_{\alpha \in A} V_{\alpha}$ which satisfies the following properties:
\begin{enumerate}
\item The homeomorphism $g_{3}$ pointwise fixes $g_{2} \circ g_{1} \circ g(\alpha)$.
\item For any connected component $C$ of $g_{2} \circ g_{1} \circ g(\partial V_{\alpha}) \cap \mathring{V}_{\alpha}$: $g_{3}(C) \subset \mathring{U}_{2}$.
\end{enumerate}
Then, the set $g_{3} \circ g_{2} \circ g_{1} \circ g(\partial V_{\alpha}) \Delta \partial V_{\alpha}$ is contained in $\mathring{U}_{2}$.

We impose that the homeomorphism $g_{4}$ is supported in $U_{2}$ and satisfies the following property: the homeomorphism $g_{4}$ is equal to $(g_{3} \circ g_{2} \circ g_{1} \circ g)^{-1}$ on the closed set $g_{3} \circ g_{2} \circ g_{1} \circ g(\partial V_{\alpha})$. The construction of $g_{3}$ has enabled the construction of $g_{4}$ with the above properties. Thus, as the homeomorphism $g_{4} \circ g_{3} \circ g_{2} \circ g_{1} \circ g$ pointwise fixes $\bigcup \limits_{\alpha \in A} \partial V_{\alpha}$, the map $g_{5}: S \rightarrow S$, which is equal to $g_{4} \circ g_{3} \circ g_{2} \circ g_{1} \circ g$ on $\bigcup \limits_{\alpha \in A} V_{\alpha}$ and to the identity outside this set, is a homeomorphism of $S$ supported in $\bigcup \limits_{\alpha \in A} V_{\alpha}$. Then, let $g_{6}= (g_{5} \circ g_{4} \circ g_{3} \circ g_{2} \circ g_{1} \circ g)^{-1}$. Then the homeomorphism $g_{6}$ is supported in $U_{2}$ and we have:
$$g=g_{1}^{-1} \circ g_{2}^{-1} \circ g_{3}^{-1} \circ g_{4}^{-1} \circ g_{5}^{-1} \circ g_{6}^{-1}.$$
This implies that $\mathrm{Frag}_{\mathcal{U}}(g) \leq 6$, which proves the lemma.
\end{proof}

\section{Case of the torus}

In this section, we prove Proposition \ref{diamfrag} in the case of the torus $\mathbb{T}^{2}=\mathbb{R}^{2}/\mathbb{Z}^{2}$.
We set $D_{0}=[0,1]^{2}$ and the covering $\Pi$ is given by the projection $\mathbb{R}^{2} \rightarrow \mathbb{R}^{2}/\mathbb{Z}^{2}$. We denote by $A_{0}$ (respectively $A_{1}$, $B_{0}$, $B_{1}$) the closed annulus $[-\frac{1}{4}, \frac{1}{2}] \times \mathbb{R} / \mathbb{Z} \subset \mathbb{T}^{2}$ (respectively $[\frac{1}{4}, 1] \times \mathbb{R} / \mathbb{Z}$, $\mathbb{R} / \mathbb{Z} \times [-\frac{1}{4}, \frac{1}{2}]$, $\mathbb{R} / \mathbb{Z} \times [\frac{1}{4}, 1]$). For any integer $i$, we denote by $\tilde{A}_{0}^{i}$ (respectively $\tilde{A}_{1}^{i}$, $\tilde{B}_{0}^{i}$, $\tilde{B}_{1}^{i}$) the band of the plane $[i-\frac{1}{4}, i+\frac{1}{2}] \times \mathbb{R}$ (respectively $[i+\frac{1}{4},i+ 1] \times \mathbb{R}$, $\mathbb{R} \times [i-\frac{1}{4}, i+\frac{1}{2}]$, $\mathbb{R} \times [i+\frac{1}{4},i+ 1]$). Finally, for $i \in \mathbb{Z}$ and $j \in \left\{0,1\right\}$, we denote by $\tilde{\alpha}_{j}^{i}$ (respectively $\tilde{\beta}_{j}^{i}$) the curve $\left\{i+\frac{j}{2} \right\} \times \mathbb{R}$ (respectively $\mathbb{R} \times \left\{i+\frac{j}{2} \right\}$). Let $\mathcal{U}$ be the cover of the torus $\mathbb{T}^{2}$ defined by:
$$\begin{array}{rcl}
\mathcal{U} & = & \left\{ I \times J, \ I,J \in \left\{ \left[-\frac{1}{4},\frac{1}{2} \right], \left[\frac{1}{4}, 1 \right] \right\} \right\} \\
 & = & \left\{ A_{j}\cap B_{j'}, \ j,j' \in \left\{0,1 \right\} \right\}.
\end{array}$$
For a compact subset $A$ of $\mathbb{R}^{2}$, we set:
$$ \mathrm{length}(A)= \mathrm{card} \left\{ (i,j) \in \mathbb{Z} \times \left\{ 0, 1 \right\} , \ \tilde{\alpha}_{j}^{i} \cap A \neq \emptyset \right\}$$
and:
$$ \mathrm{height}(A)= \mathrm{card} \left\{ (i,j) \in \mathbb{Z} \times \left\{ 0, 1 \right\} , \ \tilde{\beta}_{j}^{i} \cap A \neq \emptyset \right\}.$$
Let us notice that, for any compact subset $A$ of $\mathbb{R}^{2}$:
$$ \left\{
\begin{array}{l}
\mathrm{length}(A) \leq 2 \mathrm{diam}_{\mathcal{D}}(A) \\
\mathrm{height}(A) \leq 2 \mathrm{diam}_{\mathcal{D}}(A)
\end{array}
\right.
.
$$
Let us fix a homeomorphism $g$ in $\mathrm{Homeo}_{0}(\mathbb{T}^{2})$ and a lift $\tilde{g}$ of $g$. Let $i_{max, \alpha} \in \mathbb{Z}$ and $j_{max,\alpha} \in \left\{0,1 \right\}$ (respectively $i_{\max, \beta}$ and $j_{max,\beta}$) be the integers which satisfy:
$$i_{max, \alpha}+ \frac{1}{2}j_{max, \alpha}= \max \left\{ i + \frac{1}{2}j, \ \tilde{g}(D_{0}) \cap \tilde{\alpha}_{j}^{i} \neq \emptyset \right\}$$ 
(respectively:
$$i_{max, \beta}+ \frac{1}{2}j_{max, \beta}= \max \left\{ i + \frac{1}{2}j, \ \tilde{g}(D_{0}) \cap \tilde{\beta}_{j}^{i} \neq \emptyset \right\}).$$

Let $(i_{\alpha},j_{\alpha})$ (respectively $(i_{\beta},j_{\beta})$) be the pair such that the interior of the band $\tilde{A}^{i_{\alpha}}_{j_{\alpha}}$ (respectively $\tilde{B}^{i_{\beta}}_{i_{\alpha}}$) contains the curve $\tilde{\alpha}^{i_{max,\alpha}}_{j_{max,\alpha}}=\tilde{\alpha}_{max}$ (respectively $\tilde{\beta}^{i_{max,\beta}}_{j_{max,\beta}}=\tilde{\beta}_{max}$).
Suppose that $\mathrm{height}(\tilde{g}(D_{0}))>3$ or that $\mathrm{length}(\tilde{g}(D_{0}))>3$. 
Notice that the connected components of $\mathring{A}_{j_{\alpha}} \cap g(\Pi(\partial D_{0}))$ can be split into two classes:
\begin{enumerate}
\item On the one hand, the connected components which are homeomorphic to $\mathbb{R}$ which will be called \emph{regular connected component} of $\mathring{A}_{j_{\alpha}} \cap g(\Pi(\partial D_{0}))$.
\item On the other hand, there exists at most one connected component homeomorphic to the union of two transverse straight lines in $\mathbb{R}^{2}$. This is the connected component which contains the point $g(0,0)$. We will call it \emph{singular connected component} of $\mathring{A}_{j_{\alpha}} \cap g(\Pi(\partial D_{0}))$.
\end{enumerate}
We claim that one of the following cases occurs.

\begin{itemize}
\item[\textbf{First case}.] There exists a connected component $\tilde{C}$ of $\Pi^{-1}(\mathring{A}_{j_{\alpha}}) \cap \tilde{g}(\partial D_{0})$ such that:
\begin{enumerate}
\item The ends of $\tilde{C}$ belong to two different connected component of the boundary of $\Pi^{-1}(A_{j_{\alpha}})$.
\item $ \mathrm{height}( \tilde{C}) \leq 3.$
\end{enumerate}

\item[\textbf{Second case}.] There exists a connected component $\tilde{C}$ of $\Pi^{-1}(\mathring{B}_{j_{\beta}}) \cap \tilde{g}(\partial D_{0})$ such that:
\begin{enumerate}
\item The ends of $\tilde{C}$ belong to two different connected components of the boundary of $\Pi^{-1}(B_{j_{\beta}})$.
\item $ \mathrm{length}( \tilde{C}) \leq 3.$
\end{enumerate}
\end{itemize}

Let us prove this claim. Suppose first that the length of $\tilde{g}(D_{0})$ is greater than $3$. Then, there exists a connected component $\tilde{C}$ of $\Pi^{-1}(\mathring{A}_{j_{\alpha}}) \cap \tilde{g}(\partial D_{0})$ whose ends belong to different boundary components of $\Pi^{-1}(A_{j_{\alpha}})$. If the first case does not occur, the height of $\tilde{C}$ is greater than $3$. Then, there exists a connected component $\tilde{C}'$ of $\mathring{B}_{j_{\beta}} \cap \tilde{C}$ whose ends belong to two different connected components of the boundary of $B_{j_{\beta}}$. In this case, the length of the component $\tilde{C}'$ is at most $1$: the second case occurs. Finally, if the length of $\tilde{g}(D_{0})$ is smaller or equal to $3$ and the height of this compact is greater than $3$, then any connected component of $\Pi^{-1}(\mathring{B}_{j_{\beta}}) \cap \tilde{g}(\partial D_{0})$ satisfies the properties of the second case.

The next lemmas will allow us to complete the proof of Proposition \ref{diamfrag} in the case of the $2$-dimensional torus.

\begin{lemma} \label{hertore}
In the first case above, there exists a homeomorphism $h$ supported in $A_{j_{\alpha}}$ which satisfies the following properties:
\begin{enumerate}
\item If $p_{2}: \mathbb{R}^{2} \rightarrow \mathbb{R}$ denotes the projection on the second coordinate, we have:
$$ \sup_{x \in \mathbb{R}^{2}} \left| p_{2} \circ \tilde{h}(x)-p_{2}(x) \right| <3.$$
\item $\mathrm{height}(\tilde{h}\circ \tilde{g}(D_{0})) \leq \mathrm{height}(\tilde{g}(D_{0}))$.
\item $\mathrm{length}(\tilde{h}\circ \tilde{g}(D_{0})) \leq \mathrm{length}(\tilde{g}(D_{0}))-1$.
\end{enumerate}
We have of course a symmetric statement in the second case.
\end{lemma}

\begin{lemma} \label{intore}
There exists a constant $C'>0$ such that, for any homeomorphism $g$ in $\mathrm{Homeo}_{0}(\mathbb{T}^{2})$ which satisfies the following properties:
$$ \left\{
\begin{array}{l}
\mathrm{length}(\tilde{g}(D_{0})) \leq 3 \\
\mathrm{height}(\tilde{g}(D_{0})) \leq 3
\end{array}
\right.
,$$
we have:
$$\mathrm{Frag}_{\mathcal{U}}(g) \leq C'.$$
\end{lemma}

\begin{proof}[Proof of Proposition \ref{diamfrag} in the case of the torus $\mathbb{T}^{2}$]
Using Proposition \ref{fragbord} in the case of the annulus, we see that there exists a constant $C>0$ such that, for any homeomorphism $h$ supported in $A_{j_{\alpha}}$ (respectively in $B_{j_{\beta}}$) with 
$$ \sup_{x \in \mathbb{R}^{2}} \left| p_{2} \circ \tilde{h}(x)-p_{2}(x) \right| <3$$
(respectively
$$ \sup_{x \in \mathbb{R}^{2}} \left| p_{1} \circ \tilde{h}(x)-p_{1}(x) \right| <3),$$
we have:
$$\mathrm{Frag}_{\mathcal{U}}(h)\leq C.$$
Using Lemma \ref{hertore}, we see that, after composing the homeomorphism $g$ with at most 
$$C . (\max(\mathrm{height}(\tilde{g}(D_{0}))-3,0)+ \max(\mathrm{length}(\tilde{g}(D_{0}))-3,0)$$
homeomorphisms supported in one of the discs of $\mathcal{U}$, we obtain a homeomorphism $f_{1}$ which satisfies the hypothesis of Lemma \ref{intore}:
$$\mathrm{Frag}_{\mathcal{U}}(f_{1}) \leq C'.$$
Therefore:
$$\mathrm{Frag}_{\mathcal{U}}(g) \leq 4 C \mathrm{diam}_{\mathcal{D}}(\tilde{g}(D_{0}))+C'.$$
The proposition is proved in the case of the torus $\mathbb{T}^{2}$.
\end{proof}

Let us now turn to the proof of the two above lemmas.

\begin{proof}[Proof of Lemma \ref{hertore}]
Suppose that the first case occurs (the proof in the second case is identical). Let $h$ be a homeomorphism supported in $A_{j_{\alpha}}$ which satisfies the following properties:
\begin{enumerate}
\item For any regular connected component $C$ of $g(\Pi(\partial D_{0})) \cap \mathring{A}_{j_{\alpha}}$ whose both ends belong to the same connected component of $A_{j_{\alpha}}$:
$$ h(C) \cap \Pi(\tilde{\alpha}^{i_{max,\alpha}}_{j_{max,\alpha}}) = \emptyset$$
and, if we denote by $\tilde{C}$ the lift of $C$ which is contained in $\tilde{g}(\partial D_{0})$ and by $q_{min}$ and $q_{max}$ the ends of $\tilde{C}$ with $p_{2}(q_{min}) < p_{2}(q_{max})$, then:
$$p_{2}(\tilde{h}(\tilde{C})) = [p_{2}(q_{min}),p_{2}(q_{max})].$$
\item The homeomorphism $h$ fixes the projection of any connected component of $\tilde{g}(\partial D_{0}) \cap \Pi^{-1}(\mathring{A}_{j_{\alpha}})$ whose ends belong to different connected components of the boundary of $\Pi^{-1}(A_{j_{\alpha}})$.
\item If the point $g(0,0)$ belongs to $\mathring{A_{j_{\alpha}}}$, we add the following condition. Let $C_{0}$ be the singular connected component of $g(\Pi(\partial D_{0})) \cap \mathring{A}_{j_{\alpha}}$. If there exists a lift $\tilde{C}_{0}$ of the component $C_{0}$ which meets $\tilde{g}(\partial D_{0})$ and the curve $\tilde{\alpha}_{max}$, we impose the following condition. Let us denote by $C_{1}$, $C_{2}$, $C_{3}$ and $C_{4}$ the connected components of $C_{0}-\left\{g(0,0)\right\}$. Only three of these connected components admit a lift contained in $\tilde{g}(D_{0})$ which meets the interior of $\tilde{A}^{i_{\alpha}}_{j_{\alpha}}$: for the last connected component, the two lifts of this one contained in $\tilde{g}(D_{0})$ are necessarily contained in the interior of $\tilde{A}^{i_{\alpha}-1}_{j_{\alpha}}$. We may suppose that these three connected components are $C_{1}$, $C_{2}$ and $C_{3}$. Let $\tilde{C}_{1}$, $\tilde{C}_{2}$ and $\tilde{C}_{3}$ be the respective lifts of $C_{1}$, $C_{2}$ and $C_{3}$ so that these three lifts share an end $\tilde{q}$ in common. For an integer $i$ between $1$ and $3$, let $\tilde{q}_{i}$ be the end of $\tilde{C}_{i}$ different from the point $\tilde{q}$. We may suppose that:
$$p_{2}(\tilde{q}_{1}) < p_{2}(\tilde{q}_{2})< p_{2}(\tilde{q}_{3}).$$
Then, for any integer $i$ between $1$ and $3$, we add the following condition:
$$h(C_{i}) \cap \tilde{\alpha}_{max}= \emptyset.$$
Moreover:
$$p_{2}(\tilde{h}(\tilde{C}_{1})) = [p_{2}(\tilde{q}_{1}), p_{2}(\tilde{q}_{2})],$$
$$p_{2}(\tilde{h}(\tilde{C}_{2})) = \left\{ p_{2}(\tilde{q}_{2})\right\},$$
$$p_{2}(\tilde{h}(\tilde{C}_{3})) = [p_{2}(\tilde{q}_{2}), p_{2}(\tilde{q}_{3})].$$
\end{enumerate}
We claim that such a homeomorphism $h$ satisfies the wanted properties. First, the existence of a connected component $\tilde{C}$ of $\Pi^{-1}(\mathring{A}_{j_{\alpha}}) \cap \tilde{g}(\partial D_{0})$ whose ends belong to two different connected components of the boundary of $\Pi^{-1}(A_{j_{\alpha}})$ and whose height is less than or equal to $3$ (and therefore $ \sup p_{2}(\tilde{C}) - \inf p_{2}(\tilde{C}) \leq 2$)  and the fact that the homeomorphism $h$ pointwise fixes the projection of this connected component imply that:
$$ \sup_{x \in \mathbb{R}^{2}} \left| p_{2} \circ \tilde{h}(x)-p_{2}(x) \right| <3.$$
The condition on the ordinates of the images by $h$ of the connected component of $\mathring{A}_{j_{\alpha}} \cap g(\Pi(\partial D_{0}))$ imply that:
$$ \mathrm{height}(\tilde{h}\circ \tilde{g}(D_{0})) \leq \mathrm{height}(\tilde{g}(D_{0})).$$
Finally, by construction, the set $\tilde{h}\circ \tilde{g}(D_{0})$ does not meet the curve $\tilde{\alpha}^{i_{max,\alpha}}_{j_{max,\alpha}}$ and meets only curves of the form $\tilde{\alpha}_{j}^{i}$ already met by the set $\tilde{g}(D_{0})$. Thus:
$$\mathrm{length}(\tilde{h}\circ \tilde{g}(D_{0})) \leq \mathrm{length}(\tilde{g}(D_{0}))-1.$$
Lemma \ref{hertore} is proved.
\end{proof}

\begin{proof}[Proof of Lemma \ref{intore}]
During this proof, we will often use the following result, which is a direct consequence of Proposition \ref{diamfrag2} in the case of the annulus. There exists a constant $\lambda>0$ such that, for any homeomorphism $\eta$ in $\mathrm{Homeo}_{0}(\mathbb{T}^{2})$ supported in $\mathring{A}_{0}$ or in $\mathring{A}_{1}$ which satisfies:
$$\mathrm{height}(\tilde{\eta}(D_{0})) \leq 12,$$
we have:
$$\mathrm{Frag}_{\mathcal{U}}(\eta) \leq \lambda.$$
First, notice that the inequality $\mathrm{length}(\tilde{g}(D_{0})) \leq 3$ implies the inequality $\mathrm{length}(\tilde{g}(\tilde{\alpha}_{0}^{0})) \leq 1$. Indeed, suppose that $\mathrm{length}(\tilde{g}(\tilde{\alpha}_{0}^{0}))>1$. As one of the edges of the square $\partial D_{0}$ is contained in $\tilde{\alpha}_{0}^{0}$ and as the curve $\tilde{g}(\tilde{\alpha}_{0}^{1})$ meets two curves among the $\tilde{\alpha}_{j}^{i}$ that $\tilde{g}(\tilde{\alpha}_{0}^{0})$ does not meet, we have:
$$ \mathrm{length}(\tilde{g}(D_{0})) \geq \mathrm{length}(\tilde{g}(\tilde{\alpha}_{0}^{0}))+2>3.$$
Now, let $g$ be a homeomorphism which satisfies the hypothesis of Lemma \ref{intore}. We denote by $n(\tilde{g}(\tilde{\alpha}_{0}^{0}))$ the number of connected components of $\bigcup \limits_{i,j} \partial \tilde{A}^{i}_{j}$ met by the path $\tilde{g}(\tilde{\alpha}_{0}^{0})$. As the length of $\tilde{g}(\tilde{\alpha}_{0}^{0})$ is less than or equal to $1$, then $n(\tilde{g}(\tilde{\alpha}_{0}^{0})) \leq 3 $. We will now prove that, after composing $g$ by a homeomorphism whose fragmentation length with respect to $\mathcal{U}$ is less than or equal to $3\lambda$ if necessary, we may suppose that $n(\tilde{g}(\tilde{\alpha}_{0}^{0}))=0$.

Suppose that $n(\tilde{g}(\tilde{\alpha}_{0}^{0}))>0$. Choose a pair $(i_{0},j_{0}) \in \mathbb{Z} \times \left\{ 0, 1 \right\}$ such that:
the set $\tilde{g}(D_{0})$ meets $\tilde{A}^{i_{0}}_{j_{0}}$ but meets only one connected component of the boundary of $\tilde{A}^{i_{0}}_{j_{0}}$ that we denote by $c_{i_{0},j_{0}}$. Let $\tilde{A}^{i_{1}}_{j_{1}}$ be the unique band among the $\tilde{A}^{i}_{j}$ whose interior contains the curve $c_{i_{0},j_{0}}$. Then: $j_{1} \neq j_{0}$.

For instance, the band $\tilde{A}^{i_{0}}_{j_{0}}$ can be the rightmost band met by the path $\tilde{g}(\tilde{\alpha}_{0}^{0})$.

\begin{itemize}
\item[\textbf{First case}.] We suppose that the set $\tilde{g}(D_{0})$ meets the two connected components of the boundary of $\tilde{A}^{j_{1}}_{i_{1}}$. Let $h$ be a homeomorphism in $\mathrm{Homeo}_{0}(\mathbb{T}^{2})$ supported in the interior of $A_{j_{0}}$ which satisfies the following properties:
\begin{enumerate}
\item For any connected component $\tilde{C}$ of $\tilde{g}(\partial D_{0}) \cap \Pi^{-1}(A_{j_{0}})$ which is not contained in the interior of $A_{j_{1}}$, we have:
$$\left\{
\begin{array}{l}
h(\Pi(\tilde{C})) \subset \mathring{A}_{j_{1}} \\
p_{2}(h(\Pi(\tilde{C}))) \subset p_{2}(\Pi(\tilde{C}))
\end{array}
\right.
.
$$
\item The homeomorphism $h$ pointwise fixes the other connected components of $g(\Pi(\partial D_{0})) \cap A_{j_{0}}$.
\item $\sup_{x \in \mathbb{R}^{2}} \left|p_{2} \circ \tilde{h}(x)-p_{2}(x) \right| <2$.
\end{enumerate}
Notice that the penultimate condition is compatible with the other ones. Indeed, as the height of $\tilde{g}(D_{0})$ is less than or equal to $3$, then, for any connected component $\tilde{C}$ of $\tilde{g}(\partial D_{0}) \cap \Pi^{-1}(A_{j_{0}})$, we have: $\mathrm{height}(\tilde{C}) \leq 3$. Therefore, we can choose $h$ so that the support of $h$ is contained in a disjoint union of discs which have a height less than or equal to three. For such a homeomorphism $h$, the following properties are satisfied:
$$ \left\{
\begin{array}{l}
\mathrm{Frag}_{\mathcal{U}}(h) \leq \lambda \\
n(\tilde{h} \circ \tilde{g}(\tilde{\alpha}^{0}_{0})) < n(\tilde{g}(\tilde{\alpha}^{0}_{0})) \\
\mathrm{height}( \tilde{h} \circ \tilde{g}(D_{0})) \leq \mathrm{height}(\tilde{g}(D_{0}))
\end{array}
\right.
.
$$
The second one comes from the fact that the set $\tilde{h}\circ \tilde{g}(\tilde{\alpha}_{0}^{0})$ does not meet anymore one of the connected components of the boundary of $\tilde{A}_{j_{1}}^{i_{1}}$.\\
\item[\textbf{Second case}.]  Suppose that the set $\tilde{g}(D_{0})$ does not meet the boundary of $\tilde{A}^{j_{1}}_{i_{1}}$. Likewise, we build a homeomorphism in $\mathrm{Homeo}_{0}(\mathbb{T}^{2})$ supported in $\mathring{A}_{j_{1}}$ such that the curve $\tilde{h} \circ \tilde{g}(\tilde{\alpha}_{0}^{0})$ does not meet the band $\tilde{A}_{j_{0}}^{i_{0}}$ anymore and such that:
$$ \left\{
\begin{array}{l}
\mathrm{Frag}_{\mathcal{U}}(h) \leq \lambda \\
n(\tilde{h} \circ \tilde{g}(\tilde{\alpha}^{0}_{0})) < n(\tilde{g}(\tilde{\alpha}^{0}_{0})) \\
\mathrm{height}( \tilde{h} \circ \tilde{g}(D_{0})) \leq \mathrm{height}(\tilde{g}(D_{0}))
\end{array}
\right.
.
$$
\end{itemize}

Thus, it suffices to prove the following property. There exists a constant $C>0$ such that, if $g$ is a homeomorphism in $\mathrm{Homeo}_{0}(\mathbb{T}^{2})$ with $n(\tilde{g}(\tilde{\alpha}_{0}^{0}))=0$ and $\mathrm{height}(\tilde{g}(D_{0})) \leq 3$, then $\mathrm{Frag}_{\mathcal{U}}(g) \leq C$. Let us consider such a homeomorphism $g$.
\begin{itemize}
\item[\textbf{First case}.] $g(\alpha_{0}) \nsubseteq A_{0}$. Let $h$ be a homeomorphism supported in the annulus $A_{1}$ which preserves the horizontal foliation such that: $h(g(\alpha_{0})) \subset A_{0}$. The preservation of this foliation implies that $\mathrm{Frag}_{\mathcal{U}}(h) \leq \lambda$. We are led to the second case.
\item[\textbf{Second case}.] $g(\alpha_{0}) \subset A_{0}$. Let $h$ be a homeomorphism supported in the annulus $A_{0}$ which is equal to the homeomorphism $g$ in a neighbourhood of the curve $\alpha_{0}$. As the height of $\tilde{g}(D_{0})$ is less than or equal to $3$, we may suppose moreover that: $\mathrm{height}(\tilde{h}(D_{0})) \leq 3$, because we may suppose that $\sup_{x \in \mathbb{R}^{2}} \left\| \tilde{h}(\tilde{x})-\tilde{x} \right\|<2$. Thus: $\mathrm{Frag}_{\mathcal{U}}(h) \leq \lambda$. Moreover:
$$\mathrm{height}(\tilde{h}^{-1} \circ \tilde{g}(D_{0})) \leq 6.$$
\end{itemize}

We have pointwise fixed $\alpha$ which is one of the boundary components of $A_{1}$. By an analogous procedure, we can find a homeomorphism $h'$ such that $h'^{-1} \circ h^{-1} \circ g$ pointwise fixes a neighbourhood of the boundary of $A_{1}$ and such that:
$$ \left\{
\begin{array}{l}
\mathrm{Frag}_{\mathcal{U}}(h') \leq 2\lambda \\
\mathrm{height}(\tilde{h}'^{-1} \circ \tilde{h}^{-1} \circ \tilde{g}(D_{0}))\leq 12
\end{array}
\right.
.
$$
We denote by $h_{1}$ the homeomorphism supported in $A_{1}$ which is equal to $h'^{-1} \circ h^{-1} \circ g$ on $A_{1}$. The height of $\tilde{h}_{1}(D_{0})$ is less than or equal to $12$ and that is why: $\mathrm{Frag}_{\mathcal{U}}(h_{1}) \leq \lambda$. Moreover, the homeomorphism $h_{2}=h_{1}^{-1} \circ h'^{-1} \circ h^{-1} \circ g$ is supported in $A_{2}$. The image of $D_{0}$ under $\tilde{h}_{2}$ is less than or equal to $12$: $\mathrm{Frag}_{\mathcal{U}}(h_{2}) \leq \lambda$. Finally, $\mathrm{Frag}_{\mathcal{U}}(g) \leq 5 \lambda$ in this case.
\end{proof}

\section{Case of higher genus closed surfaces}

In this section, we prove Proposition \ref{diamfrag} for a higher genus closed surface $S$. Let us begin by describing the cover $\mathcal{U}$ that we use in what follows. Let $p$ be the point of $S$ which is the image under $\Pi$ of a vertex of the polygon $\partial D_{0}$. Let us denote by $\tilde{A}$ the set of edges of the polygon $\partial D_{0}$ and by $A$ the set of curves which are the images under $\Pi$ of an edge in $\tilde{A}$. Let:
$$B= \left\{ \gamma(\alpha), \ \left\{
\begin{array}{l}
\alpha \in \tilde{A} \\
\gamma \in \Pi_{1}(S) 
\end{array}
\right.
\right\}= \Pi^{-1}(\Pi(\tilde{A})).$$
We denote by $U_{0}$ a closed disc of $S$ whose interior contains the point $p$ and which satisfies the following property: if $\tilde{U}_{0}$ is a lift of $U_{0}$ and $\tilde{p}$ is a lift of the point $p$, then the disc $\tilde{U}_{0}$ meets only edges in $B$ for which one end is $\tilde{p}$ and the boundary $\partial \tilde{U}_{0}$ meets each of them in exactly one point.
For any edge $\alpha$ in $A$, we denote by $V_{\alpha}$ a closed disc which does not contain the point $p$ so that the following properties are satisfied:
\begin{enumerate}
\item For any edge $\alpha$ in $A$, the set $V_{\alpha} \cup U_{0}$ is a neighbourhood of the edge $\alpha$.
\item For any edge $\alpha$ in $A$, the set $V_{\alpha} \cap U_{0}$ is the disjoint union of two closed discs.
\item The discs $V_{\alpha}$ are pairwise disjoint.
\end{enumerate}
We denote by $U_{1}$ a closed disc which contains the union of the $V_{\alpha}$. Finally, we denote by $U_{2}$ a closed disc which does not meet any edge in $A$ and which satisfies the following properties:
\begin{enumerate}
\item For any edge $\alpha$ in $A$, the closed set $U_{2}\cap V_{\alpha}$ is homeomorphic to the disjoint union of two closed discs.
\item The union of the interior of the disc $U_{2}$ with the interior of the disc $U_{0}$ and with the interiors of the discs $V_{\alpha}$ is equal to $S$.
\item The closed set $(\bigcup_{\alpha} V_{\alpha} \cup U_{2})\cap U_{0}$ is homeomorphic to an annulus for which one component of the boundary is $\partial U_{0}$.
\end{enumerate}
Let $\mathcal{U}= \left\{U_{0}, U_{1}, U_{2} \right\}$. The following lemmas will allow us to complete the proof of Proposition \ref{diamfrag}.

\begin{figure}[ht]
\begin{center}
\includegraphics{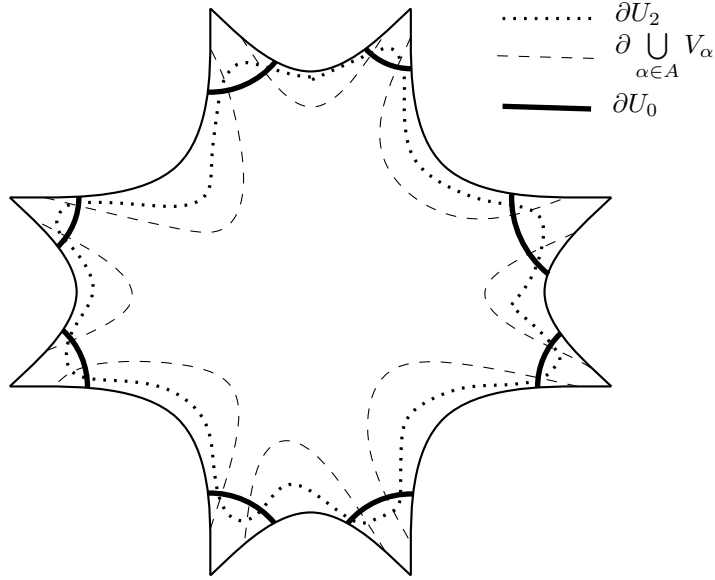}
\end{center}
\caption{Notations in the case of higher genus closed surfaces}
\end{figure}

\begin{lemma} \label{herhyp} Let $f$ be a homeomorphism in $\mathrm{Homeo}_{0}(S)$. Suppose that $\mathrm{el}_{D_{0}}(\tilde{f}(D_{0}))\geq 4g$. Then there exists a homeomorphism $h$ in $\mathrm{Homeo}_{0}(S)$ which satisfies the following properties:
\begin{enumerate}
\item $\mathrm{Frag}_{\mathcal{U}}(h) \leq 8g-2$.
\item $\mathrm{el}_{D_{0}}(\tilde{h} \circ \tilde{f}(D_{0})) \leq \mathrm{el}_{D_{0}}(\tilde{f}(D_{0}))-1$.
\end{enumerate}
\end{lemma}

\noindent \textbf{Remark} We did not try to obtain an optimal upper bound of the fragmentation length of a homeomorphism with $\mathrm{el}_{D_{0}}(\tilde{h} \circ \tilde{f}(D_{0})) \leq \mathrm{el}_{D_{0}}(\tilde{f}(D_{0}))-1$.

\begin{lemma} \label{inhyp}
There exists a constant $C'>0$ such that, for any homeomorphism $f$ in $\mathrm{Homeo}_{0}(S)$ with $\mathrm{el}_{D_{0}}(\tilde{f}(D_{0})) \leq 4g$, we have:
$$ \mathrm{Frag}_{\mathcal{U}}(f)\leq C'.$$
\end{lemma}

\begin{proof}[End of the proof of Proposition \ref{diamfrag}. Case of a higher genus closed surface.]
By the Lefschetz fixed point theorem, the homeomorphism $f$ has a contractible fixed point. Hence:
$$\tilde{f}(D_{0})\cap D_{0} \neq \emptyset$$
and
$$\mathrm{el}_{D_{0}}(\tilde{f}(D_{0})) \leq diam_{\mathcal{D}}(\tilde{f}(D_{0})).$$
Therefore, the two above lemmas allow us to complete the proof of Proposition \ref{diamfrag}.
\end{proof}

To complete the proof of Lemma \ref{herhyp}, we will need some combinatorial lemmas concerning the group $\Pi_{1}(S)$ which we state in the following subsection.

\subsection{Some combinatorial lemmas}

Recall that two fundamental domains $D_{1}$ and $D_{2}$ in $\mathcal{D}$ are \emph{adjacent} if the intersection of $D_{1}$ with $D_{2}$ is an edge common to the polygons $\partial D_{1}$ and $\partial D_{2}$. Recall also that $\mathcal{G}$ is a generating set of $\Pi_{1}(S)$ consisting of deck transformations which send the fundamental domain $D_{0}$ to a fundamental domain adjacent to $D_{0}$.

We call \emph{geodesic word} a word $\gamma$ whose letters are elements of $\mathcal{G} \subset \Pi_{1}(S)$ such that the length of the word $\gamma$ is equal to $l_{\mathcal{G}}(\gamma)$ (by abuse of notation, we also denote by $\gamma$ the image of the word $\gamma$ in the group $\Pi_{1}(S)$). 

We now describe a more geometric way to see the words whose letters are elements of $\mathcal{G}$. We call \emph{path in $\mathcal{D}$ of origin $D_{0}$} any finite sequence $(D_{0},D_{1}, \ldots, D_{p})$ of fundamental domains in $\mathcal{D}$ such that two consecutive fundamental domains in this sequence are adjacent. Such a path in $\mathcal{D}$ is said to be \emph{geodesic} if, moreover, for any index $i$, $d_{\mathcal{D}}(D_{0},D_{i})=i$. Notice that there is a bijective map between words on the elements of $\mathcal{G}$ and the paths of origin $D_{0}$ in $\mathcal{D}$: to a word $l_{1} \ldots l_{p}$, one can associate the path $(D_{0}, l_{1}(D_{0}), l_{1}l_{2}(D_{0}), \ldots, l_{1}l_{2} \ldots l_{p}(D_{0}))$. This last application is a bijective map and sends the geodesic words to geodesic paths in $\mathcal{D}$.

For a homeomorphism $h$ in $\mathrm{Homeo}_{0}(S)$, we call \emph{maximal face} for $h$ any fundamental domain in $\mathcal{D}$ at distance $\mathrm{el}_{D_{0}}(\tilde{h}(D_{0}))$ from $D_{0}$. We want to prove that, after a composing of $h$ with a number independent from $h$ of homeomorphisms supported in one of the discs in $\mathcal{U}$, the image of $D_{0}$ does not meet maximal faces for $h$ anymore. There will be two different kinds of maximal faces for $h$. The first ones, which we call \emph{non-exceptional}, are not problematic: after a composing $h$ with four homeomorphisms, each of them being supported in one of the discs of $\mathcal{U}$, the image of the fundamental domain $D_{0}$ will not meet these faces anymore. These faces are the ones which satisfy the following property: in the set of faces adjacent to $D$, there is only one element which is at distance $d_{\mathcal{D}}(D,D_{0})-1$ from $D_{0}$. The faces in $\mathcal{D}$ which do not satisfy this property are called \emph{exceptional}. In their case, we will have to understand the relative arrangement of the nearby fundamental domains in $\mathcal{D}$. 

Let us describe more precisely the crucial property used in this proof. Let us denote by $D$ an exceptional face and by $\gamma$ a geodesic word such that $\gamma(D_{0})=D$. Let $(D_{0},D_{1}, \ldots, D_{M}=D)$ be the geodesic path in $\mathcal{D}$ corresponding to the geodesic word $\gamma$. We will see later(see Lemma \ref{adjacence}) that the $2g-1$ last faces in this sequence share a vertex in common. The crucial property is the following: \emph{if $1 \leq k \leq 2g-2$, for any geodesic path of the form $(D_{0}, \ldots, D_{M-k}, D'_{M-k+1}, \ldots, D'_{M})$, where the face $D'_{M-k+1}$ is different from the face $D_{M-k+1}$, then the faces $D'_{M-k+1}, \ldots, D'_{M}$ are not exceptional} (see Lemma \ref{faceexc}).

By replacing the face $D_{0}$ with any other fundamental domain $D_{1}$ in $\mathcal{D}$ and the generating set $\mathcal{G}$ by the generating set consisting of deck transformations which send $D_{1}$ to a face adjacent to $D_{1}$, we can define the notion of exceptional faces with respect to $D_{1}$. All the following statements dealing with exceptional faces (with respect to $D_{0}$) can be generalized to the case of an exceptional face with respect to any fundamental domain in $\mathcal{D}$. We implicitly use this remark during the proof of Lemma \ref{geodexc2}. 

Let:
$$\mathcal{G}=\left\{ a_{i}^{\epsilon}, \ 1 \leq i \leq g \mbox{ et } \epsilon \in \left\{-1,1 \right\} \right\} \cup \left\{ b_{i}^{\epsilon}, \ 1 \leq i \leq g \mbox{ et } \epsilon \in \left\{-1,1 \right\} \right\}$$
so that:
$$\Pi_{1}(S)= \left< (a_{i})_{1 \leq i \leq g}, \ (b_{i})_{1 \leq i \leq g} | [a_{1},b_{1}] \ldots [a_{g},b_{g}]=1 \right>.$$
Let us denote by $\Lambda$ the set of cyclic permutations of the words $[a_{1},b_{1}] \ldots [a_{g},b_{g}]$ and $[b_{g},a_{g}] \ldots [b_{1},a_{1}]$. In terms of paths in $\mathcal{D}$, these words correspond to a circle around one of the vertices of the polygon $\partial D_{0}$:

\begin{lemma} \label{adjacence} For any face $D$ in $\mathcal{D}$ and any word $\lambda_{1} \ldots \lambda_{4g}$ in $\Lambda$, the faces $\lambda_{1} \ldots \lambda_{i}(D)$, for $1 \leq i \leq 4g$, share a point in common.
\end{lemma}

\begin{proof}
We prove that, given a word $\lambda$ in $\Lambda$, the fundamental domains $\lambda_{1} \ldots \lambda_{i}(D_{0})$, for $1 \leq i \leq 4g$, share a point in common. This last property implies the lemma thanks to the transitivity of the action of the group $\Pi_{1}(S)$ on the set $\mathcal{D}$.

Let us denote by $X$ the set of $4g$-tuples $(\delta_{i})_{1 \leq i \leq 4g}$ of elements of $\mathcal{D}$ which satisfy the following properties:
\begin{enumerate}
\item $\delta_{4g}=D_{0}$.
\item There exists a vertex $\tilde{p}$ of $D_{0}$ such that the set of elements of $\mathcal{D}$ which contain the point $\tilde{p}$ is $\left\{ \delta_{i}, \ 1 \leq i \leq 4g \right\}$.
\item Any circle centered at $\tilde{p}$ of sufficiently small diameter and counterclockwise oriented meets successively the fundamental domains $\delta_{1}, \ldots, \ \delta_{4g}$ in this order. In particular, the faces $\delta_{i}$ and $\delta_{i+1}$ are adjacent.
\end{enumerate}
The set $X$ is naturally isomorphic to the set of vertices of the polygon $\partial D_{0}$. An element $a=(\delta_{i})_{1 \leq i \leq 4g}$ in $X$ is associated to a word $\varphi(a)=\lambda=\lambda_{1} \ldots \lambda_{4g}$ in $\Lambda$ defined in the following way: the letter $\lambda_{1}$ is the unique deck transformation in $\mathcal{G}$ which sends $D_{0}$ to $\delta_{1}$. The second letter $\lambda_{2}$ is the unique deck transformation in $\mathcal{G}$ such that $\lambda_{1} \lambda_{2}(D_{0})=\delta_{2}$. Likewise, if we suppose that we have built the letters $\lambda_{1}, \ldots, \ \lambda_{i}$ such that $\lambda_{1} \ldots \lambda_{i}(D_{0})= \delta_{i}$, the letter $\lambda_{i+1}$ is defined by the relation $\lambda_{1} \ldots \lambda_{i+1}(D_{0})= \delta_{i+1}$. Finally: $\delta_{1} \ldots \delta_{4g}(D_{0})=D_{0}$ so the word $\delta_{1} \ldots \delta_{4g}$ belongs to the set $\Lambda$.

Thus, we have built an injective map which, to any vertex $\tilde{p}$ of $D_{0}$, associates a word $\lambda$ in $\Lambda$ such that the fundamental domains $\lambda_{1} \ldots \lambda_{i}(D_{0})$, for $1 \leq i \leq 4g$, share the point $\tilde{p}$ in common. Notice that the word $\lambda^{-1}$ satisfies also this property. Moreover, as the cardinality of the set $\Lambda$ is $4g$ and as the cardinality of the set of vertices of the polygon $\partial D_{0}$ is $2g$, we obtain the following property: for a word $\lambda$ in $\Lambda$, the fundamental domains $\lambda_{1} \ldots \lambda_{i}(D_{0})$, for $1 \leq i \leq 4g$, share a point in common.
\end{proof}

The next lemma describes the shape of the geodesic words which send the face $D_{0}$ to an exceptional face. This lemma, as well as the following combinatorial lemmas, are proved at the end of this section.

\begin{lemma} \label{geodexc}
Let $D$ be an exceptional face different from $D_{0}$. For any geodesic word $\gamma$ with $\gamma(D_{0})=D$, one of the following properties holds:
\begin{enumerate}
\item The $2g$ last letters of the word $\gamma$ are a subword of a word of $\Lambda$.
\item The $4g-1$ last letters of $\gamma$ are the concatenation of two subwords $\lambda_{1}$ and $\lambda_{2}$ with respective length $2g$ and $2g-1$ of words of $\Lambda$ such that, if we denote by $a$ the last letter of $\lambda_{1}$ and by $b$ the first letter of $\lambda_{2}$, then the word $ab$ is not contained in any word in $\Lambda$.
\end{enumerate}
Moreover, there exists a geodesic word $\gamma$ such that $\gamma(D_{0})=D$ which satisfies the first property above. We denote by $l_{1} \ldots l_{2g}$ its $2g$ last letters, where $l_{1} \ldots l_{4g} \in \Lambda$. Moreover, the $2g$ first letters of any geodesic word for which this first property holds are $l_{1} \ldots l_{2g}$ or $l_{4g}^{-1} \ldots l_{2g+1}^{-1}$.
\end{lemma}

In the case $g=2$, an example of a geodesic word associated to an exceptional face with the first property above is $[a_{1},b_{1}]=[b_{2},a_{2}]$ and an example of a geodesic word associated to an exceptional face with the second property above is 
$$\begin{array}{rcl}
a_{2}^{-1}b_{2}^{-1}a_{1}b_{1}^{2}a_{1}^{-1}b_{1}^{-1} & = & a_{2}^{-1}b_{2}^{-1}a_{1}b_{1}a_{1}^{-1}[a_{1},b_{1}]\\
 & = & b_{2}^{-1}a_{2}^{-1}b_{1} [b_{2},a_{2}].
\end{array}
$$
The first property holds for this last word.

Let us fix an exceptional face $D$. Let $l_{1} \ldots l_{4g}$ be a word in $\Lambda$ and $\gamma$ be a geodesic word whose $2g$ last letters are $l_{1} \ldots l_{2g}$ such that $\gamma(D_{0})=D$. Let $\gamma=\gamma' l_{1} \ldots l_{2g}$ and, for $0 \leq i \leq 2g$:
$$ \left\{
\begin{array}{l}
D_{i}^{1}=\gamma' l_{1}  \ldots l_{2g-i}(D_{0}) \\
D_{i}^{2}=\gamma' l_{4g}^{-1} \ldots l_{2g+i+1}^{-1}(D_{0})
\end{array}
\right.
.
$$
Then: $D_{0}^{1}=D_{0}^{2}=D$ et $D^{1}_{2g}=D^{2}_{2g}$. By Lemma \ref{adjacence}, all the fundamental domains that we just defined meet in one point: they are the elements of the set of fundamental domains in $\mathcal{D}$ which contain this point.

\begin{figure}[ht]
\begin{center}
\includegraphics{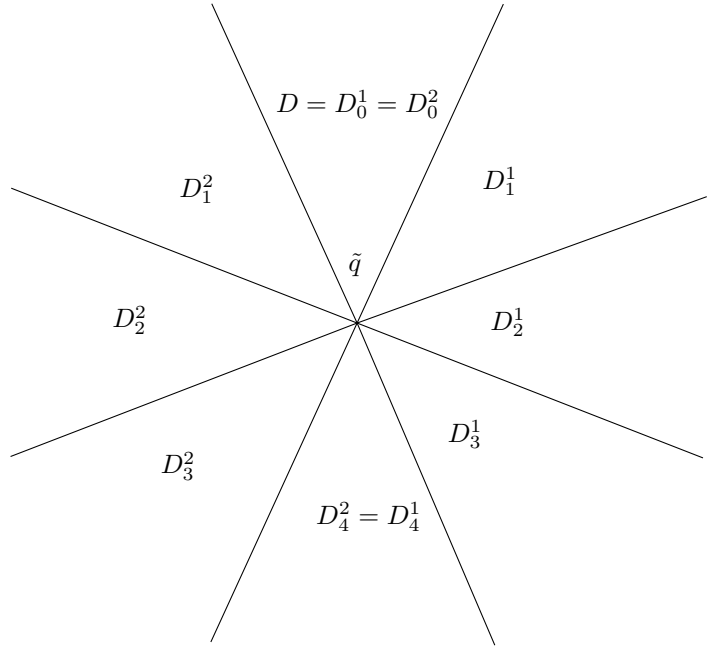}
\end{center}
\caption{The $D_{i}^{j}$'s for a genus $2$ surface}
\end{figure}

For a natural number $l \geq 1$, we call \emph{face of type $(0,l)$} any fundamental domain $D$ in $\mathcal{D}$ which is at distance $l$ from $D_{0}$ and which satisfies the following property: in the set of faces adjacent to $D$, only one element is at distance $l-1$ from $D_{0}$, \emph{i.e.} this face is not exceptional and is at distance $l$ from $D_{0}$. In this case, the other faces adjacent to $D$ are at distance $l+1$ from the fundamental domain $D_{0}$. This last fact is a consequence of the following remark: if we denote by $m$ a word on elements of $\mathcal{G}$ and by $l$ a letter in $\mathcal{G}$, the elements $ml$ and $l$ in the group $\Pi_{1}(S)$ do not have the same length $l_{\mathcal{G}}$ modulo $2$ as the relations which define this group have even length. By using the notion of geodesic word, another (equivalent) definition of faces of type $(0,l)$ can be given: a face of type $(0,l)$ is a fundamental domain $D$ in $\mathcal{D}$ such that all the geodesic words $\gamma$ with $\gamma(D_{0})=D$ have the same last letter and their length is $l$.

For any integer $k$ between $0$ and $l$, we define by induction the set of faces of types $(k,l)$. A \emph{face of type $(k,l)$} is a fundamental domain $D$ in $\mathcal{D}$ which is at distance $l-k$ from $D_{0}$ and which satisfies the following property: all the faces adjacent to $D$, except one, are faces of type $(k-1,l)$. Therefore, a face of type $(k,l)$ is also a face of type $(0,l-k)$ (or even $(k-i,l-i)$, for $0 \leq i \leq k$). An equivalent definition of faces of type $(k,l)$ is the following. Let us consider a geodesic word $\gamma'$ of length $l-k$ such that $\gamma'(D_{0})=D$. The face $D$ is a face of type $(k,l)$ if and only if, for any reduced word $m$ with length less than or equal to $k$ such that the word $\gamma'm$ is reduced, the face $\gamma'm(D_{0})$ is not exceptional. This definition can also be interpreted in terms of geodesic paths in $\mathcal{D}$. Let us denote by $(D_{0}, \ldots, D_{l-k})$ a geodesic path in $\mathcal{D}$. The fundamental domain $D_{l-k}$ is a face of type $(k,l)$ if and only if for any geodesic extension of the form $(D_{0}, \ldots, D_{l-k}, D_{l-k+1}, \ldots, D_{l})$ of this last path, the faces $D_{l-k}, \ldots, D_{l}$ are not exceptional. The crucial property described above can be translated in the following way: for any exceptional face $D$, for any integer $1 \leq j \leq 2g-2$, the faces adjacent to $D_{j}^{1}$ and different from $D_{j-1}^{1}$ and $D_{j+1}^{1}$ are faces of type $(j-1, d_{\mathcal{D}}(D,D_{0}))$. Notice that the face $D_{j}^{1}$ is not a face of type $(j,d_{\mathcal{D}}(D,D_{0}))$ as the face $D$, which is exceptional, is at distance $j$ from $D$.

The following lemma will play a crucial role in the proof of Lemma \ref{herhyp} and is deduced from Lemma \ref{geodexc}.

\begin{lemma} \label{faceexc}
For any indices $i$ between $1$ and $2g-2$ and $j \in \left\{ 1,2 \right\}$, the fundamental domains adjacent to $D_{i}^{j}$ which are different from $D_{i+1}^{j}$ and from $D_{i-1}^{j}$ are faces of type $(i-1, d_{\mathcal{D}}(D_{0},D))$.
\end{lemma}

The next lemma is symmetric to Lemma \ref{geodexc}.

\begin{lemma} \label{geodexc2} Let $D_{1}$ be a fundamental domain in $\mathcal{D}$. Suppose that there exist two geodesic words with distinct first letters $a$ and $b$ such that:
$$ \gamma_{1}(D_{0})=\gamma_{2}(D_{0})=D_{1}.$$
In this case, the fundamental domain $D_{0}$ is an exceptional face with respect to $D_{1}$.
Then there exists a geodesic word $\gamma$ such that $\gamma(D_{0})=D_{1}$ whose $2g$ first letters $\lambda_{1}  \ldots \lambda_{2g}$ are a subword of a word $\lambda_{1} \ldots \lambda_{4g}$ in $\Lambda$. Moreover, the fundamental domains $D_{0}$, $a(D_{0})$ and $b(D_{0})$ share a point $\tilde{p}$ with the following property in common: the fundamental domains in $\mathcal{D}$ which contain the point $\tilde{p}$ are faces of the form $\lambda_{1} \ldots \lambda_{i}(D_{0})$ or $\lambda_{4g}^{-1} \ldots \lambda_{4g-i+1}^{-1}(D_{0})$, with $0 \leq i \leq 2g$.
\end{lemma}

For a homeomorphism $h$ in $\mathrm{Homeo}_{0}(S)$, we denote by $l(h)$ the maximum of the $d_{\mathcal{D}}(D, D_{0})$, where $D$ is a fundamental domain in $\mathcal{D}$ which contains the image under the homeomorphism $\tilde{h}$ of a vertex of the polygon $\partial D_{0}$.

\begin{lemma} \label{sommets} Let $h$ be a homeomorphism in $\mathrm{Homeo}_{0}(S)$. Suppose that there exists a fundamental domain $D_{1}$ in $\mathcal{D}$ whose interior contains the image under $\tilde{h}$ of a vertex $\tilde{p}$ of the polygon $\partial D_{0}$. Then the following assertions are equivalent:
\begin{enumerate}
\item $d_{\mathcal{D}}(D_{1},D_{0})=l(h)$.
\item The fundamental domain $D_{0}$ is an exceptional face with respect to $D_{1}$.
\end{enumerate}
Then the face $D_{1}$ is unique among the faces which satisfy the properties above. In this case, there exists a word $\lambda_{1} \lambda_{2} \ldots \lambda_{4g}$ in $\Lambda$ and a geodesic word $\gamma$ such that $\gamma(D_{0})=D_{1}$ and the $2g$ first letters of $\gamma$ are $\lambda_{1} \lambda_{2} \ldots \lambda_{2g}$: $\gamma=\lambda_{1} \lambda_{2} \ldots \lambda_{2g} \gamma'$. Moreover, the vertices of the polygon $\partial D_{0}$ are the points of the form $\tilde{p}_{i}= \lambda_{i}^{-1} \lambda_{i-1}^{-1} \ldots \lambda_{1}^{-1}(\tilde{p})$ or $\tilde{p}'_{i}= \lambda_{4g-i+1} \lambda_{4g-i+2} \ldots \lambda_{4g}(\tilde{p})$. These points are pairwise distinct except in the two following cases: $\tilde{p}'_{0}=\tilde{p}_{0}=\tilde{p}$ and $\tilde{p}_{2g}=\tilde{p}'_{2g}$.
\end{lemma}

Let us come now to the proof of Lemma \ref{herhyp}. 

\subsection{Proof of Lemma \ref{herhyp}}
\begin{proof}[Proof of Lemma \ref{herhyp}] Let $f$ be a homeomorphism in $\mathrm{Homeo}_{0}(S)$ such that $\mathrm{el}_{D_{0}}(\tilde{f}(D_{0})) \geq 4g$. The proof is decomposed into two parts. First, we build a homeomorphism $\eta_{1}$ so that the set $\tilde{\eta}_{1} \circ \tilde{f} (D_{0})$ does not meet faces of type $(i, \mathrm{el}_{D_{0}}(\tilde{f}(D_{0})))$ for $0 \leq i \leq 2g-2$ anymore. Then, we build a homeomorphism $\eta_{2}$ so that the set $\tilde{\eta}_{2} \circ \tilde{\eta}_{1} \circ \tilde{f}(D_{0})$ does not meet either exceptional maximal faces for $f$. In these constructions, we will make sure that the quantities $\mathrm{Frag}_{\mathcal{U}}(\eta_{i})$ are bounded by a constant independent from the chosen homeomorphism $f$. Let us give more details now.

\begin{lemma} \label{nettoyage}
Let $h$ be a homeomorphism in $\mathrm{Homeo}_{0}(S)$. Suppose that $\mathrm{el}_{D_{0}}(\tilde{h}(D_{0})) \geq 4g$. Then there exists a homeomorphism $\eta$ in $\mathrm{Homeo}_{0}(S)$ such that:
\begin{enumerate}
\item $\mathrm{Frag}_{\mathcal{U}}(\eta) \leq 4 (2g-2)+1$.
\item $\mathrm{el}_{D_{0}}(\tilde{\eta} \circ \tilde{h}(D_{0})) \leq \mathrm{el}_{D_{0}}(\tilde{h}(D_{0}))$.
\item One of the following properties holds:
\begin{enumerate}
\item $\mathrm{el}_{D_{0}}(\tilde{\eta} \circ \tilde{h}(D_{0})) \leq \mathrm{el}_{D_{0}}(\tilde{h}(D_{0}))-1$.
\item The set $\tilde{\eta} \circ \tilde{h}(D_{0})$ does not meet faces of type $(i, \mathrm{el}_{D_{0}}(\tilde{h}(D_{0})))$ for $0 \leq i \leq 2g-2$.
\end{enumerate}
\end{enumerate}
\end{lemma}

\begin{proof} 
Notice that there are two kinds of connected components of $h(\Pi(\partial D_{0}))-\Pi(\partial D_{0})$: the connected components homeomorphic to $\mathbb{R}$ which will be called \emph{regular} and (if the image under $h$ of the vertex of $\Pi(\partial D_{0})$ does not belong to $\Pi(\partial D_{0})$, what is assumed in the lemmas below) a connected component called \emph{singular} homeomorphic to the union of pairwise transverse $2g$ straight lines of the plane which meet in one point. This last connected component is the one which contains the vertex of $\Pi(\partial D_{0})$. This last kind of component will raise technical issues and will require lemmas throughout the proof. The reader may skip the lemmas which deal with this singular component on a first reading. The following lemma is one of those.

\begin{lemma} \label{sommets2}
Let $h$ be a homeomorphism in $\mathrm{Homeo}_{0}(S)$. Take an integer $j$ in $[0,2g-2]$. Suppose that the following properties hold:
\begin{enumerate}
\item $\mathrm{el}_{D_{0}}(\tilde{h}(D_{0})) \geq 4g$.
\item The point $h(p)$ does not belong to the set $\Pi(\partial D_{0})$.
\item The set $\tilde{h}(D_{0})$ does not meet faces of type $(i, \mathrm{el}_{D_{0}}(\tilde{h}(D_{0})))$ for $0 \leq i <j$.
\item The image under $\tilde{h}$ of a vertex $\tilde{p}$ of the polygon $\partial D_{0}$ belongs to a face $D_{1}$ of type $(j, \mathrm{el}_{D_{0}}(\tilde{h}(D_{0})))$.
\end{enumerate}
In this case, the image under the homeomorphism $\tilde{h}$ of any vertex of the polygon $\partial D_{0}$ different from $\tilde{p}$ does not belong to a face of type $(j, \mathrm{el}_{D_{0}}(\tilde{h}(D_{0})))$. Moreover, the face $D_{0}$ is exceptional with respect to $D_{1}$.
\end{lemma}

\begin{proof}
Suppose first that $j=0$. Lemma \ref{sommets} implies that the images under the homeomorphism $\tilde{h}$ of the other vertices of the polygon $\partial D_{0}$ belong to fundamental domains in $\mathcal{D}$ strictly closer to $D_{0}$ than $D_{1}$. Suppose now that $j \geq 1$. We prove by contradiction that the face $D_{1}$ is exceptional with respect to $D_{0}$. Denote by $s(D_{0})$, where $s$ is a deck transformation in $\mathcal{G}$, a face adjacent to $D_{0}$ which contains the point $\tilde{p}$. Suppose by contradiction that $d_{\mathcal{D}}(s(D_{0}),D_{1})=d_{\mathcal{D}}(D_{0},D_{1})+1$. Then:
$$ \left\{
\begin{array}{l}
d_{\mathcal{D}}(D_{0},s^{-1}(D_{1}))=d_{\mathcal{D}}(D_{0},D_{1})+1 \\
\tilde{h}(s^{-1}(\tilde{p})) \in s^{-1}(D_{1})
\end{array}
\right.
.
$$
Let us prove that the fundamental domain $s^{-1}(D_{1})$ is a face of type $(j-1, \mathrm{el}_{D_{0}}(\tilde{h}(D_{0})))$. Let $\gamma$ be a geodesic word such that $\gamma (D_{0})=D_{1}$. As $\mathrm{el}_{D_{0}}(\tilde{h}(D_{0}))\geq 4g$, the length of the word $\gamma$ is greater than or equal to $2g$. Moreover, as $d_{\mathcal{D}}(s(D_{0}),D_{1})=d_{\mathcal{D}}(D_{0},D_{1})+1$, the word $s^{-1} \gamma$ is geodesic. If we concatenate  $i \in [0,j]$ letters $a_{1},a_{2}, \ldots, a_{i}$ on the right with $\gamma$ so that the word $\gamma a_{1} a_{2} \ldots a_{i}$ is reduced, then the $2g$ last letters of the obtained word are not a subword of a word in $\Lambda$, as the fundamental domain $D_{1}$ is a face of type $(j,\mathrm{el}_{D_{0}}(\tilde{h}(D_{0})))$. Therefore, if we concatenate $i \in [0,j-1]$ letters $a_{1},a_{2}, \ldots, a_{i}$ on the right with the geodesic word $s^{-1} \gamma$ so that the obtained word is reduced, the $2g-1$ last letters of the obtained word are not a subword of a word in $\Lambda$. By Lemma \ref{geodexc}, the faces $s^{-1} \gamma a_{1} a_{2} \ldots a_{i}(D_{0})$ are not exceptional so the face $s^{-1}(D_{1})$ is a face of type $(j-1, \mathrm{el}_{D_{0}}(\tilde{h}(D_{0})))$. This contradicts the hypothesis of the lemma.

Thus, the face $D_{0}$ is exceptional with respect to $D_{1}$ and, using Lemma \ref{sommets}, we see that the images under the homeomorphism $\tilde{h}$ of the vertices of $\partial D_{0}$ distinct from $\tilde{p}$ belong to fundamental domains in $\mathcal{D}$ strictly closer to $D_{0}$ than $D_{1}$, which proves the lemma.
\end{proof}

Let $M= \mathrm{el}_{D_{0}}(\tilde{h}(D_{0}))$. Consider a little perturbation of the identity $\eta_{0}$ supported in the interior of one of the discs in $\mathcal{U}$ so that:
$$ \left\{
\begin{array}{l}
\mathrm{el}_{D_{0}}(\tilde{\eta}_{0} \circ \tilde{h}(D_{0})) \leq M \\
\eta_{0} \circ h(p) \notin \Pi(\partial D_{0})
\end{array}
\right.
.
$$
Notice that, if $\mathrm{el}_{D_{0}}(\tilde{\eta}_{0} \circ \tilde{h}(D_{0})) \leq M-1$, then the lemma is proved with $\eta=\eta_{0}$. Suppose now that, for an integer $j \in [0,2g-2]$, we have built a homeomorphism $\eta_{j}$ in $\mathrm{Homeo}_{0}(S)$ such that:
\begin{enumerate}
\item $\mathrm{Frag}_{\mathcal{U}}(\eta_{j}) \leq 4(j-1)+1$.
\item $\mathrm{el}_{D_{0}}(\tilde{\eta}_{j} \circ \tilde{h}(D_{0}))=M$.
\item The set $\tilde{\eta}_{j}(\tilde{h}(D_{0}))$ does not meet the faces of type $(i,M)$ for $0 \leq i <j$.
\item The point $\eta_{j} \circ h (p)$ does not belong to $\Pi(\partial D_{0})$.
\end{enumerate}
We will build a homeomorphism $\eta_{j+1}$ so that the set $\tilde{\eta}_{j+1} \circ \tilde{h}(D_{0})$ does not meet the faces of type $(j,M)$ either. This homeomorphism will be built by composing the homeomorphism $\eta_{j}$ with four homeomorphisms $f_{1}$, $f_{2}$, $f_{3}$ and $f_{4}$ each supported in the interior of one of the discs in $\mathcal{U}$. The homeomorphisms $f_{i}$ for $1 \leq i \leq 3$ will satisfy the following property $P$:
$$ \left\{ D \in \mathcal{D}, D \cap \tilde{f}_{i} \ldots \tilde{f}_{1} \circ \tilde{\eta}_{j} \circ \tilde{h}(D_{0}) \neq \emptyset \right\} = \left\{ D \in \mathcal{D}, D \cap \tilde{\eta}_{j} \circ \tilde{h}(D_{0}) \neq \emptyset \right\}.$$
If the image under $\tilde{\eta}_{j} \circ \tilde{h}$ of a vertex $\tilde{p}$ of the polygon $\partial D_{0}$ belongs to a face $D$ of type $(j,M)$, \emph{i.e.} the homeomorphism $\eta_{j} \circ h$ satisfies the hypothesis of the previous lemma, we denote by $\tilde{C}_{1}$ the connected component of $\tilde{\eta}_{j} \circ \tilde{h}(\partial D_{0}) \cap \mathring{D}$ which contains the point $\tilde{\eta}_{j} \circ \tilde{h}(\tilde{p})$. This is the unique connected component of $\tilde{\eta}_{j} \circ \tilde{h}(\partial D_{0})- \Pi^{-1}(\Pi(\partial D_{0}))$ which contains the image under the homeomorphism $\tilde{\eta}_{j} \circ \tilde{h}$ of a vertex of the polygon $\partial D_{0}$ which is contained in a face of type $(j,M)$, by the previous lemma. Notice that $\Pi(\tilde{C}_{1})$ is contained in the singular component of $\eta_{j} \circ h(\Pi(D_{0})) - \Pi(\partial D_{0})$.

\begin{figure}[ht]
\begin{center}
\includegraphics{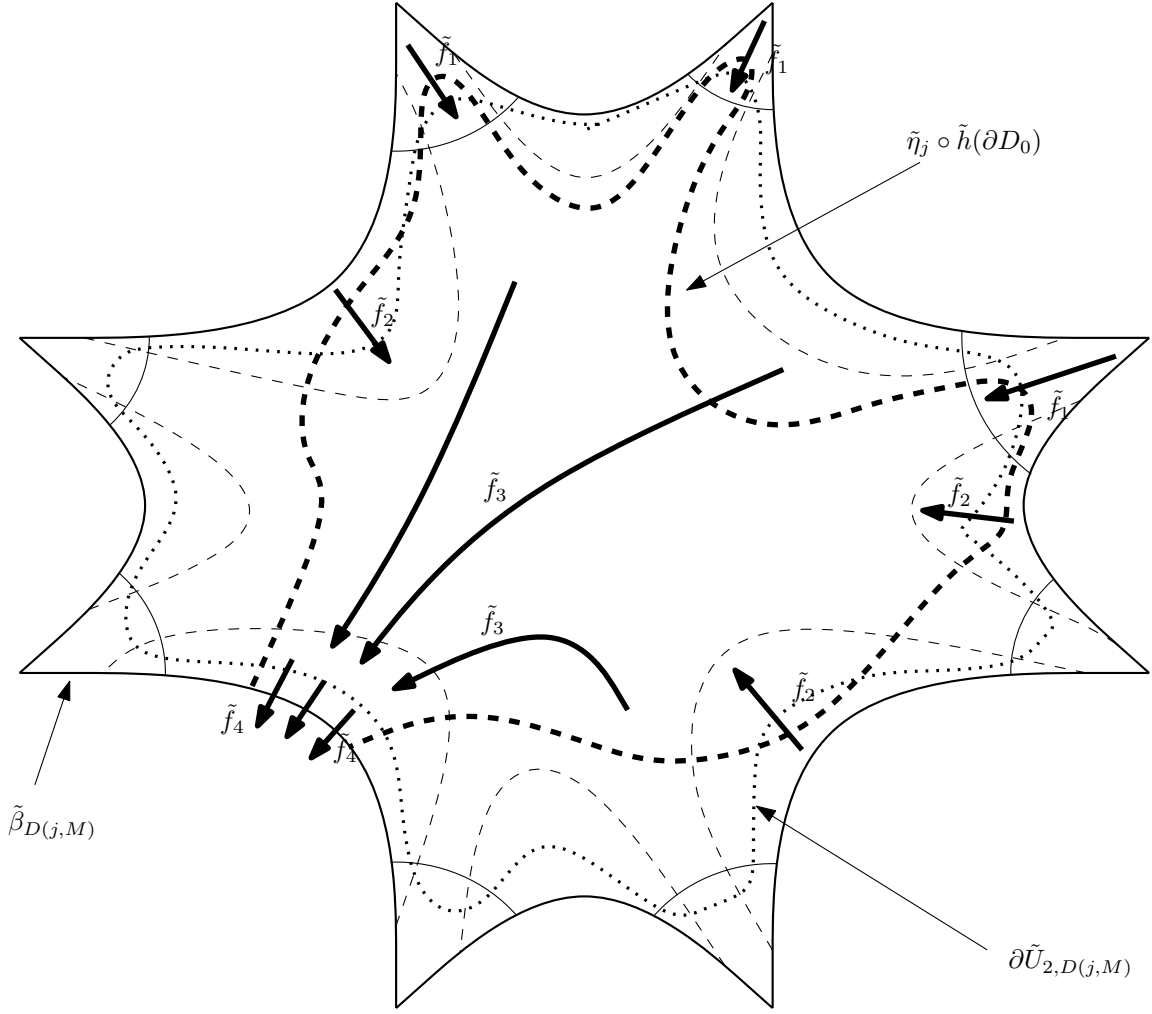}
\end{center}
\caption{Idea of the proof of Lemma \ref{nettoyage}: the face $D(j,M)$}
\end{figure}

Let $f_{1}$ be a homeomorphism supported in the interior of the disc $U_{0}$ with the following properties:
\begin{enumerate}
\item The homeomorphism $f_{1}$ globally preserves each edge in $A$.
\item For any connected component $C$ of $\mathring{U}_{0} \cap \eta_{j} \circ h(\Pi(\partial D_{0}))$ which does not contain the point $p$, we have
$$f_{1}(C) \subset \bigcup_{\alpha \in A} \mathring{V}_{\alpha} \cup \mathring{U}_{2}.$$
\item Case of the singular component: if the hypothesis of the previous lemma hold for the homeomorphism $\eta_{j} \circ h$, we require moreover that the image of $\Pi(\tilde{C}_{1})$ under $f_{1}$ is contained in the open set
$$\bigcup_{\alpha \in A} \mathring{V}_{\alpha} \cup \mathring{U}_{2}.$$
Notice that this condition is not implied by the other ones when $\Pi(\tilde{C}_{1})$ is contained in a connected component of $\mathring{U}_{0} \cap \eta_{j} \circ h(\Pi(\partial D_{0}))$ which contains the point $p$.
\end{enumerate}
Notice that, as the set $\tilde{C}_{1}$ is contained in a face of type $(j,M)$, the set $\overline{\Pi(\tilde{C}_{1})}$ does not contain the point $p$ (otherwise the closed set $\overline{\tilde{C}}_{1}$ would meet a face of type $(j-1,M)$, which is excluded by hypothesis on the homeomorphism $\eta_{j}$).
To build such a homeomorphism $f_{1}$, it suffices to take the time $1$ of the flow of a vector field for which the point $p$ is a repulsive fixed point, which is tangent to the edges of $A$ and is supported in the open disc $\mathring{U}_{0}$. As the homeomorphism $f_{1}$ globally preserves $\Pi^{-1}(\Pi(\partial D_{0}))$, it satisfies property $P$. Denote by $D(j,M)$ a face of type $(j,M)$. Recall that, by definition, if $j \geq 1$, all the faces adjacent to $D(j,M)$, except one, are of type $(j-1,M)$. Let $\tilde{\beta}_{D(j,M)}$ be the edge common to both the face $D(j,M)$ and the unique face adjacent to $D(j,M)$ which is at distance $d_{\mathcal{D}}(D(j,M),D_{0})-1$ from the fundamental domain $D_{0}$. Then, by hypothesis, any connected component of $\tilde{\eta}_{j} \circ \tilde{h} (\partial D_{0}) \cap D(j,M)$ has ends contained in the interior $\tilde{\beta}_{D(j,M)}- \partial \tilde{\beta}_{D(j,M)}$ of the edge $\tilde{\beta}_{D(j,M)}$. Let us denote by $\tilde{U}_{2, D(j,M)}$ the lift of the disc $U_{2}$ contained in the fundamental domain $D(j,M)$. Then the construction of the homeomorphism $f_{1}$ implies:
$$ \tilde{f}_{1} \circ \tilde{\eta}_{j} \circ \tilde{h}(\partial D_{0}) \cap D(j,M) \subset \mathring{\tilde{U}}_{2,D(j,M)} \cup \Pi^{-1}(\bigcup_{\alpha \in A} V_{\alpha}).$$

Let $f_{2}$ be a homeomorphism $\mathrm{Homeo}_{0}(S)$ which is supported in the union of the discs $V_{\alpha}$, where $\alpha$ varies over $A$, which satisfies the following properties:
\begin{enumerate}
\item The homeomorphism $f_{2}$ pointwise fixes all the edges in $A$.
\item For any edge $\alpha$ in $A$ and any connected component $C$ of $f_{1} \circ \eta_{j} \circ h(\Pi(\partial D_{0})) \cap V_{\alpha}$ which does not meet the edge $\alpha$, we have $f_{2}(C) \subset \mathring{U}_{2}$.
\item Case of the singular component: if the homeomorphism $\eta_{j} \circ h$ satisfies the hypothesis of the previous lemma, we require moreover that $f_{2} \circ f_{1}(\Pi(\tilde{C}_{1})) \subset \mathring{U}_{2}$.
\end{enumerate}
Let $\tilde{V}_{\tilde{\beta}_{D(j,M)}}$ be the lift of the disc $V_{\Pi(\tilde{\beta}_{D(j,M)})}$ which meets the edge $\tilde{\beta}_{D(j,M)}$. As the homeomorphism $\tilde{f}_{2}$ pointwise fixes $\Pi^{-1}(\Pi(\partial D_{0}))$, it satisfies property $P$. Moreover, by construction of the homeomorphism $f_{2}$, we have, for any face $D(j,M)$ of type $(j,M)$:
$$\tilde{f}_{2} \circ \tilde{f}_{1} \circ \tilde{\eta}_{j} \circ \tilde{h}(\partial D_{0}) \cap D(j,M) \subset \mathring{\tilde{V}}_{\tilde{\beta}_{D(j,M)}} \cup \mathring{\tilde{U}}_{2,D(j,M)}.$$
With the same method, we build a homeomorphism $f_{3}$ supported in the interior of $U_{2}$ such that, for any face $D(j,M)$ of type $(j,M)$, we have:
$$\tilde{f}_{3} \circ \tilde{f}_{2} \circ \tilde{f}_{1} \circ \tilde{\eta}_{j} \circ \tilde{h}(\partial D_{0}) \cap D(j,M) \subset \mathring{\tilde{V}}_{\tilde{\beta}_{D(j,M)}}.$$
As this homeomorphism pointwise fixes $\Pi^{-1}(\Pi(\partial D_{0}))$, it also satisfies property $P$.
Finally, let $f_{4}$ be a homeomorphism in $\mathrm{Homeo}_{0}(S)$ supported in the disjoint union of the open discs $\mathring{V}_{\alpha}$, where $\alpha$ varies over the set $A$, which satisfies the following properties for any edge $\alpha$ in $A$:
\begin{enumerate}
\item For any connected component $C$ of $f_{3} \circ f_{2} \circ f_{1} \circ \eta_{j} \circ h(\Pi(\partial D_{0})) \cap \mathring{V}_{\alpha}$ whose ends belong to the same connected component of $V_{\alpha}-\alpha$, we have $f_{4}(C) \cap \alpha = \emptyset$.
\item The homeomorphism $f_{4}$ pointwise fixes any other regular connected component of $f_{3} \circ f_{2} \circ f_{1} \circ \eta_{j} \circ h(\Pi(\partial D_{0})) \cap \mathring{V}_{\alpha}$.
\item Case of the singular component: if the homeomorphism $\eta_{j} \circ h$ satisfies the hypothesis of the previous lemma, if $\tilde{C}'_{1}$ denotes the connected component of $\tilde{f}_{3} \circ \tilde{f}_{2} \circ \tilde{f}_{1} \circ \tilde{\eta}_{j} \circ \tilde{h}(\partial D_{0}) \cap \Pi^{-1}( \cup_{\alpha} \mathring{V}_{\alpha})$ which contains the image under the homeomorphism $\tilde{f}_{3} \circ \tilde{f}_{2} \circ \tilde{f}_{1} \circ \tilde{\eta}_{j} \circ \tilde{h}$ of a vertex of the polygon $\partial D_{0}$ and which meets a face of type $(j,M)$, then:
$$ f_{4}(\Pi(\tilde{C}'_{1})) \cap \alpha = \emptyset.$$
\item In the case where the homeomorphism $\eta_{j} \circ h$ does not satisfy the hypothesis of the previous lemma, then the homeomorphism $f_{4}$ pointwise fixes the potential connected component of $f_{3} \circ f_{2} \circ f_{1} \circ \eta_{j} \circ h(\Pi(\partial D_{0})) \cap \mathring{V}_{\alpha}$ which is not homeomorphic to $\mathbb{R}$ and has ends in the two connected components of $V_{\alpha}-\alpha$.
\end{enumerate}
We now prove that the homeomorphism $\eta_{j+1}=f_{4} \circ f_{3} \circ f_{2} \circ f_{1} \circ \eta_{j}$ satisfies the required property, namely that $\mathrm{el}_{D_{0}}(\tilde{\eta}_{j+1} \circ \tilde{h}(D_{0})) \leq \mathrm{el}_{D_{0}}(\tilde{\eta}_{j} \circ \tilde{h}(D_{0}))$ and that the set $\tilde{\eta}_{j+1} \circ \tilde{h}(D_{0})$ does not meet the faces of type $(i,M)$ for $0 \leq i \leq j$. We will distinguish several pieces of the curve $\tilde{f}_{3} \circ \tilde{f}_{2} \circ \tilde{f}_{1} \circ \tilde{\eta}_{j} \circ \tilde{h}(\partial D_{0})$: the piece $\tilde{k}_{1}=\tilde{f}_{3} \circ \tilde{f}_{2} \circ \tilde{f}_{1} \circ \tilde{\eta}_{j} \circ \tilde{h}(\partial D_{0})-\Pi^{-1}(\cup_{\alpha} V_{\alpha})$ and the piece $\tilde{k}_{2}=\tilde{f}_{3} \circ \tilde{f}_{2} \circ \tilde{f}_{1} \circ \tilde{\eta}_{j} \circ \tilde{h}(\partial D_{0}) \cap \Pi^{-1}(\cup_{\alpha} V_{\alpha})$. In each of these cases, we prove that the image under $f_{4}$ of the chosen piece does not meet new faces (\emph{i.e.} which were not met by the curve $\tilde{f}_{3} \circ \tilde{f}_{2} \circ \tilde{f}_{1} \circ \tilde{\eta}_{j} \circ \tilde{h}(\partial D_{0})$)  and does not meet faces of type $(j,M)$.
\paragraph{First case} If $\tilde{C}$ is the closure of a connected component of $\tilde{k}_{1}$, then $f_{4}(\tilde{C})=\tilde{C}$ is contained in a face which belongs to the set:
$$\left\{ D \in \mathcal{D}, D \cap \tilde{f}_{3} \circ \tilde{f}_{2} \circ \tilde{f}_{1} \circ \tilde{\eta}_{j} \circ \tilde{h}(D_{0}) \neq \emptyset \right\} =\left\{ D \in \mathcal{D}, D \cap \tilde{\eta}_{j} \circ \tilde{h}(D_{0}) \neq \emptyset \right\}$$
and is not contained in a face of type $(j,M)$ because, for any face $D(j,M)$ of type $(j,M)$:
$$\tilde{f}_{3} \circ \tilde{f}_{2} \circ \tilde{f}_{1} \circ \tilde{\eta}_{j} \circ \tilde{h}(\partial D_{0}) \cap D(j,M) \subset \mathring{\tilde{V}}_{\tilde{\beta}_{D(j,M)}}.$$
\paragraph{Second case} If $\tilde{C}$ is a connected component of $\tilde{k}_{2}$ whose ends do not belong to the same connected component of $\Pi^{-1}(\cup_{\alpha} V_{\alpha}-\alpha)$ and, in the case where the homeomorphism $\eta_{j} \circ h$ satisfies the hypothesis of the previous lemma, which does not contain the image under the homeomorphism $\tilde{f}_{3} \circ \tilde{f}_{2} \circ \tilde{f}_{1} \circ \tilde{\eta}_{j} \circ \tilde{h}$ of a vertex of the polygon $\partial D_{0}$ then $\tilde{f}_{4}(\tilde{C})=\tilde{C}$ in the faces of the set:
$$\left\{ D \in \mathcal{D}, D \cap \tilde{\eta}_{j} \circ \tilde{h}(D_{0}) \neq \emptyset \right\}$$
and does not meet faces of type $(j,M)$. 
\paragraph{Third case} If $\tilde{C}$ is a connected component of $\tilde{k}_{2}$ whose ends all belong to the same connected component of $\Pi^{-1}(\cup_{\alpha} V_{\alpha}-\alpha)$, then the subset $\tilde{f}_{4}(\tilde{C})$ is contained in the interior of the fundamental domain in $\mathcal{D}$ which contains the ends of $\tilde{C}$ and which, therefore, is not a face of type $(j,M)$. Indeed, for any face $D(j,M)$ of type $(j,M)$:
$$\tilde{f}_{3} \circ \tilde{f}_{2} \circ \tilde{f}_{1} \circ \tilde{\eta}_{j} \circ \tilde{h}(\partial D_{0}) \cap D(j,M) \subset \mathring{\tilde{V}}_{\tilde{\beta}_{D(j,M)}}.$$
Moreover, such a face belongs to the set:
$$\left\{ D \in \mathcal{D}, D \cap \tilde{\eta}_{j} \circ \tilde{h}(D_{0}) \neq \emptyset \right\}.$$
\paragraph{Fourth case} Let us finally address the case where the homeomorphism $\eta_{j} \circ h$ satisfies the hypothesis of the previous lemma and where $\tilde{C}$ is a connected component of $\tilde{k}_{2}$ which contains the image under the homeomorphism $\tilde{f}_{3} \circ \tilde{f}_{2} \circ \tilde{f}_{1} \circ \tilde{\eta}_{j} \circ \tilde{h}$ of a vertex of the polygon $\partial D_{0}$. Let $\tilde{p}$ be the vertex of the polygon whose image under the homeomorphism $\tilde{f}_{3} \circ \tilde{f}_{2} \circ \tilde{f}_{1} \circ \tilde{\eta}_{j} \circ \tilde{h}$ belongs to a face $D_{1}$ of type $(j,M)$. By Lemmas \ref{sommets2} and \ref{sommets}, there exists a geodesic word of the form $\lambda_{1} \lambda_{2} \ldots \lambda_{2g} \gamma$, where the word $\lambda_{1} \lambda_{2} \ldots \lambda_{4g}$ belongs to $\Lambda$, which sends the face $D_{0}$ to the face $D_{1}$. Let us denote by $\gamma'$ the word $\gamma$ without the last letter. By construction of the homeomorphism $f_{4}$, by Lemma \ref{sommets}, the set $\tilde{f}_{4}(\tilde{C})$ is contained in the union of the following fundamental domains:
$$\begin{array}{l}
\lambda_{1} \ldots \lambda_{2g} \gamma'(D_{0}) \\
\lambda_{i+1} \ldots \lambda_{2g} \gamma(D_{0}) \mbox{ if } 1 \leq i \leq 2g \\
\lambda_{i+1} \ldots \lambda_{2g} \gamma'(D_{0}) \mbox{ if } 1 \leq i \leq 2g \\
\lambda_{4g-i}^{-1} \ldots \lambda_{2g}^{-1} \gamma(D_{0}) \mbox{ if } 1 \leq i \leq 2g \\
\lambda_{4g-i}^{-1} \ldots \lambda_{2g}^{-1} \gamma'(D_{0}) \mbox{ if } 1 \leq i \leq 2g.
\end{array}
$$
These fundamental domains are each at distance less than or equal to $M-j-1$ from $D_{0}$ and are not faces of type $(i,M)$ if $0 \leq i \leq j$.
Lemma \ref{nettoyage} is proved because, either $\mathrm{el}_{D_{0}}(\tilde{\eta}_{j+1} \circ \tilde{h}(D_{0})) < \mathrm{el}_{D_{0}}(\tilde{h}(D_{0}))$ and $\eta=\eta_{j+1}$ is appropriate, or one can continue the process until the other property is eventually satisfied.
\end{proof}

For a homeomorphism $h$ in $\mathrm{Homeo}_{0}(S)$, we denote by $\mathcal{F}_{h}$ the union of the set of exceptional faces which are maximal for the homeomorphism $h$ with the set of fundamental domains in $\mathcal{D}$ at distance less than or equal to $\mathrm{el}_{D_{0}}(\tilde{h}(D_{0}))-1$ and greater than or equal to $\mathrm{el}_{D_{0}}(\tilde{h}(D_{0}))-(2g-2)$ from $D_{0}$ and which share a vertex in common with an exceptional face which is maximal for $h$. By Lemma \ref{faceexc}, the faces $D$ which belong to this last kind satisfy the following property: if $\tilde{p}$ denotes the vertex of the boundary of $D$ which belongs to an exceptional maximal face, any face adjacent to $D$ which does not contain the point $\tilde{p}$ is a face of type $(i, \mathrm{el}_{D_{0}}(\tilde{h}(D_{0})))$, for an integer $i$ between $0$ and $2g-3$.

\begin{lemma} \label{nettoyage2}
Let $h$ be a homeomorphism in $\mathrm{Homeo}_{0}(S)$ with the following properties:
\begin{enumerate}
\item $h(p) \notin \Pi(\partial D_{0})$.
\item $\mathrm{el}_{D_{0}}(\tilde{h}(D_{0})) \geq 4g$.
\item The set $\tilde{h}(D_{0})$ does not meet the faces of type $(i,\mathrm{el}_{D_{0}}(\tilde{h}(D_{0})))$ for any index $0 \leq i \leq 2g-2$.
\end{enumerate}
Then, there exists a homeomorphism $\eta$ in $\mathrm{Homeo}_{0}(S)$ with the following properties:
\begin{enumerate} 
\item For any fundamental domain $D$ in $\mathcal{F}_{h}$, the connected components of $\tilde{\eta} \circ \tilde{h}(\partial D_{0}) \cap D$ are contained in $\Pi^{-1}(\mathring{U}_{0})$.
\item $\eta \circ h(p) \notin \Pi(\partial D_{0})$.
\item $\mathrm{el}_{D_{0}}(\tilde{\eta} \circ \tilde{h}(D_{0})) \leq \mathrm{el}_{D_{0}}(\tilde{h}(D_{0}))$.
\item $\mathrm{Frag}_{\mathcal{U}}(\eta) \leq 4$.
\item The set $\tilde{\eta} \circ \tilde{h}(D_{0})$ does not meet faces of type $(i,\mathrm{el}_{D_{0}}(\tilde{h}(D_{0})))$ for $0 \leq i \leq 2g-2$.
\end{enumerate}
\end{lemma}

\begin{proof}
During this proof, we need the following lemma which allows us to deal with the singular components:

\begin{lemma} \label{sommets3}
Let $h$ be a homeomorphism of $S$ which satisfies the hypothesis of Lemma \ref{nettoyage2}. Suppose that there exists a vertex $\tilde{p}$ of the polygon $\partial D_{0}$ such that the point $\tilde{h}(\tilde{p})$ belongs to a fundamental domain $D_{1}$ in $\mathcal{F}_{h}$ at distance $i$ from an exceptional face $D_{max}$ which is maximal for $h$, with $0 \leq i \leq 2g-2$. Then there exist two subwords $\lambda_{1}  \ldots \lambda_{2g}$ and $\lambda'_{1} \ldots \lambda'_{2g-1}$ of words $\lambda_{1} \ldots \lambda_{4g}$ and $\lambda'_{1} \ldots \lambda'_{4g}$ in $\Lambda$ and a geodesic word of the form $\lambda_{1} \ldots \lambda_{2g} \gamma \lambda'_{1} \ldots \lambda'_{2g-1}$ such that:
\begin{enumerate}
\item $\lambda_{1} \ldots \lambda_{2g} \gamma \lambda'_{1} \ldots \lambda'_{2g-1-i}(D_{0})=D_{1}$.
\item $\lambda_{1} \ldots \lambda_{2g} \gamma \lambda'_{1} \ldots \lambda'_{2g-1}(D_{0})=D_{max}$. 
\item The vertices of the polygon $\partial D_{0}$ are the points of the form $\lambda_{i}^{-1}  \ldots \lambda_{1}^{-1}(\tilde{p})$ or $\lambda_{4g-i+1} \ldots \lambda_{4g}(\tilde{p})$.
\end{enumerate}
\end{lemma}

\noindent \textbf{Remark} The lemma implies in particular that the point $\tilde{p}$ is the unique vertex of the polygon $\partial D_{0}$ whose image under $\tilde{h}$ belongs to a fundamental domain in $\mathcal{F}_{h}$.

\begin{proof}
Let us denote by $\tilde{p}'$ the vertex of the polygon $\partial D_{0}$ such that the point $\tilde{h}(\tilde{p}')$ belongs to a fundamental domain $D'_{1}$ in $\mathcal{D}$ at distance $l(h)$ from $D_{0}$. Then, by Lemma \ref{sommets}, $D'_{1}= \lambda_{1} \ldots \lambda_{2g} \gamma'(D_{0})$, where $\lambda_{1} \ldots \lambda_{2g}$ is a subword of length $2g$ of a word $\lambda_{1} \ldots \lambda_{4g}$ in $\Lambda$ and $\lambda_{1} \ldots \lambda_{2g} \gamma'$ is a geodesic word. Moreover, by the same lemma, after replacing $\lambda_{1} \ldots \lambda_{2g}$ with $\lambda_{4g}^{-1} \ldots \lambda_{2g+1}^{-1}$, we may suppose that $\tilde{p}=\lambda_{j}^{-1} \ldots \lambda_{1}^{-1}(\tilde{p}')$, where $0 \leq j \leq 2g$. Therefore, the face $D_{1}$ is a face of the form $D_{1}=\lambda_{j+1} \ldots \lambda_{2g} \gamma'(D_{0})$. As the face $D_{1}$ belongs to $\mathcal{F}_{h}$, by Lemma \ref{geodexc}, we have $\gamma'=\gamma \lambda'_{1} \ldots \lambda'_{2g-i-1}$, where $\lambda'_{1} \ldots \lambda'_{2g-1}$ is a subword of length $2g-1$ of a word in $\Lambda$ and:
$$D_{max}=\lambda_{j+1} \ldots \lambda_{2g} \gamma \lambda'_{1} \ldots \lambda'_{2g-1}(D_{0}).$$
The lemma will be proved if $j=0$. Suppose by contradiction that $j \geq 1$. As $d_{\mathcal{D}}(D'_{1},D_{0}) \leq d_{\mathcal{D}}(D_{max},D_{0})$, then $j \leq i$. Moreover, by Lemma \ref{geodexc}, the faces of the form $\lambda_{1} \ldots \lambda_{2g} \gamma' \lambda'_{1} \ldots \lambda'_{2g-i-1}a_{1} \ldots a_{k}(D_{0})$, where $0 \leq k \leq i-j$, the letters $a_{i}$ are elements of $\mathcal{G}$ and the word $\lambda_{1} \ldots \lambda_{2g} \gamma' \lambda'_{1}  \ldots \lambda'_{2g-i-1} a_{1} \ldots a_{k}$ is reduced, are not exceptional, so that the face $D'_{1}$ is a face of type $(i-j, \mathrm{el}_{D_{0}}(\tilde{h}(D_{0})))$. This contradicts the fact that the set $\tilde{h}(\partial D_{0})$ does not meet faces of this type.
\end{proof}

\begin{figure}[ht]
\begin{center}
\includegraphics{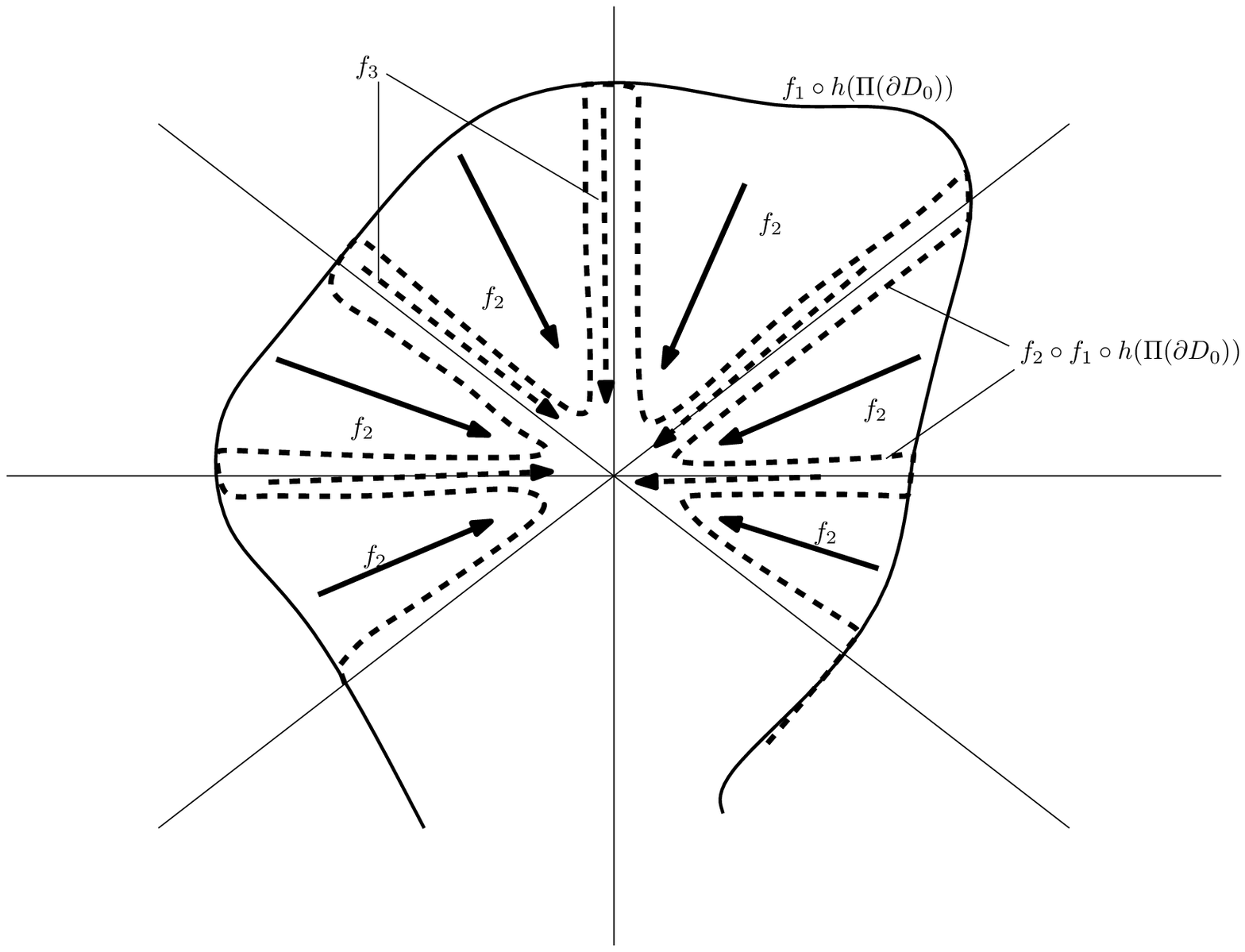}
\end{center}
\caption{Illustration of the proof of Lemma \ref{nettoyage2}}
\end{figure}

By methods similar to those used to prove Lemma \ref{nettoyage}, we build a homeomorphism $f_{1}$ which is the composition of a homeomorphism supported in $U_{0}$ with a homeomorphism supported in the union of the $V_{\alpha}$'s, which globally preserves $\Pi^{-1}(\Pi(\partial D_{0}))$ and has the following property. Let $D$ be a fundamental domain in $\mathcal{F}_{h}$. Then the face $D$ has exactly two adjacent faces which belong to $\mathcal{F}_{h}$ and we denote by $\tilde{\alpha}_{D}$ and $\tilde{\beta}_{D}$ the edges common to both the boundaries of one of these faces and of $D$.  We denote by $\tilde{U}_{2,D}$ the lift of the disc $U_{2}$ contained in $D$, $\tilde{V}_{\tilde{\alpha}_{D}}$ the lift of $V_{\Pi(\tilde{\alpha}_{D})}$ which meets $\tilde{\alpha}_{D}$, $\tilde{V}_{\tilde{\beta}_{D}}$ the lift of $V_{\Pi(\tilde{\beta}_{D})}$ which meets $\tilde{\beta}_{D}$ and $\tilde{U}_{0,D}$ the lift of $U_{0}$ which contains the point $\tilde{\alpha}_{D} \cap \tilde{\beta}_{D}$. Then, for any connected component $\tilde{C}$ of $\tilde{h}(\partial D_{0}) \cap D$, we have:
$$ \tilde{f}_{1}(\tilde{C}) \subset \tilde{U}_{0,D} \cup \tilde{V}_{\tilde{\alpha}_{D}} \cup \tilde{V}_{\tilde{\beta}_{D}} \cup \tilde{U}_{2,0}.$$
Moreover, if no end of $\tilde{C}$ meets $E$, where $E$ is one of the sets $\tilde{U}_{0,D}$, $\tilde{V}_{\tilde{\alpha}_{D}}$ or $\tilde{V}_{\tilde{\beta}_{D}}$ and if $\tilde{C}$ does not have one end in $\tilde{V}_{\tilde{\alpha}_{D}}$ and the other in $\tilde{V}_{\tilde{\beta}_{D}}$ then $\tilde{f}_{1}(\tilde{C})$ does not meet $E$.

If the homeomorphism $h$ does not satisfy the hypothesis of the previous lemma, we denote by $\mathcal{C}$ the set of connected components of $f_{1} \circ h(\Pi(\partial D_{0}))-\Pi(\partial D_{0})$ whose ends belong all either to the same edge in $A$, or to two consecutive edges in $A$ (\emph{i.e.} edges which admit lifts which share a point in common and are contained in a same face in $\mathcal{D}$). If the homeomorphism $h$ satisfies the hypothesis of the previous lemma, we define the set $\mathcal{C}$ as the union of the above set with the set $\left\{ \Pi(\tilde{C}_{1}) \right\}$, where $\tilde{C}_{1}$ is the unique connected component of $\tilde{f}_{1} \circ \tilde{h}(\partial D_{0})- \Pi^{-1}(\Pi(\partial D_{0}))$ which contains the image under the homeomorphism $\tilde{f}_{1} \circ \tilde{h}$ of a vertex of $\partial D_{0}$ and which is contained in a face in $\mathcal{F}_{h}$.

We build a homeomorphism $f_{2}$ which is supported in $U_{2}$ with the following property: given two consecutive edges $\alpha$ and $\beta$, for any element $C$ in $\mathcal{C}$ whose ends belong to $\alpha \cup \beta$: $f_{2}(C) \subset V_{\alpha} \cup V_{\beta} \cup U_{0}$. Moreover, if the ends of $C$ do not meet a set $E$ among $V_{\alpha}$, $V_{\beta}$ or $U_{0}$, then $f_{2}(C)$ is disjoint from $E$. The construction implies that, for any fundamental domain $D$ in $\mathcal{F}_{h}$ and any connected component $\tilde{C}$ of $\tilde{f}_{1} \circ \tilde{h}(\partial D_{0}) \cap D$:
$$ \tilde{f}_{2}(\tilde{C}) \subset \tilde{U}_{0,D} \cup \tilde{V}_{\tilde{\alpha}_{D}} \cup \tilde{V}_{\tilde{\beta}_{D}}.$$
Moreover, if the set $\tilde{C}$ does not meet a disc $E$ among $\tilde{U}_{0,D}$, $\tilde{V}_{\tilde{\alpha}_{D}}$ or $\tilde{V}_{\tilde{\beta}_{D}}$, then $\tilde{f}_{2}(\tilde{C})$ does not meet this disc either. As the homeomorphism $f_{2}$ is supported in $U_{2}$:
$$ \left\{ D \in \mathcal{D}, \tilde{f}_{2} \circ \tilde{f}_{1} \circ \tilde{h}(D_{0}) \cap D \neq \emptyset \right\} = \left\{ D \in \mathcal{D}, \tilde{h}(D_{0}) \cap D \neq \emptyset \right\}.$$
Let $f_{3}$ be a homeomorphism supported in the union of the $V_{\alpha}$'s with the following properties:
\begin{enumerate}
\item For any edge $\alpha$ in $A$ and any connected component $C$ of $f_{2} \circ f_{1} \circ h(\Pi(\partial D_{0}))\cap V_{\alpha}$ whose ends belong to a same connected component of $U_{0} \cap V_{\alpha}$, then $f_{3}(C) \subset U_{0}$.
\item For any connected component $C$ of $f_{2} \circ f_{1} \circ h(\Pi(\partial D_{0}))\cap V_{\alpha}$ which does not meet the edge $\alpha$: $f_{3}(C) \cap \alpha = \emptyset$.
\item If $\tilde{C}_{1}$ is a connected component of $\tilde{f}_{2} \circ \tilde{f}_{1} \circ \tilde{h}(\partial D_{0})- \Pi^{-1}(\Pi(\partial D_{0}))$ which contains the image under the homeomorphism $\tilde{f}_{2} \circ \tilde{f}_{1} \circ \tilde{h}$ of a vertex of the polygon $\partial D_{0}$ and which is contained in a face in $\mathcal{F}_{h}$, then $f_{3}( \Pi(\tilde{C}_{1})) \subset U_{0}$.
\end{enumerate}
Let $D$ be a face in $\mathcal{F}_{h}$ at distance $i < 2g-2$ from an exceptional face which is maximal for $h$. We prove that, for any fundamental domain $D'$ in $\mathcal{D}$ and any connected component $\tilde{C}$ of $D' \cap \tilde{f}_{2} \circ \tilde{f}_{1} \circ \tilde{h}(\partial D_{0})$:
$$ \tilde{f}_{3}(\tilde{C}) \cap D \subset \tilde{U}_{0,D}.$$
If the face $D'$ is not adjacent to $D$, as $\tilde{f}_{3}(\tilde{C})$ is contained in the set of faces adjacent to $D'$, then: $\tilde{f}_{3}(\tilde{C}) \cap D = \emptyset$. By Lemma \ref{faceexc}, a face adjacent to $D$ is:
\begin{enumerate}
\item Either a face of type $(i-1, \mathrm{el}_{D_{0}}(\tilde{h}(D_{0})))$.
\item Or at distance $\mathrm{el}_{D_{0}}(\tilde{h}(D_{0}))+1$ from the face $D_{0}$.
\item Or in $\mathcal{F}_{h}$.
\end{enumerate}
In the first two cases, the faces do not meet $\tilde{f}_{2} \circ \tilde{f}_{1} \circ \tilde{h}(\partial D_{0})$. Therefore, it suffices to address the two following cases:
\begin{enumerate}
\item The face $D'$ belongs to $\mathcal{F}_{h}$ and is adjacent to $D$.
\item $D'=D$.
\end{enumerate}
In the first case, let $\tilde{\alpha}=D \cap D'$ and $\tilde{V}_{\tilde{\alpha}}$ be the lift of $V_{\Pi(\tilde{\alpha})}$ which meets $\tilde{\alpha}$. Notice that any point of $\tilde{C}$ which does not meet $\tilde{V}_{\tilde{\alpha}}$ has an image disjoint from $D$. Moreover, by construction of $f_{3}$, any connected component of $\tilde{C} \cap \tilde{V}_{\tilde{\alpha}}$ which does not meet $\tilde{\alpha}$ has an image under $\tilde{f}_{3}$ which does not meet the fundamental domain $D$. Let us denote by $\tilde{C}_{1}$ a connected component of $\tilde{C} \cap \tilde{V}_{\tilde{\alpha}}$ which meets $\tilde{\alpha}$ and denote by $\tilde{C}'_{1}$ the connected component of $\tilde{V}_{\tilde{\alpha}}$ which contains $\tilde{C}_{1}$. The connected component $\tilde{C}'_{1}$ has necessarily its both ends contained in $\tilde{U}_{0,D}$ by the properties satisfied by $\tilde{f}_{2} \circ \tilde{f}_{1} \circ \tilde{h}$. Therefore, the set $\tilde{f}_{3}(\tilde{C}_{1})$ is contained in the set $\tilde{f}_{3}(\tilde{C}'_{1})$ which is itself contained in $\tilde{U}_{0,D}$. This proves the above result in the first case.
In the second case, the same kind of arguments implies that $\tilde{f}_{3}(\tilde{C})\cap D \subset \tilde{U}_{0,D}$.

Finally, let $f_{4}$ be a homeomorphism in $\mathrm{Homeo}_{0}(S)$ supported in the union of the $V_{\alpha}$'s with the following properties:
\begin{enumerate}
\item The homeomorphism $f_{4}$ globally preserves $\Pi(\partial D_{0})$.
\item For any connected component $\tilde{C}$ of $\tilde{f}_{3} \circ \tilde{f}_{2} \circ \tilde{f}_{1} \circ \tilde{h}(\partial D_{0})-\Pi^{-1}(\Pi(\partial D_{0}))$ contained in a face in $\mathcal{F}_{h}$ at distance $2g-2$ from an exceptional maximal face: $\tilde{f}_{4}(\tilde{C}) \subset \Pi^{-1}(U_{0})$.
\item $f_{4}(U_{0}) \subset U_{0}$.
\end{enumerate}
Then the homeomorphism $\eta=f_{4} \circ f_{3} \circ f_{2} \circ f_{1}$ satisfies the following property, for any face $D$ in $\mathcal{F}_{h}$:
$$ \tilde{f}_{4} \circ \tilde{f}_{3} \circ \tilde{f}_{2} \circ \tilde{f}_{1}(\partial D_{0})\cap D \subset \Pi^{-1}(U_{0}).$$
Moreover:
$$ \left\{ D \in \mathcal{D}, D \cap \tilde{f}_{4} \circ \tilde{f}_{3} \circ \tilde{f}_{2} \circ \tilde{f}_{1} \circ \tilde{h}(\partial D_{0}) \neq \emptyset \right\} = \left\{ D \in \mathcal{D}, D \cap \tilde{f}_{3} \circ \tilde{f}_{2} \circ \tilde{f}_{1} \circ \tilde{h}(\partial D_{0}) \neq \emptyset \right\}.$$
Therefore, in order to prove the lemma, it suffices to prove that any fundamental domain in $\mathcal{D}$ met by $\tilde{f}_{3} \circ \tilde{f}_{2} \circ \tilde{f}_{1} \circ \tilde{h}(\partial D_{0})$ is at distance at most $\mathrm{el}_{D_{0}}(\tilde{h}(D_{0}))$ from $D_{0}$ and is not a face of type $(i,\mathrm{el}_{D_{0}}(\tilde{h}(D_{0})))$ for $0 \leq i \leq 2g-2$. Let $D$ be a fundamental domain in $\mathcal{D}$. If $\tilde{C}$ is a connected component of $\tilde{f}_{2} \circ \tilde{f}_{1} \circ \tilde{h}(\partial D_{0}) \cap D$ which does not contain the image under the homeomorphism $\tilde{f}_{2} \circ \tilde{f}_{1} \circ \tilde{h}$ of a vertex of the polygon $\partial D_{0}$, then the set $\tilde{f}_{3}(\tilde{C})$ meets only fundamental domains in $\mathcal{D}$ intersected by $\tilde{f}_{2} \circ \tilde{f}_{1} \circ \tilde{h}(\partial D_{0})$. If $\tilde{C}$ is a connected component of $\tilde{f}_{2} \circ \tilde{f}_{1} \circ \tilde{h}(\partial D_{0}) \cap D$ which contains the image under the homeomorphism $\tilde{f}_{2} \circ \tilde{f}_{1} \circ \tilde{h}$ of a vertex of the polygon $\partial D_{0}$, then either the homeomorphism $\tilde{h}$ does not satisfy the hypothesis of Lemma \ref{sommets3} and the last claim remains true, or it satisfies the hypothesis of this lemma and it suffices to apply it to complete the proof.
\end{proof}

We now complete the proof of Lemma \ref{herhyp}. Let $M=\mathrm{el}_{D_{0}}(\tilde{f}(D_{0}))$. By Lemmas \ref{nettoyage} and \ref{nettoyage2}, after possibly composing the homeomorphism $f$ with $8g-3$ homeomorphisms which are each supported in the interior of one of the discs of $\mathcal{U}$, we may suppose that the homeomorphism $f$ satisfies the following properties:
\begin{enumerate}
\item $f(p) \notin \Pi(\partial D_{0})$;
\item The set $\tilde{f}(D_{0})$ does not meet faces of type $(i,M)$, for any index $i \in [0,2g-2]$.
\item For any fundamental domain $D$ in $\mathcal{F}_{f}$ (defined just before Lemma \ref{nettoyage2}), the set $\tilde{f}(\partial D_{0}) \cap D$ is contained in $\tilde{U}_{0,D}$, where $\tilde{U}_{0,D}$ is the lift of $U_{0}$ which meets $D$, meets an exceptional maximal face and meets only fundamental domains in $\mathcal{D}$ at distance less than $M$ from $D_{0}$.
\end{enumerate}
Two distinct connected components $\xi_{1}$ and $\xi_{2}$ of $U_{0}-\Pi(\partial D_{0})$ are said to be \emph{adjacent} if $\overline{\xi}_{1} \cap \overline{\xi}_{2}$ is an interval which is not reduced to a point. Two connected components $\xi_{1}$ and $\xi_{2}$ of $U_{0}-\Pi(\partial D_{0})$ are said to be \emph{almost adjacent} if there exists a connected component $\xi$ of $U_{0}-\Pi(\partial D_{0})$ different from $\xi_{1}$ and from $\xi_{2}$ which is adjacent to $\xi_{1}$ and to $\xi_{2}$. Then such a connected component $\xi$ is unique: we call it the \emph{adjacency face} of $\xi_{1}$ and $\xi_{2}$.

In the case where any connected component of $\tilde{f}(\partial D_{0}) \cap \Pi^{-1}(U_{0})$ which contains the image under the homeomorphism $\tilde{f}$ of a vertex of the polygon $\partial D_{0}$ avoids the exceptional faces which are maximal for $f$, we denote by $\mathcal{C}$ the set of connected components of $f(\Pi(\partial D_{0})) \cap \mathring{U}_{0}$ whose ends all belong:
\begin{enumerate}
\item Either to the same connected component of $U_{0}-\Pi(\partial D_{0})$.
\item Or to the interior of an interval of the form $\partial U_{0} \cap \overline{\xi_{1} \cup \xi_{2}}$, where $\xi_{1}$ and $\xi_{2}$ are adjacent connected components of $U_{0}-\Pi(\partial D_{0})$.
\item Or to the interior of an interval of the form $\partial U_{0} \cap \overline{\xi_{1} \cup \xi \cup \xi_{2}}$, where $\xi_{1}$ and $\xi_{2}$ are connected components of $U_{0}-\Pi(\partial D_{0})$ which are almost adjacent with adjacency face $\xi$. 
\end{enumerate}
In the case where there exists a connected component $\tilde{C}_{1}$ of $\tilde{f}(\partial D_{0}) \cap \Pi^{-1}(U_{0})$ which contains the image under the homeomorphism $\tilde{f}$ of a vertex $\tilde{p}$ of the polygon $\partial D_{0}$ which meets an exceptional face which is maximal for $f$ (such a connected component is unique by Lemma \ref{sommets3}), the set $\mathcal{C}$ is the union of the previous set with the point $\Pi(\tilde{C}_{1})$.

\begin{figure}[ht]
\begin{center}
\includegraphics{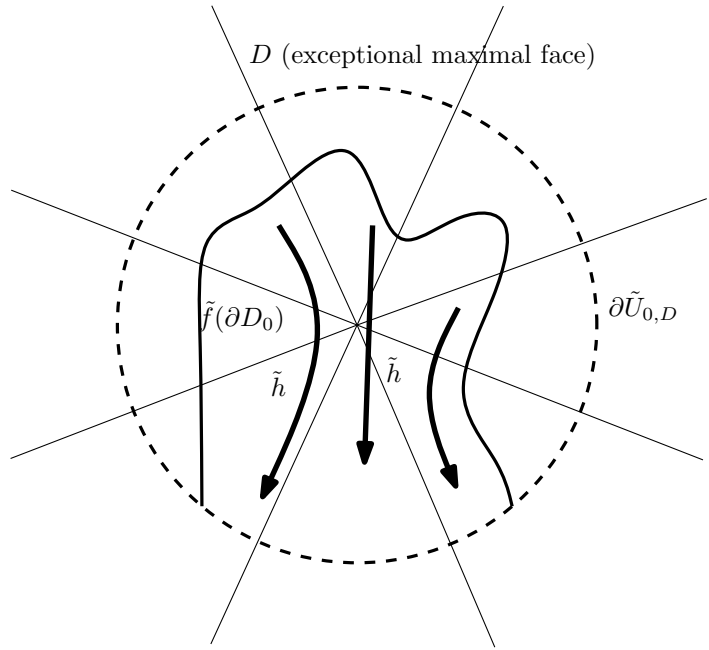}
\end{center}
\caption{End of the proof of Lemma \ref{herhyp}}
\end{figure}

Let $h$ be a homeomorphism supported in $\mathring{U}_{0}$ with the following properties:
\begin{enumerate}
\item For any connected component $C$ in $\mathcal{C}$ whose ends belong to a same face or to two adjacent faces, $h(C)$ is contained in the interior of the union of the closures of connected components of $U_{0}-\Pi(\partial D_{0})$ that meet the ends of $C$.
\item For any connected component $C$ in $\mathcal{C}$ whose ends belong to two almost adjacent connected components $\xi_{1}$ and $\xi_{2}$ of $U_{0}-\Pi(\partial D_{0})$ and to their adjacency face $\xi$, then $h(C) \subset \mathring{k}$, with $k=\overline{\xi_{1}} \cup \overline{\xi_{2}} \cup \overline{\xi}$.
\item The homeomorphism $h$ pointwise fixes any connected component of $f(\Pi(\partial D_{0})) \cap U_{0}$ which does not contain an element of $\mathcal{C}$.
\end{enumerate}
We claim that $\mathrm{el}_{D_{0}}(\tilde{h} \circ \tilde{f}(D_{0})) \leq \mathrm{el}_{D_{0}}(\tilde{f}(D_{0}))-1=M-1$,which completes the proof of Lemma \ref{herhyp}.

The faces at distance $M$ from $D_{0}$ can be split into two types: the exceptional maximal ones, and those of type $(0,M)$. We prove that the set $\tilde{h} \circ \tilde{f}(D_{0})$ meets neither the first ones nor the second ones.

First, for a point $\tilde{y}$ in $\tilde{f}(\partial D_{0})$ which does not belong to $\Pi^{-1}(\mathring{U}_{0})$, we have $\tilde{h}(\tilde{y})=\tilde{y}$ and the point $\tilde{y}$ belongs neither to an exceptional maximal face nor to a face of type $(0,M)$ by the properties satisfied by $f$. Thus, the point $\tilde{h}(\tilde{y})$ does not meet a fundamental domain in $\mathcal{D}$ at distance $M$ from $D_{0}$.

Let $\tilde{C}$ be a connected component of $\tilde{f}(\partial D_{0})\cap \Pi^{-1}(U_{0})$ which does not contain the image under $\tilde{f}$ of a vertex of $\partial D_{0}$. 

Let $D$ be an exceptional maximal face for $f$. Let us prove that $D \cap \tilde{h}(\tilde{C})= \emptyset$. If the lift $\tilde{U}_{0}$ of the disc $U_{0}$ which contains $\tilde{C}$ does not meet $D$, then this property holds. Suppose now that the lift of the disc $U_{0}$ which contains $\tilde{C}$ meets $D$. We now use notation from Lemma \ref{faceexc}. By this lemma, the faces $D_{i}^{j}$, for $ 1 \leq i \leq 2g-2$ and $j \in \left\{1,2\right\}$, belong to $\mathcal{F}_{f}$. By the properties satisfied by the homeomorphism $f$, the connected component $\tilde{C}$ has necessarily its ends contained in $D^{1}_{2g-1}$, $D^{2}_{2g-1}$ or $D^{1}_{2g}=D^{2}_{2g}$. But the connected components $\Pi(\mathring{D}^{1}_{2g-1})$ and $\Pi(\mathring{D}^{2}_{2g-1})$ of $U_{0}-\Pi(\partial D_{0})$ are almost adjacent with adjacency face $\Pi(\mathring{D}^{1}_{2g})$. This implies the following inclusion: $\tilde{h}(\tilde{C}) \subset D^{1}_{2g-1} \cup D^{2}_{2g-1} \cup D^{1}_{2g}$. In particular: $\tilde{h}(\tilde{C}) \cap D= \emptyset$.

Let $D$ be a fundamental domain in $\mathcal{D}$ of type $(0,M)$. Let us prove that $\tilde{h}(\tilde{C}) \cap D= \emptyset$. By the properties satisfied by $\tilde{f}$, the set $\tilde{C}$ does not meet $D$. The set $\tilde{h}(\tilde{C})$ meets the face $D$ only in the following case: the two ends of $\tilde{C}$ belong to two distincts fundamental domains which are adjacent to $D$. However, these two fundamental domains would be at distance $M-1$ from $D_{0}$ (they cannot be at distance $M+1$ from $D_{0}$ by definition of $M$), which would contradict the fact that a fundamental domain $D$ is a face of type $(0,M)$.

It remains to deal with the case of a connected component $\tilde{C}$ of $\tilde{f}(\partial D_{0}) \cap \Pi^{-1}(U_{0})$ which contains the image under $\tilde{f}$ of a vertex of the polygon $\partial D_{0}$. The case where no connected component of this kind meets an exceptional maximal face is easy. Otherwise, we have to use Lemma \ref{sommets3} to obtain an explicit expression of the fundamental domains met by the image under $\tilde{h}$ of such connected components. We notice that those faces are not maximal for $\tilde{f}$.

This completes the proof of Lemma \ref{herhyp}.
\end{proof}

\subsection{Proof of Lemma \ref{inhyp}}

\begin{proof}[Proof of Lemma \ref{inhyp}]
The proof of this lemma is analogous to the proof of Lemma \ref{intore}. Let $\beta$ and $\gamma$ be simple closed curves of $S$ which are homotopic and which are not homotopic to a point. We denote by $l(\gamma, \beta)$ the number of connected components of $\Pi^{-1}(\beta)$ that a connected component of $\Pi^{-1}(\gamma)$ meets. Let us denote by $\alpha$ an edge in $A$ and by $\alpha'$ a simple closed curve isotopic to $\alpha$ and disjoint from $\alpha$. Let $S_{\alpha'}$ be the complement of an open tubular neighbourhood of $\alpha'$ and let $S_{\alpha}$ be the complement of an open tubular neighbourhood of $\alpha$ so that $\mathring{S}_{\alpha'} \cup \mathring{S}_{\alpha}=S$. Let $f$ be a homeomorphism in $\mathrm{Homeo}_{0}(S)$ with $\mathrm{el}_{D_{0}}(\tilde{f}(D_{0})) \leq 4g$. Throughout the proof, $\eta$ denotes a positive constant which will be fixed later. We will use the following result, which is a consequence of Lemma \ref{diamfrag2} applied to neighbourhoods of $S_{\alpha}$ and of $S_{\alpha'}$: there exists $\lambda_{\eta}>0$ such that, for any homeomorphism $h$ in $\mathrm{Homeo}_{0}(S_{\alpha})$ or in $\mathrm{Homeo}_{0}(S_{\alpha'})$ with $\mathrm{el}_{D_{0}}(\tilde{h}(D_{0})) \leq \eta$, we have $\mathrm{Frag}_{\mathcal{U}}(h) \leq \lambda_{\eta}$.

Let us give the idea of the proof. By composing the homeomorphism $f$ with at most $16g$ homeomorphisms with fragmentation length (with respect to $\mathcal{U}$) less than or equal to $\lambda_{\eta}$, we obtain a homeomorphism $f_{1}$ which sends the curve $\alpha$ to a curve disjoint from $\alpha$ and contained in $\mathring{S}_{\alpha'}$. Then, after composing $f_{1}$ with a homeomorphism supported in $S_{\alpha'}$ which is equal to $f_{1}^{-1}$ on a neighbourhood of $f_{1}(\alpha)$ and with fragmentation length bounded by $\lambda_{\eta}$, we obtain a homeomorphism $f_{2}$ which is equal to the identity on a neighbourhood of $\alpha$ and isotopic to the identity relative to $\alpha$. By composing $f_{2}$ with at most three homeomorphisms supported in $S_{\alpha}$ or in $S_{\alpha'}$ and and with fragmentation length bounded by $\lambda_{\eta}$, we obtain a homeomorphism $f_{3}$ which pointwise fixes a neighbourhood of the boundary of $S_{\alpha}$ and isotopic to the identity relative to this boundary. Then the homeomorphism $f_{3}$ can be written as a product of a homeomorphism in $\mathrm{Homeo}_{0}(S_{\alpha})$ and of a homeomorphism in $\mathrm{Homeo}_{0}(S_{\alpha'})$ with disjoint supports. The previous statement applied to these two homeomorphisms implies that the fragmentation length of $f_{3}$ is less than or equal to $2\lambda_{\eta}$. Of course, the constant $\eta$ will have to be large enough so that this proof works.

Let us give now some details. Let $\alpha_{1}$ and $\alpha_{2}$  (respectively $\alpha'_{1}$ and $\alpha'_{2}$) be the two connected components of the boundary of $S_{\alpha}$ (respectively of $S_{\alpha'}$). For any two disjoint subsets $A$ and $B$ of $\tilde{S}$, we denote by $\delta(A,B)$ the number of connected components of $\Pi^{-1}(\alpha_{1} \cup \alpha_{2} \cup \alpha'_{1} \cup \alpha'_{2})$ disjoint from $A$ and from $B$ which separate $A$ and $B$. Let $M(f)$ be the maximum of $\delta(\tilde{S}',\tilde{\alpha})$, where $\tilde{S}'$ varies over all connected components of $\Pi^{-1}(S_{\alpha})$ or of $\Pi^{-1}(S_{\alpha'})$ which meet $\tilde{f}(\tilde{\alpha})$. As, by hypothesis, we have $\mathrm{el}_{D_{0}}(\tilde{f}(D_{0}))\leq 4g$, then $M(f) \leq 16g$. Notice that, if $\tilde{S}'$ is a connected component of $\Pi^{-1}(S_{\alpha})$ or of $\Pi^{-1}(S_{\alpha'})$ such that $\delta(\tilde{S}',\tilde{\alpha})=M(f)$, then any connected component of $\tilde{f}(\tilde{\alpha})\cap \tilde{S}'$ has its ends in the same connected component of $\partial \tilde{S}'$. Let $S'=\Pi(\tilde{S}')$ and $S"$ be the surface $S_{\alpha}$ if $S'=S_{\alpha'}$, or the surface $S_{\alpha'}$ if $S'=S_{\alpha}$. Denote by $h_{1}$ a homeomorphism supported in $S'$ with the following properties:
\begin{enumerate}
\item $\mathrm{el}_{D_{0}}(\tilde{h}_{1}(D_{0})) \leq 4g$.
\item For any connected component $C$ of $f(\alpha) \cap S'$ whose ends are in the same connected component of $\partial S'$ and homotopic to a path on the boundary of $S'$: $h_{1}(C) \subset S"$.
\end{enumerate}
These two properties are compatible because $\mathrm{el}_{D_{0}}(\tilde{f}(D_{0})) \leq 4g$. Notice that we have $\mathrm{el}_{D_{0}}(\tilde{h}_{1} \circ \tilde{f}(D_{0})) \leq 8g$ and $ \mathrm{Frag}_{\mathcal{U}}(h_{1}) \leq \lambda_{\eta}$ if $\eta \geq 4g$. Moreover, for any connected component $\tilde{S}'$ of $\Pi^{-1}(S')$ with $d(\tilde{\alpha}, \tilde{S}')=M(f)$ and for any connected component $\tilde{C}$ of $\tilde{f}(\tilde{\alpha}) \cap \tilde{S}'$: $\tilde{h}_{1}(\tilde{C}) \subset \Pi^{-1}(S")$. Now, let $h_{2}$ be a homeomorphism supported in $S"$ with the following properties:
\begin{enumerate}
\item $\mathrm{el}_{D_{0}}(\tilde{h}_{2}(D_{0})) \leq 8g$.
\item For any connected component $C$ of $h_{1} \circ f(\alpha) \cap S"$ whose ends are in the same connected component of $\partial S"$ and homotopic to a path on the boundary of $S"$: $h_{2}(C) \subset S'$.
\end{enumerate}
These two properties are compatible because $\mathrm{el}_{D_{0}}(\tilde{h}_{1} \circ \tilde{f}(\partial D_{0})) \leq 8g$. Notice that we have $\mathrm{el}_{D_{0}}(\tilde{h}_{2} \circ \tilde{h}_{1} \circ \tilde{f}(\partial D_{0})) \leq 16g$ and $\mathrm{Frag}_{\mathcal{U}}(h_{2}) \leq \lambda_{\eta}$ if $\eta \geq 16g$. Moreover, we have $M(h_{2} \circ h_{1} \circ f) \leq M(f)-2$. We repeat this process at most $8g$ times so that, after composing the homeomorphism $f$ with at most $16g$ homeomorphisms with fragmentation length less than or equal to $\lambda_{\eta}$ (by taking $\eta \geq 2^{8g}.4g$), we obtain a homeomorphism $f_{1}$ which sends the curve $\alpha$ to a curve disjoint from $\alpha$ and which satisfies the following inequality:
$$ \mathrm{el}_{D_{0}}(\tilde{f}_{1}(D_{0})) \leq 2^{8g+1}.4g.$$
After composing the homeomorphism $f_{1}$ with four homeomorphisms with fragmentation length less than or equal to $\lambda_{\eta}$ (if we take $\eta \geq 2^{8g+4}.4g$), we obtain a homeomorphism $f_{3}$ which pointwise fixes a neighbourhood of $\partial S_{\alpha}$ and which is isotopic to the identity relative to this neighbourhood with:
$$ \mathrm{el}_{D_{0}}(\tilde{f}_{3}(D_{0})) \leq 2^{8g+5}.4g.$$
As written at the beginning of this proof, it suffices to take $\eta \geq 2^{8g+5}.4g$ to complete the proof of Lemma \ref{inhyp}.
\end{proof}

\subsection{Proof of the combinatorial lemmas}

\begin{proof}[Proof of Lemma \ref{geodexc}]
Let us describe the Dehn algorithm that we will use. Let $m$ be a reduced word on elements of $\mathcal{G}$. At each step of the algorithm, we look for a subword $f$ of $m$ with length greater than $2g$ which is contained in a word $f.\lambda'$ of $\Lambda$ (such a word $f$ will be said to be \emph{simplifiable}) and whose length is maximal among such words (it is said to be \emph{maximal} in $m$). The word $\lambda'$ will be called the \emph{complementary word} of $f$. Then we replace in $m$ the subword $f$ by the word $\lambda'^{-1}$ whose length is strictly smaller (the words in $\Lambda$ have length $4g$) and we make if necessary the free group reductions to obtain a new reduced word. By a theorem by Dehn (see \cite{LS}), a reduced word represents the trivial element in $\Pi_{1}(S)$ if and only if, after implementing a finite number of steps of this algorithm, we obtain the empty word.

Let us give some general facts on the group $\Pi_{1}(S)$ which are immediate and are used below.

\textbf{Fact 1} For any two letters $a$ and $b$ in $\mathcal{G}$, there exists at most one word in $\Lambda$ whose two first letters are given by $ab$. The other words in $\Lambda$ which contain the word $ab$ are a cyclic permutation of this one.

\textbf{Fact 2} For any letter $a$ in $\mathcal{G}$, there exists exactly two words in $\Lambda$ whose last letter (respectively first letter) is $a$. If $b$ et $c$ denote the penultimate letters (respectively the second letters) of these words, then the word $b^{-1}c$ is not a subword of a word in $\Lambda$.

\textbf{Fact 3} For any two letters $a$ and $b$ in $\mathcal{G}$ such that the word $ab$ is contained in a word of $\Lambda$, let us denote by $m_{1}$ the word of $\Lambda$ with first letter $b$ but whose last letter $l_{1}$ is different from $a$ and by $m_{2}$ the word in $\Lambda$ whose last letter is $a$ but whose first letter $l_{2}$ is not $b$. Then $l_{2}^{-1}l_{1}^{-1}$ is not contained in a word in $\Lambda$.

We will use Fact 2 in the following situation: if, at a given step of Dehn algorithm, we have a reduced word of the form $macm'$, where $acm'$ is a subword of a word in $\Lambda$, $ma$ is a simplifiable word and $mac$ is not simplifiable, then, after replacing $ma$ by the inverse of its complementary word, we obtain a word of the form $m"cm'$, where $m"c$ is not contained in any word in $\Lambda$. As for Fact 3, we will use it in the following situation: suppose that, at a given step of Dehn algorithm, we have a word of the form $mabm'$, where $ab$ is a subword of a word in $\Lambda$ and $ma$ as well as $bm'$ are simplifiable. Suppose moreover that the words $mab$ and $abm'$ are not simplifiable (these are not subwords of words in $\Lambda$). Then after replacement of the words $ma$ and $bm'$ by the inverse of their complementary words, we obtain a word of the form $nl_{2}^{-1}l_{1}^{-1}n'$ where the words $nl_{2}^{-1}l_{1}^{-1}$ and $l_{2}^{-1}l_{1}^{-1}n'$ are not contained in any subword of words in $\Lambda$.

Let us come back to the proof of the lemma. As $D$ is an exceptional face, there exist two geodesic words $\gamma_{1}$ and $\gamma_{2}$ with distinct last letters such that $\gamma_{1}(D_{0})=D$ and $\gamma_{2}(D_{0})=D$. We now prove that one of them satisfies necessarily the first property given by the lemma and both of them satisfy one of the properties stated in the lemma. Moreover, if both of them satisfy the first property of the lemma, there exists a word $l_{1} \ldots l_{4g}$ in $\Lambda$ such that the $2g$ last letters of $\gamma_{1}$ are $l_{1} \ldots l_{2g}$ and the $2g$ last letters of $\gamma_{2}$ are $l_{4g}^{-1}  \ldots l_{2g+1}^{-1}$. These two results imply all the claims of the lemma.

Then take two geodesic words $\gamma_{1}$ and $\gamma_{2}$ with distinct last letters such that $\gamma_{1}(D_{0})=D$ and $\gamma_{2}(D_{0})=D$. The word $\gamma_{1} \gamma_{2}^{-1}$ is reduced but represents the trivial element in the group $\Pi_{1}(S)$. We apply now the algorithm just described to this word to prove the lemma. As the words $\gamma_{1}$ and $\gamma_{2}$ are geodesic, they do not contain simplifiable words. Let $\lambda'$ be a simplifiable word which is maximal for $\gamma_{1} \gamma_{2}^{-1}$. Let $\lambda_{3}$ be the complementary word of $\lambda'$. Then we have a decomposition of the word $\lambda'$, $\lambda'=\lambda_{1}\lambda_{2}$, with:
$$ \left\{
\begin{array}{l}
\gamma_{1}= \hat{ \gamma}_{1} \lambda_{1} \\
\gamma_{2}= \hat{\gamma}_{2} \lambda_{2}^{-1}
\end{array}
\right.
.
$$
By the previous remark, the words $\lambda_{1}$ and $\lambda_{2}$ are nonempty. The words $\hat{\gamma}_{1}$ and $\hat{\gamma}_{2}$ are geodesic. Moreover, as the words $\gamma_{1}$ and $\gamma_{2}$ are both geodesic, the words $\lambda_{1}$ and $\lambda_{2}$ are not simplifiable. Thus, if the length of $\lambda'$ is $4g$, the words $\lambda_{1}$ and $\lambda_{2}$ have both length $2g$. We now prove the following fact.

\paragraph{Fact} Such a word $\lambda'$ has necessarily length greater than $4g-2$.

Suppose first that the length of $\lambda'$ is less than or equal to $4g-3$ (\emph{i.e.} the length of $\lambda_{3}$ is greater than $2$). After the first step of the algorithm, we obtain the word $\hat{\gamma}_{1} \lambda_{3}^{-1} \hat{\gamma}_{2}^{-1}$ which is reduced as $\lambda'$ is maximal. Moreover, the concatenation of the word $\lambda_{3}^{-1}$ with the first letter of the word $\hat{\gamma}_{2}^{-1}$ is not contained in any word in $\Lambda$, and similarly for the concatenation of the last letter of the word $\hat{\gamma}_{1}$ with the word $\lambda_{3}^{-1}$. Suppose by induction that, at a given step of the algorithm, we obtain a reduced word of the following form:
$$ \tilde{\gamma}_{1}\eta_{1}\eta_{2} \ldots \eta_{k} \tilde{\gamma}_{2}^{-1},$$
where $k\geq 1$, the words $\tilde{\gamma}_{1}$ and $\tilde{\gamma}_{2}$ are geodesic and the words $\eta_{i}$ are each contained in a word of $\Lambda$, have length smaller than $2g$ and satisfy the following properties:
\begin{enumerate}
\item The words $\eta_{1}$ and $\eta_{k}$ have length greater than $1$ and, if they are both of length $2$, then $k>1$.
\item For any index $i$ between $1$ and $k-1$, the concatenation of the last letter of $\eta_{i}$ with the first letter of $\eta_{i+1}$ is not contained in any word in $\Lambda$.
\item The concatenation of the word $\eta_{k}$ with the first letter of the word $\tilde{\gamma}_{2}^{-1}$ is not contained in any word in $\Lambda$ and similarly for the concatenation of the last letter of the word $\tilde{\gamma}_{1}$ with the word $\eta_{1}$.
\end{enumerate}
Let us apply a new step of the algorithm. A simplifiable subword $\lambda'$ of the above word is necessarily contained in one of the words $\tilde{\gamma}_{1}\eta_{1}$ or $\eta_{k} \tilde{\gamma}_{2}^{-1}$ by the second property above and by using the fact that each of the $\eta_{i}$'s has length smaller than $2g$. We may suppose, without loss of generality, that such a subword is contained in $\tilde{\gamma}_{1}\eta_{1}$. By combining Fact 1 with the third property above, we obtain that the last letter $a$ of the word $\lambda'=\lambda'_{1}a$ is also the first letter of the word $\eta_{1}=a \eta'_{1}$. As the word $\tilde{\gamma}_{1}$ is geodesic, it does not contain any simplifiable subword, so the word $\lambda'_{1}$, that it contains, has length $2g$. After applying the algorithm, we obtain the word:
$$\tilde{\gamma}'_{1} \tilde{\lambda}^{-1} \eta'_{1} \eta_{2} \ldots \eta_{k} \tilde{\gamma}_{2}^{-1},$$
where $\tilde{\gamma}_{1}=\tilde{\gamma}'_{1}\lambda'_{1}$ and $\tilde{\lambda}$ is the complementary word of $\lambda'$. The obtained words $\tilde{\gamma}'_{1}$ and $\tilde{\gamma}_{2}$ are geodesic. The word $\tilde{\lambda}$, of length $2g-1$, has length smaller than $2g$ and greater than $1$. Moreover, if $k=1$, the length of $\eta_{1}$ is greater than $2$ so the length of $\eta'_{1}$ is greater than $1$. Fact 2 implies that the concatenation of the last letter of $\tilde{\lambda}^{-1}$ with the first letter of $\eta'_{1}$ is not contained in any word in $\Lambda$. Finally, the third property is satisfied for this decomposition: denoting by $l$ the last letter of $\tilde{\gamma}'_{1}$, if the word $l \tilde{\lambda}^{-1}$ were a subword of a word in $\Lambda$, then, by Fact 1, the first letter of the word $\lambda'$ would be $l^{-1}$, which would contradict the fact that the word $\tilde{\gamma}_{1}$ is reduced. At each step of the algorithm, the sum of the lengths of the geodesic words at the beginning and at the end of this decomposition strictly decreases. Therefore, after applying a finite number of steps of the algorithm, we obtain a word of the following form:
$$ \tilde{\gamma}_{1}\eta_{1}\eta_{2} \ldots \eta_{k} \tilde{\gamma}_{2}^{-1},$$
where $k \geq 1$, which satisfies the three properties that we just described as well as the following property: the length of $\tilde{\gamma}_{1}$ as well as the length of $\tilde{\gamma}_{2}$ are less than $2g$. In this case, we can see that the considered word does not contain subwords for a word in $\Lambda$ with length greater than $2g$, a contradiction.

Let us come back to the first step of the algorithm. Then the considered word $\lambda'$ has length $4g-2$ or $4g-1$, if its length is not $4g$. Suppose now that the length of $\lambda'$ is $4g-2$. We want to find a contradiction.

After the first step of the algorithm, we obtain a reduced word of the form $\hat{\gamma}_{1} \lambda_{3} \hat{\gamma}_{2}^{-1}$, where the length of $\lambda_{3}=ab$ is $2$. As before, the concatenation of the last letter of $\hat{\gamma}_{1}$ with the word $\lambda_{3}$ as well as the concatenation of the word $\lambda_{3}$ with the first letter of $\hat{\gamma}_{2}^{-1}$ is not contained in any word of $\Lambda$. Without loss of generality, we may suppose that, during the second step of the algorithm, we choose a subword of a word in $\Lambda$ of the form $b\tilde{\lambda}_{2}$, where the word $\tilde{\lambda}_{2}$ is the concatenation of the $2g$ first letters of the word $\hat{\gamma}_{2}^{-1}$. Let us use notation from Fact 3. After applying a step of the algorithm, we obtain a word of the form $\hat{\gamma}_{1} a \eta_{1} \tilde{\gamma}_{2}^{-1}$, where the length of $\eta_{1}$ is $2g-1$ and the first letter of $\eta_{1}$ is $l_{1}^{-1}$. While the subwords chosen during the algorithm do not meet $\hat{\gamma}_{1}$, we obtain words of the form $\hat{\gamma}_{1} a \eta_{1}\eta_{2} \ldots \eta_{k} \tilde{\gamma}_{2}^{-1}$, where the properties 1) and 2) are satisfied as well as property 3) for $\tilde{\gamma}_{2}$ alone and where the first letter of $\eta_{1}$ is $l_{1}^{-1}$. After the first step for which we replace a subword which meets $\hat{\gamma}_{1}$, we obtain a word of the form:
$$\tilde{\gamma}_{1} \eta_{0} \eta_{1} \ldots \eta_{k} \tilde{\gamma}_{2}^{-1},$$
where the last letter of the word $\eta_{0}$ is $l_{2}^{-1}$ and the first letter of $\eta_{1}$ is $l_{1}^{-1}$. Fact 3 implies the situation is the same as before, a contradiction.

Finally, in the case where the length of $\lambda'$ is $4g-1$, one of the two geodesic words $\gamma_{1}$ or $\gamma_{2}$ satisfies necessarily the first property of the lemma. Similarly, after implementing the algorithm, we see that the second geodesic word satisfies the second property of the lemma.
\end{proof}

\begin{proof}[Proof of Lemma \ref{faceexc}]
The cases $j=1$ and $j=2$ are symmetric to each other: suppose that $j=1$. Take an index $2 \leq i' \leq 2g-1$ (think that i'=2g-i). By induction on the length of $m$, we prove that, for any reduced word $m$ of length less than or equal to $2g-i'$ with a first letter distinct from $l_{i'+1}$ and from $l_{i'}^{-1}$:
\begin{enumerate}
\item The word $\gamma'l_{1} l_{2} \ldots l_{i'}m$ is geodesic.
\item The fundamental domain $\gamma'l_{1} l_{2} \ldots l_{i'}m(D_{0})$ is not exceptional.
\end{enumerate}
Suppose that the property holds for a word $m$ as above of length less than $2g-i'$.  Let $l$ be a letter in $\mathcal{G}$ different from the inverse of the last letter of $m$ (or different from $l_{i'+1}$ and from $l_{i'}^{-1}$ if the word $m$ is empty). As the fundamental domain $\gamma'l_{1} l_{2} \ldots l_{i'}m(D_{0})$ is not an exceptional face, then:
$$d_{\mathcal{D}}(\gamma'l_{1} l_{2} \ldots l_{i'}ml(D_{0}), D_{0})=d_{\mathcal{D}}(\gamma'l_{1} l_{2} \ldots l_{i'}m(D_{0}), D_{0})+1$$
and the word $\gamma'l_{1} l_{2} \ldots l_{i'}ml$ is geodesic. Moreover, as the length of $ml$ is less than or equal to $2g-i'$, the word $\gamma'l_{1} l_{2} \ldots l_{i'}ml$ is not of one of the forms given in Lemma \ref{geodexc}. Therefore, the face $\gamma'l_{1} l_{2} \ldots l_{i'}ml(D_{0})$ is not exceptional. This completes the proof of Lemma \ref{faceexc}.
\end{proof}

\begin{proof}[Proof of Lemma \ref{geodexc2}] The generating set of the group $\Pi_{1}(S)$ given by the deck transformations which send the fundamental domain $D_{1}$ on a fundamental domain in $\mathcal{D}$ adjacent to $D_{1}$ is $\gamma_{1} \mathcal{G} \gamma_{1}^{-1}$. By Lemma \ref{geodexc}, there exists a geodesic word on elements of $\gamma_{1} \mathcal{G} \gamma_{1}^{-1}$ whose $2g$ last letters determine a word
$$(\gamma_{1} \lambda_{2g}^{-1} \gamma_{1}^{-1})(\gamma_{1} \lambda_{2g-1}^{-1} \gamma_{1}^{-1}) \ldots (\gamma_{1} \lambda_{1}^{-1} \gamma_{1}^{-1}),$$
where $\lambda_{1} \lambda_{2} \ldots \lambda_{4g} \in \Lambda$, which sends the face $D_{1}$ to the face $D_{0}$. Thus, in the group $\Pi_{1}(S)$:
$$ \gamma_{1}^{-1}= \gamma_{1} \eta^{-1} \lambda_{2g}^{-1} \lambda_{2g-1}^{-1} \ldots \lambda_{1}^{-1} \gamma_{1}^{-1},$$
where $\eta^{-1} \lambda_{2g}^{-1} \lambda_{2g-1}^{-1} \ldots \lambda_{1}^{-1}$ is a geodesic word on elements of $\mathcal{G}$. Let $\gamma$ be the word $\lambda_{1} \lambda_{2} \ldots \lambda_{2g} \eta$. Then, in the group $\Pi_{1}(S)$: $\gamma=\gamma_{1}.$
Thus, the geodesic word $\gamma$ satisfies the required properties. The second point of the lemma comes from the above argument and from Lemma \ref{adjacence}.
\end{proof}

\begin{proof}[Proof of Lemma \ref{sommets}]
Let us denote by $s(D_{0})$ and $s'(D_{0})$, where $s$ and $s'$ are deck transformations in $\mathcal{G}$, the faces which are adjacent to the face $D_{0}$ and which contain the point $\tilde{p}$. Suppose that $d_{\mathcal{D}}(D_{0},D_{1})=l(h)$. If the relation $d_{\mathcal{D}}(s(D_{0}),D_{1})=d_{\mathcal{D}}(D_{0},D_{1})+1$ held, then we would have $d_{\mathcal{D}}(D_{0},s^{-1}(D_{1}))>l(h)$ and the vertex $s^{-1}(\tilde{p})$ of $\partial D_{0}$ would satisfy:
$$ \tilde{h}(s^{-1}(\tilde{p}))=s^{-1}(\tilde{h}(\tilde{p})) \in s^{-1}(D_{1})$$
which is not possible by definition of $l(h)$. Thus, necessarily:
$$d_{\mathcal{D}}(s(D_{0}),D_{1})=d_{\mathcal{D}}(s'(D_{0}),D_{1})=d_{\mathcal{D}}(D_{0},D_{1})-1.$$
The face $D_{0}$ is exceptional with respect to $D_{1}$. By Lemma \ref{geodexc}, there exists a word $\lambda_{1} \lambda_{2} \ldots \lambda_{4g}$ in $\Lambda$ such that:
$$\left\{
\begin{array}{l}
\gamma= \lambda_{1} \lambda_{2} \ldots \lambda_{2g} \gamma'= \lambda_{4g}^{-1} \ldots \lambda_{2g+1}^{-1} \gamma' \\
\gamma(D_{0})=D_{1}
\end{array}
\right.
.$$
Moreover, by the same lemma, the point $\tilde{p}$ is common to the faces $\lambda_{1} \lambda_{2} \ldots \lambda_{i}(D_{0})$ and $\lambda_{4g}^{-1} \lambda_{4g-1} \ldots \lambda_{4g-i+1}^{-1}(D_{0})$ for an integer $i$ between $0$ and $2g$. Let $i$ be an integer between $0$ and $2g$. The point $\tilde{p}$ is a vertex of the polygon $\lambda_{1} \lambda_{2} \ldots \lambda_{i}(D_{0})$ so the point $\lambda_{i}^{-1}\lambda_{i-1}^{-1} \ldots \lambda_{1}^{-1}(\tilde{p})$ belongs to the polygon $\partial D_{0}$. Therefore, we have $4g$ pairwise distinct points which are vertices of the polygon $\partial D_{0}$: we have obtained in this way all the vertices of the polygon $\partial D_{0}$. Moreover, if $i \geq 1$:
$$ \left\{
\begin{array}{l}
\tilde{h}(\lambda_{i}^{-1} \lambda_{i-1}^{-1} \ldots \lambda_{1}^{-1}(\tilde{p})) \in \lambda_{i+1} \lambda_{i+2} \ldots \lambda_{2g} \gamma'(D_{0}) \\
\tilde{h}(\lambda_{4g-i+1} \lambda_{4g-i+2} \ldots  \lambda_{4g}(\tilde{p})) \in \lambda_{4g-i}^{-1} \lambda_{4g-i-1}^{-1} \ldots \lambda_{2g+1}^{-1} \gamma'(D_{0})
\end{array}
\right.
$$
so the image under the homeomorphism $\tilde{h}$ of the vertices of the polygon $\partial D_{0}$ which are different from $\tilde{p}$ belong to the interior of fundamental domains $D$ in $\mathcal{D}$ with the following property: the face $D_{0}$ is not exceptional with respect to $D$, by Lemma \ref{geodexc}. This implies the converse and the uniqueness of the face $D_{1}$.
\end{proof}

\section{Distortion elements with a fast orbit growth}

In this section, we prove Theorem \ref{exemple}.

First notice that it suffices to prove Theorem \ref{exemple} for sequences $(v_{n})_{n \geq 1}$ with the following additional properties:
\begin{enumerate}
\item The sequence $(v_{n})_{n \geq 1}$ is strictly increasing.
\item The sequence $(v_{n+1}-v_{n})_{n \geq 1}$ is decreasing.
\end{enumerate}
Let us prove this. Suppose we have proved Theorem \ref{exemple} for strictly increasing sequences. If $(v_{n})_{n \geq 1}$ is any sequence, it suffices to apply the theorem to the sequence $( \sup_{k \leq n}v_{k}+1-\frac{1}{2^{n}})_{n \geq 1}$ to deduce the general theorem. Suppose now that the theorem is proved only for sequences which satisfy the two properties above. Let us prove that it is true for any strictly increasing sequence. Let $(v_{n})_{n \geq 1}$ be a strictly increasing sequence such that the sequence $(\frac{v_{n}}{n})_{n}$ converges to $0$. Let $A$ be the convex hull in $\mathbb{R}^{2}$ of the set 
$$ \left\{ (n, t), \   n \geq 1 \mbox{ et } t \leq v_{n} \right\}$$
and let $w_{n}= \sup \left\{ t \in \mathbb{R}, \ (n,t) \in A \right\}$. The sequence $(w_{n})_{n \geq 1}$ satisfies the two properties above and $\lim_{n \rightarrow + \infty} \frac{w_{n}}{n}=0$. Then it suffices to apply the theorem to this sequence to prove it for the sequence $(v_{n})_{n \geq 1}$.

In what follows, we suppose that $(v_{n})_{n \geq 1}$ is a sequence which satisfies the hypothesis of Theorem \ref{exemple} as well as the two above properties.

Let $\mathbb{A} = \mathbb{R}/ \mathbb{Z} \times [-1,1]$ and let $\alpha$ be the curve $\left\{ 0 \right\} \times [-1,1] \subset \mathbb{A}$.
The homeomorphism $f$ in $\mathrm{Homeo}_{0}(\mathbb{A}, \partial \mathbb{A})$ which we are going to build will satisfy the following property:
$$ \exists x \in \mathring{\mathbb{A}}, \ v_{n}+ \frac{1}{2^{n}} \geq p_{2}(\tilde{f}^{n}(x))-p_{2}(x) \geq v_{n}$$
where $p_{2}:\mathbb{R} \times (-1,1) \rightarrow \mathbb{R}$ denotes the projection. As $f$ is compactly supported, this guarantee that the property
$$ \forall n \geq 0, \ \delta(\tilde{f}^{n}([0,1] \times [0,1])) \geq v_{n}$$
holds.
Now, let us consider the following embedding of $\mathbb{R}$ in $\mathring{\mathbb{A}}$:
$$ \begin{array}{rrcl}
L: & \mathbb{R} & \rightarrow & \mathring{\mathbb{A}}= \mathbb{R} / \mathbb{Z} \times (-1,1)  \\
 & x & \mapsto & (x \ \mbox{mod 1},g(x))
\end{array}
$$
where $g$ is a continuous strictly increasing function whose limit is $\frac{1}{2}$ as $x$ tends to $+\infty$ and whose limit is $-\frac{1}{2}$ as $x$ tends to $-\infty$. We identify a tubular neighbourhood $T$ of $L(\mathbb{R})$ with the band $\mathbb{R} \times [-1,1]$, where the real line $\mathbb{R}$ is identified with the curve $L(\mathbb{R})$ via the map $L$ so that, for any integer $j$, the path $\left\{j \right\} \times [-1,1]$ is contained in $\alpha$. Let $h$ be a homeomorphism of the line $L$, identified with $\mathbb{R}$, with the following properties:
\begin{enumerate}
\item The map $x \mapsto h(x)-x$ is decreasing on the interval $[0,+\infty)$ and $\lim_{x \rightarrow + \infty}h(x)-x = 0$.
\item The homeomorphism $h$ is equal to the identity on $(-\infty,-1]$.
\item For any natural numbers $i$ and $n$: $h^{n}(i) \notin \mathbb{Z}$.
\item For any natural number $n$: $h^{n}(0)=v_{n}+ \frac{\epsilon_{n}}{2^{n}}$, where $\epsilon_{n}$ is equal to $1$ if $v_{n}$ is an integer and vanishes otherwise.
\end{enumerate}

The "$\epsilon_{n}$" in the fourth property makes this property compatible with the third one. Let $f$ be the homeomorphism defined on $T$ by:
$$ \begin{array}{rrcl}
f: & \mathbb{R} \times [-1,1] & \rightarrow & \mathbb{R} \times [-1,1] \\
 & (x,t) & \mapsto & ((1- \left| t \right|)h(x) + \left| t \right| x,t)
\end{array}.$$

This extends continuously to a homeomorphism in $\mathrm{Homeo}_{0}(\mathbb{A}, \partial \mathbb{A})$ that we denote by $f$ by abuse. This extension is possible thanks to the fifth property satisfied by $h$ which makes sure that the homeomorphism $f$ is close to the identity when we are close to the circle $\mathbb{R} / \mathbb{Z} \times \left\{ \frac{1}{2} \right\}$. The third property satisfied by $h$ makes sure that, for any nonnegative integers $i$, $j$ and $n$, the curve $f^{n}(\left\{i \right\} \times (-1,1))$ is transverse to the curve $\left\{ j \right\} \times (-1,1)$. 
For any curve $\beta$ in the annulus $\mathbb{A}$, let $l(\beta,\alpha)$ be the number of connected components of $\Pi^{-1}(\alpha)$ met by a lift of $\beta$.
In order to prove that the homeomorphism $f$ is a distortion element, the crucial proposition is the following:

\begin{proposition} \label{fragexemple}
Let $l$ be a positive integer and let $\lambda_{l}=l(f^{l}(\alpha), \alpha)$. There exist two homeomorphisms $g_{1}$ and $g_{2}$ in $\mathrm{Homeo}_{0}(\mathbb{A}, \partial \mathbb{A})$ supported respectively in the complement of $\alpha$ and in a tubular neighbourhood of $\alpha$ such that:
$$ l((g_{2} \circ g_{1})^{\lambda_{l}-1}(f^{l}(\alpha)), \alpha)=1.$$
\end{proposition}

First, let us see why this property implies Theorem \ref{exemple}.

\begin{proof}[Proof of Theorem \ref{exemple}]
Let $\mathcal{U}$ be the open cover of $\mathbb{A}$ built at the beginning of Section \ref{bord}. By Lemma \ref{inbord}:
$$ \mathrm{Frag}_{\mathcal{U}}(g_{1}) \leq 6$$
and:
$$ \mathrm{Frag}_{\mathcal{U}}(g_{2}) \leq 6.$$

\noindent \textbf{Remark} Looking closely at the proof of Lemma \ref{inbord}, we can see that the upper bound can be replaced with $3$.

By Lemma \ref{inbord}:
$$\mathrm{Frag}_{\mathcal{U}}((g_{2} \circ g_{1})^{\lambda_{l}-1} \circ f^{l}) \leq 6.$$
Recall that $a_{l}=a_{\mathcal{U}}(f^{l})$ is the minimum of the $mlog(k)$ where there exists a family $(h_{i})_{1 \leq i \leq m}$ of homeomorphisms which are each supported in one of the open sets of $\mathcal{U}$ such that $f^{l}=h_{1} \circ h_{2} \circ \ldots \circ h_{m}$ and the cardinality of the set $\left\{h_{p}, \ 1 \leq p \leq m \right\}$ is $k$.
So, for any positive integer $l$:
$$a_{l} \leq (12 \lambda_{l} -6) log(18).$$
But:
$$ \frac{\lambda_{l}}{l}= \frac{l(f^{l}(\alpha), \alpha)}{l} \leq \frac{v_{l} + \frac{1}{2^{l}}}{l},$$
where the left-hand side of the inequality converges to $0$. Therefore, the sequence $(\frac{a_{l}}{l})_{l >0}$ converges to $0$. By Proposition \ref{fragdist}, the homeomorphism $f$ is a distortion element in $\mathrm{Homeo}_{0}(\mathbb{A}, \partial \mathbb{A})$. Notice that, here, the use of Proposition \ref{fragdist} is crucial as the hypothesis 
$$\lim_{n \rightarrow + \infty} \frac{\mathrm{Frag}_{\mathcal{U}}(f^{n}).log(\mathrm{Frag}_{\mathcal{U}}(f^{n}))}{n}= \lim_{n \rightarrow + \infty} \frac{\lambda_{n}}{n} =0$$
of Theorem \ref{fragdistth} does not necessarily hold.
\end{proof}

\begin{proof}[Proof of Proposition \ref{fragexemple}]
Let $g=f^{l}$ and $\lambda=\lambda_{l}=l(f^{l}(\alpha),\alpha)$. In what follows, everything will take place in the tubular neighbourhood $T$ of the line $L$ which is identified to $\mathbb{R} \times [-1,1]$. Therefore, we can "forget" the annulus $\mathbb{A}$. Let us give briefly the idea of the proof. As the curve $g(\left\{0 \right\} \times (-1,1))$ has length $\lambda$ with respect to  $\alpha$, we have no choice: in the product $(g_{2} \circ g_{1})^{\lambda-1}$, each factor must push this curve to the left and it must go across a curve of the form $\left\{ i \right\} \times (-1,1)$ at each step (under the action of each factor $g_{2} \circ g_{1}$). The curves $g(\left\{i \right\} \times (-1,1))$ are less dilated and must come back to their initial places in $\lambda$ steps. Then we must "make them wait" so that they do not come back too fast: if they come back before the time $\lambda$, they go too far to the left, which we want to avoid. On Figure \ref{distorsionhomeo15}, we represented the action of $g_{2} \circ g_{1}$ on $g(\alpha)$ on an example.

\begin{figure}[ht]
\begin{center}
\includegraphics[scale=0.75]{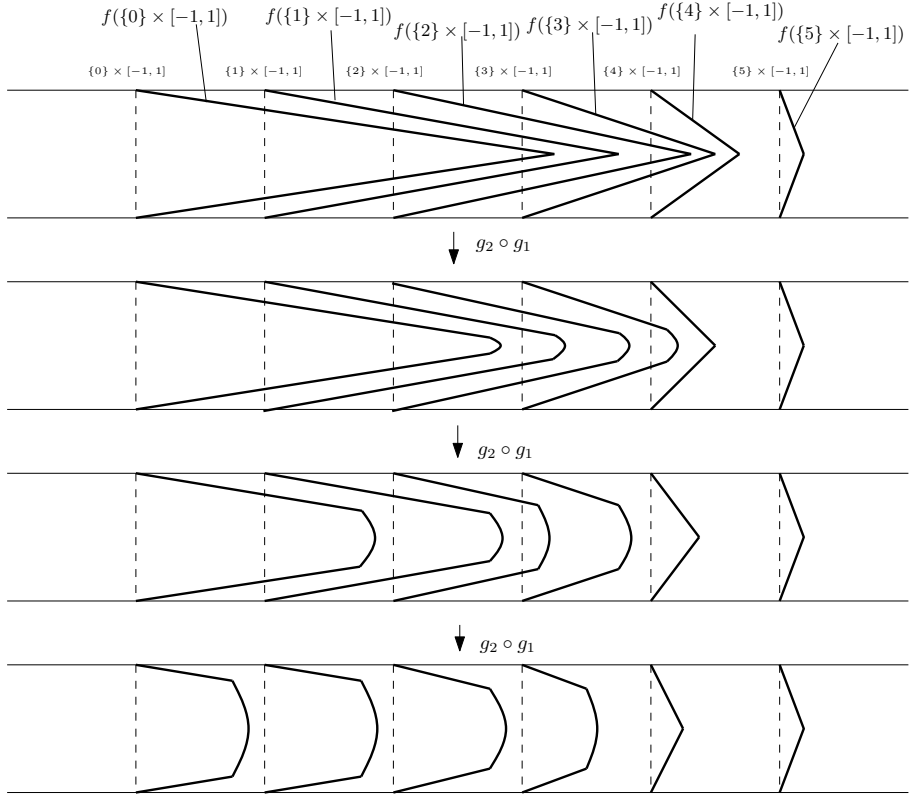}
\end{center}
\caption{The action of $g_{2} \circ g_{1}$}
\label{distorsionhomeo15}
\end{figure}

Let $N$ be the minimal nonnegative integer such that 
$$g(N,0) \in [N, N+1) \times \left\{ 0 \right\} \subset \mathbb{R} \times[-1,1] \subset \mathbb{A}.$$
In the case of Figure \ref{distorsionhomeo15}, this integer is equal to $4$. Let us take a real number $\epsilon$ in $(0, \frac{1}{2})$ such that, for any integer $i$ in $[0,N]$, any connected component of $g(\alpha) \cap ([i-\epsilon, i + \epsilon] \times [-1,1] -g(\left\{i \right\} \times (-1,1))$ joins both boundary components of $[i-\epsilon, i+\epsilon] \times (-1,1)$. The transversality property satisfied by $f$ enables us to find such a real number $\epsilon$. Let $\eta >0$ such that, for any integer $i$ in $[0,N]$, any connected component of
$$g(\alpha) \cap [i+ \frac{\epsilon}{4}, i+1- \frac{\epsilon}{4}] \times [-1,1]$$
is contained in:
$$[i+ \frac{\epsilon}{4} , i+1- \frac{\epsilon}{4}] \times (-1+ \eta,1- \eta).$$
Let us start with the construction of the homeomorphism $g_{2}$. Let $g_{2}$ be a homeomorphism with the following properties:
\begin{enumerate}
\item The homeomorphism $g_{2}$ is supported in $\bigcup \limits _{0 \leq i \leq N} (i-\epsilon,i+\epsilon) \times (-1,1)$.
\item If $P_{i}$ denotes the connected component of $[i-\epsilon, i+\epsilon] \times [-1,1]- g(\left\{ i \right\} \times [-1,1])$ which contains $\left\{i-\epsilon \right\} \times [-1,1]$ and $K_{i}$ denotes a topological closed disc contained in $P_{i}$ which contains the connected components of 
$$(g(\alpha) \cap [i-\epsilon,i+ \frac{\epsilon}{2}] \times (-1,1))-g(\left\{i \right\} \times [-1,1]),$$
we have:
$$ \forall i, \ g_{2}(K_{i}) \subset [i-\epsilon, i- \frac{\epsilon}{2}] \times (-1+\eta, 1-\eta).$$
\item The homeomorphism $g_{2}$ globally preserves each connected component of $g(\alpha) \cap [i-\epsilon, i + \epsilon] \times (-1,1)$. 
\end{enumerate}

Before defining $g_{1}$, we first need to build a sequence of integers $(n_{i})_{0 \leq i \leq N}$. For an integer $i$ between $0$ and $N$, let:
$$A_{i}= \left\{ j \in [0,N], \ \left\{
\begin{array}{l}
g(\left\{ j \right\} \times [-1,1]) \cap \left\{ i \right\} \times [-1,1] \neq \emptyset \\
g(\left\{ j \right\} \times [-1,1]) \cap \left\{ i+1 \right\} \times [-1,1] = \emptyset
\end{array}
\right.
\right\}.$$
Let $i_{0}= \max \left\{i, \left\{i \right\} \times [-1,1) \cap g(\left\{0\right\} \times (-1,1)) \neq \emptyset \right\}$. The sets $A_{0}, A_{1}, \ldots ,A_{i_{0}-1}$ are all empty but we are going to see that, for any integer $N \geq m \geq i_{0}$, the set $A_{m}$ is nonempty. In the case of Figure \ref{distorsionhomeo15}, the sets $A_{0}$, $A_{1}$ and $A_{2}$ are empty , $A_{3}= \left\{0,1 \right\}$ and $A_{4}=\left\{2,3,4 \right\}$. More generally, the family $(A_{i_{0}},A_{i_{0}+1}, \ldots, A_{N})$ is a partition of $\left\{0,1, \ldots, N \right\}$ which admits an ordering in the sense that, if $i_{0} \leq m \leq m' \leq N$, then any integer in $A_{m}$ is smaller than any integer in $A_{m'}$. 
Let us prove that if, for an integer $i$ between $0$ and $N-1$, the set $A_{i}$ is nonempty, then the set $A_{i+1}$ is nonempty. Notice that, for an integer $j$ in the interval $[0,N]$:
$$l(g(\left\{j \right\} \times (-1,1)), \alpha)= \lfloor h^{l}(j) \rfloor -j+1 $$
by construction of $f$. As the map $x \mapsto h^{l}(x)-x$ is decreasing by construction of $h$, then the map
$$ j \mapsto l(g(\left\{j \right\} \times (-1,1)), \alpha)$$
is decreasing on $[0,N]$. In particular, $i_{0}=\lambda-1$. Let $j=\max(A_{i})$. As 
$$  l(g(\left\{j+1 \right\} \times (-1,1)), \alpha) \leq l(g(\left\{j \right\} \times (-1,1)), \alpha),$$
then the curve $g(\left\{j+1 \right\} \times (-1,1))$ does not meet the curve $\left\{i+2 \right\} \times [-1,1]$ so the integer $j+1$ belongs to $A_{i+1}$ which is nonempty. For any integer $i$ between $i_{0}$ and $N$, let
$$A_{i}= \left\{ j(i), j(i)+1, \ldots, j(i+1)-1 \right\}.$$
We define by induction a finite sequence of integers $(n_{i})_{0 \leq i \leq N}$:
\begin{enumerate}
\item If $i<i_{0}$, we set $n_{i}=1$.
\item Otherwise, assuming that the $n_{k}$'s, for $k<i$, have been defined, we set
$$n_{i}= \lambda - \sum_{k=j(i+1)-1}^{i-1}n_{k}.$$
\end{enumerate}
The integer $n_{i}$ will represent the number of iterations of $g_{2} \circ g_{1}$ necessary for a curve close to $\left\{i+1 \right\} \times (-1,1)$ to cross the curve $\left\{i \right\} \times (-1,1)$. For $0 \leq j \leq N$, let $i(j)$ be the unique integer such that $j \in A_{i(j)}$. After a number of iterations of $g_{2} \circ g_{1}$ which is less than or equal to $n_{i(j)}$, the curve $g(\left\{j \right\} \times (-1,1))$ will cross $\left\{i(j) \right\} \times (-1,1)$. Then, after $n_{i(j)-1}$ iterations, it will cross the curve $\left\{i(j)-1 \right\} \times (-1,1)$ and so on... For instance, in the case of Figure \ref{distorsionhomeo15}: $n_{0}=n_{1}=n_{2}=1$, $n_{3}=2$ and $n_{4}=4$.
Let us prove by induction that, for any integer $i \geq i_{0}$ :
$$\sum_{k=j(i)}^{i-1}n_{k} < \lambda.$$
This will prove also that the integers $n_{i}$ are positive. If $i=i_{0}$, then, for $j <i_{0}$, the set $A_{j}$ is empty and we have:
$$ l(g(\left\{0 \right\} \times [-1,1]), \alpha)=i_{0}+1 \leq \lambda$$
by definition of $\lambda$. Thus:
$$ \lambda - \sum_{k=0}^{i_{0}-1}n_{k}= \lambda - i_{0} >0$$
and the property holds for $i=i_{0}$. Suppose that the property holds for $k$ between $i_{0}$ and $i$ given between $0$ and $N-1$. Then:
$$ \sum_{k=j(i+1)}^{i}n_{k}= \lambda - \sum_{k=j(i+1)-1}^{i-1}n_{k} + \sum_{k=j(i+1)}^{i-1}n_{k} = \lambda-n_{j(i+1)-1} < \lambda$$
because $n_{j(i+1)-1}>0$ by induction hypothesis. The property is proved. 

For an integer $j$ between $0$ and $N$, notice that, by construction, the connected components of 
$$ g(\left\{j \right\} \times [-1,1]) \cap \bigcup_{0 \leq i \leq N}[i+\frac{\epsilon}{4}, i+1-\frac{\epsilon}{4}] \times (-1,1)$$
join each two distinct connected components of the boundary of 
$$\bigcup_{0 \leq i \leq N}[i+\frac{\epsilon}{4}, i+1-\frac{\epsilon}{4}] \times (-1,1)$$
except one (which corresponds to the maximal integer $i$) which we will denote by $C_{j}$. Let $i(j)$ be the unique integer such that the integer $j$ belongs to $A_{i(j)}$. Then:
$$C_{j} \subset [i(j)+\frac{\epsilon}{4}, i(j)+1-\frac{\epsilon}{4}] \times (-1,1).$$
Now, we can build an appropriate homeomorphism $g_{1}$. Let $g_{1}$ be a homeomorphism which is supported in 
$$ \bigcup_{0 \leq i \leq N}(i+\frac{\epsilon}{4}, i+1-\frac{\epsilon}{4}) \times [-1,1] \subset \mathbb{R} \times [-1,1] \subset \mathbb{A}$$
and which satisfies the following properties for any integer $i$ between $0$ and $N$:
\begin{enumerate}
\item The homeomorphism $g_{1}$ globally preserves each of the connected components of $g(\alpha) \cap [i+\frac{\epsilon}{4}, i+1-\frac{\epsilon}{4}] \times [-1,1]$ which join the boundary components of $[i+\frac{\epsilon}{4}, i+1-\frac{\epsilon}{4}] \times (-1,1)$.
\item For any integer $j$ in $A_{i}$ and any integer $r< \lambda-\sum \limits _{k=j}^{i-1} n_{k}$:
$$g_{1}^{r}(C_{j}) \cap(i-\epsilon, i+\epsilon) \times [-1,1]=C_{j} \cap (i-\epsilon,i+\epsilon) \times[-1,1].$$
\item For any integer $j$ in $A_{i}$, the following inclusion holds:
$$g_{1}^{\lambda- \sum \limits _{k=j}^{i-1}n_{k}}(C_{j}) \subset K_{i}$$
(notice that these properties are compatible as $\lambda- \sum \limits _{k=j}^{i-1}n_{k}$ increases with $j$ and, moreover, $\lambda- \sum \limits _{k=j}^{i-1}n_{k} \leq n_{i}$ by definition of $n_{i}$).
\item The following inclusion holds:
$$g_{1}^{n_{i}}([i+\frac{\epsilon}{4},i+1-\frac{\epsilon}{2}] \times(-1+\eta,1-\eta)) \subset [i+ \frac{\epsilon}{4}, i+\frac{\epsilon}{2}) \times (-1+\eta,1-\eta) \cap K_{i}.$$
\item For any connected component $C$ of $g(\alpha) \cap [i+\frac{\epsilon}{4},i+1-\epsilon]\times (-1,1)$ which joins the two boundary components of $[i+\frac{\epsilon}{4},i+1-\epsilon] \times (-1,1)$, we have:
$$ \forall r<n_{i}, \ g_{1}^{r}(C) \cap(i-\epsilon,i+\epsilon)\times[-1,1]= C \cap( i-\epsilon,i+\epsilon)\times[-1,1].$$
\item For any integer $r<n_{i}$, the set $g_{1}^{r}([i+1-\epsilon,i+1- \frac{\epsilon}{4}] \times [-1,1])$ does not meet the square$[i,i+\epsilon] \times [-1,1]$.
\end{enumerate}
The second and the third above properties give the speeds with which we push back the components $C_{j}$: the third property means that the piece $C_{j}$ is pushed back in a $K_{i}$ after time $\lambda- \sum \limits _{k=j+1}^{i-1}n_{k}$ and the second condition implies that it cannot be pushed back before this time. The properties 4, 5 et 6 give the exact time necessary to cross $[ i,i+1] \times (-1,1)$.

Now, we prove that, for homeomorphisms $g_{1}$ and $g_{2}$ with the properties given above, we have:
$$l((g_{2} \circ g_{1})^{\lambda-1}(g(\alpha)),\alpha)=1.$$
Let $j$ be an integer between $0$ and $N$ and let $i=i(j)$. We denote by $\alpha_{j}$ the curve $ \left\{j \right\} \times [-1,1]$. Let us prove that, for any $j' \in [j-1,i-1]$ and any $\lambda- \sum_{ k=j}^{ j'}n_{k}>r \geq \lambda - \sum_{ k=j}^{ j'+1}n_{k}$, we have:
$$l((g_{2} \circ g_{1})^{r}\circ g(\alpha_{j}), \alpha)=l(g(\alpha_{j}), \alpha)-(i-j'-1).$$

By the two first properties satisfied by $g_{1}$ and the third property satisfied by $g_{2}$, we have, for any positive integer $r$ which is less than $\lambda - \sum \limits _{k=j}^{i-1}n_{k}$:
$$ \left\{
\begin{array}{l}
(g_{2} \circ g_{1})^{r}(g(\alpha_{j}) \cap [0,i+\epsilon] \times [-1,1])=g(\alpha_{j}) \cap [0,i+\epsilon] \times [-1,1] \\
 (g_{2} \circ g_{1})^{r}(g(\alpha_{j}))=g_{1}^{r}(g(\alpha_{j}))
\end{array}
\right.
.
$$
This implies the above property for $j'=i-1$.
Therefore:
$$ g_{1} \circ (g_{2} \circ g_{1})^{\lambda - \sum \limits _{k=j}^{i-1}n_{k}-1} \circ g(\alpha_{j})=g_{1}^{\lambda - \sum \limits _{k=j}^{i-1}n_{k}}(g(\alpha_{j})).$$
The third property satisfied by the homeomorphism $g_{1}$ implies that the intersection of the above set with $[i-\epsilon, +\infty) \times [-1,1]$ is contained in $K_{i}$. Therefore, the second property satisfied by the homeomorphism $g_{2}$ implies that:
$$ (g_{2} \circ g_{1})^{\lambda - \sum \limits _{k=j}^{i-1}n_{k}} \circ g(\alpha_{j}) \subset [j,i-\frac{\epsilon}{2}] \times [-1+\eta,1-\eta].$$
All the extremal part of the curve has been put back in $[i-\epsilon,i- \frac{\epsilon}{2}] \times(-1,1)$. The remainder has not moved.
Indeed:
$$(g_{2} \circ g_{1})^{\lambda - \sum \limits _{k=j}^{i-1}n_{k}}(g(\alpha_{j}) \cap [j,i-\epsilon] \times [-1,1])=g(\alpha_{j}) \cap [j,i-\epsilon] \times [-1,1]$$
and:
$$ \begin{array}{rcl}
l((g_{2} \circ g_{1})^{\lambda - \sum \limits _{k=j}^{i-1}n_{k}} \circ g(\alpha_{j}), \alpha) & = & i-j \\
 & = & l(g(\alpha_{j}), \alpha)-1.
\end{array}
$$
It suffices now to repeat this argument.
Suppose that, for an integer $j'$ between $j+1$ and $i-1$:
$$ \left\{
\begin{array}{l}
(g_{2} \circ g_{1})^{\lambda - \sum \limits _{k=j}^{j'}n_{k}} \circ g(\alpha_{j}) \subset [j,j'+1-\frac{\epsilon}{2}] \times (-1+ \eta,1-\eta) \\
(g_{2} \circ g_{1})^{\lambda - \sum \limits _{k=j}^{j'}n_{k}}(g(\alpha_{j}) \cap [j,j'+1-\epsilon] \times [-1,1])=g(\alpha_{j}) \cap [j,j'+1-\epsilon] \times [-1,1]
\end{array}
\right.
.
$$
We saw that this property holds for $j'=i-1$. Supposing that this property holds for an integer $j'$, we prove now that it holds for the integer $j'-1$ and also that, under this hypothesis, for $\lambda- \sum_{ k=j}^{ j'-1}n_{k}> r > \lambda - \sum_{ k=j}^{ j'}n_{k}$:
$$l((g_{2} \circ g_{1})^{r}\circ g(\alpha_{j}), \alpha)=l(g(\alpha_{j}), \alpha)-(i-j');$$
By the fifth and the sixth properties satisfied by the homeomorphism $g_{1}$ and the third property satisfied by the homeomorphism $g_{2}$, for any integer $0 \leq r<n_{j'}$:
$$(g_{2} \circ g_{1})^{r} \circ (g_{2} \circ g_{1})^{\lambda - \sum \limits _{k=j}^{j'}n_{k}}(g(\alpha_{j}) \cap [0,j'+\epsilon] \times [-1,1])=g(\alpha_{j}) \cap [0,j'+\epsilon] \times [-1,1]$$
and
$$(g_{2} \circ g_{1})^{r} \circ (g_{2} \circ g_{1})^{\lambda - \sum \limits _{k=j}^{j'}n_{k}} \circ g(\alpha_{j}) = g_{1}^{r}(g_{2} \circ g_{1})^{\lambda - \sum \limits _{k=j}^{j'}n_{k}}(g(\alpha_{j})).$$
Therefore:
$$g_{1} \circ (g_{2} \circ g_{1})^{n_{j'}-1} \circ (g_{2} \circ g_{1})^{\lambda - \sum \limits _{k=j}^{j'}n_{k}}(g(\alpha_{j}))=g_{1}^{n_{j'}}(g_{2} \circ g_{1})^{\lambda - \sum \limits _{k=j}^{j'}n_{k}}(g(\alpha_{j}))$$
so, by the fourth property satisfied by the homeomorphism $g_{1}$, the intersection of this set with $[j'+\epsilon,+\infty) \times [-1, 1]$ is contained in the set $K_{j'}$. By the second property satisfied by the homeomorphism $g_{2}$:
$$ (g_{2} \circ g_{1})^{n_{j}} \circ(g_{2} \circ g_{1})^{\lambda - \sum \limits _{k=j}^{j'}n_{k}} \circ g(\alpha_{j}) \subset [j,j'-\frac{\epsilon}{2}] \times (-1+ \eta,1-\eta)$$
and, moreover:
$$(g_{2} \circ g_{1})^{n_{j'}}\circ(g_{2} \circ g_{1})^{\lambda - \sum \limits _{k=j}^{i-1}n_{k}}(g(\alpha_{j}) \cap [j,j'-\epsilon] \times [-1,1])=g(\alpha_{j}) \cap [j,j'-\epsilon] \times [-1,1].$$
This completes the induction. One can prove, as before, that, for any $\lambda >r> \lambda - n_{j}$:
$$(g_{2} \circ g_{1})^{r} \circ g(\alpha_{j}) = g_{1}^{r-\lambda+n_{j}}(g_{2} \circ g_{1})^{\lambda - n_{j}}(g(\alpha_{j})).$$
which implies that:
$$l((g_{2} \circ g_{1})^{\lambda-1} \circ g(\alpha_{j}),\alpha)=1,$$
what we wanted to prove.
\end{proof}

\section{Generalization of the results}

In this section, we will briefly generalize the results in two directions. First, we could look at other growth speeds of words than the linear speed. Moreover, we can also consider finite families of elements instead of looking at one element and define a notion of distortion for this situation. The results are analogous to those we stated before. In what follows, let $(w_{n})_{n \geq 0}$ be a sequence of positive real numbers which tends to $+\infty$. Let us start with a definition:

\begin{definition}
Let $G$ be a group and $g$ be an element of $G$. The element $g$ is said to be $(w_{n})_{n \geq 0}$-distorted in $G$ if and only if there exists a finite set $\mathcal{G}$ in $G$ such that:
\begin{enumerate}
\item The element $g$ belongs to the group generated by $\mathcal{G}$. 
\item The inferior limit of the sequence $(\frac{l_{\mathcal{G}}(g^{n})}{w_{n}})$ is $0$.
\end{enumerate}
\end{definition}

This notion of distortion is interesting only if $\lim_{n \rightarrow + \infty}\frac{w_{n}}{n}=+ \infty$: otherwise, any element of $G$ is $(w_{n})_{n \geq 0}$-distorted. Moreover, this notion depends only on the equivalence class of $(w_{n})_{n \geq 0}$ for the following equivalence relation:
$$ (\omega_{n}) \equiv (\xi_{n}) \Leftrightarrow \exists C>0, \ \forall n \geq 0, \ \frac{1}{C} \xi_{n} \leq \omega_{n} \leq C \xi_{n}.$$

Then, one can prove the following theorems:

\begin{proposition}
Let $D$ be a fundamental domain of $\tilde{S}$ for the action of $\Pi_{1}(S)$.\\
If a homeomorphism $f$ in $\mathrm{Homeo}_{0}(S)$ (respectively in $\mathrm{Homeo}_{0}(S, \partial S)$) is $(w_{n})_{n \geq 0}$-distorted in $\mathrm{Homeo}_{0}(S)$ (respectively in $\mathrm{Homeo}_{0}(S, \partial S)$), then:
$$ \liminf_{n \rightarrow + \infty} \frac{\delta(\tilde{f}^{n}(D))}{w_{n}}=0.$$
\end{proposition}

\begin{theorem}
Let $f$ be a homeomorphism in $\mathrm{Homeo}_{0}(S)$ (respectively in $\mathrm{Homeo}_{0}(S, \partial S)$). If:
$$\liminf_{n \rightarrow + \infty} \frac{\delta(\tilde{f}^{n}(D))log(\delta(\tilde{f}^{n}(D)))}{w_{n}}=0,$$
then $f$ is $(w_{n})_{n \geq 0}$-distorted in $\mathrm{Homeo}_{0}(S)$ (respectively in $\mathrm{Homeo}_{0}(S, \partial S)$).
\end{theorem}

\begin{theorem}
Let $(v_{n})_{n \geq 0}$ be a sequence of positive real numbers such that: $\liminf_{n \rightarrow + \infty} \frac{v_{n}}{w_{n}}=0$.
Then there exists a homeomorphism $f$ in $\mathrm{Homeo}_{0}(\mathbb{R}/\mathbb{Z} \times [0,1] , \mathbb{R}/\mathbb{Z} \times \left\{ 0,1 \right\} )$ such that:
\begin{enumerate}
\item $ \forall n \geq 0, \ \delta(\tilde{f}^{n}([0,1] \times [0,1])) \geq v_{n}$.
\item The homeomorphism $f$ is $(w_{n})_{n \geq 0}$-distorted in $\mathrm{Homeo}_{0}(\mathbb{R}/\mathbb{Z} \times [0,1] , \mathbb{R}/\mathbb{Z} \times \left\{ 0,1 \right\} )$.
\end{enumerate}
\end{theorem}

For any positive integer $k$, we denote by $\mathbb{F}_{k}$ the free group on $k$ generators. Let $a_{1},a_{2}, \ldots, a_{k}$ be the standard generators of this group and $A$ be the set of these generators.

\begin{definition}
Let $G$ be a group generated by a finite set $\mathcal{G}$. A $k$-tuple $(f_{1},f_{2}, \ldots, f_{k})$ is said to be distorted if the map $\mathbb{F}_{k} \rightarrow G$, which sends the generator $a_{k}$ to $f_{k}$, is not a quasi-isometry for the distances $d_{A}$ and $d_{\mathcal{G}}$. More generally, for any group $G$, a $k$-tuple $(f_{1},f_{2}, \ldots, f_{k})$ is said to be distorted if there exists a subgroup of $G$ which is finitely generated, which contains the elements $f_{i}$ and in which this $k$-tuple is distorted.
\end{definition}

One can prove the following theorem for a compact surface $S$:

\begin{theorem}
Let $D$ be a fundamental domain of $\tilde{S}$ for the action of $\Pi_{1}(S)$.
Let $(f_{1},f_{2}, \ldots, f_{k})$ be a $k$-tuple of homeomorphisms of $S$. Suppose that there exists a sequence of words $(m_{n})_{n \geq 0}$ on the $f_{i}$'s whose sequence of lengths $(l(m_{n}))_{n}$ tend to $+\infty$ such that:
$$\lim_{n \rightarrow +\infty} \frac{\delta(m_{n}(D))log(\delta(m_{n}(D)))}{l(m_{n})}=0.$$
Then the $k$-tuple $(f_{1},f_{2}, \ldots, f_{k})$ is distorted.
\end{theorem}

\section*{Acknowledgement}

I would like to thank Frédéric Le Roux who asked me if a homeomorphism of the torus whose rotation set reduces to a single point is a distortion element. I thank him also for his careful reading of this article.

\end{document}